\documentclass[oneside, a4paper,11pt,reqno]{amsart}
\usepackage{color}

\usepackage{xcolor}
\usepackage{amssymb,amsmath,amsthm}
\usepackage[T1]{fontenc}

\newtheorem{hypo}{Hypothesis}

\newtheorem{prop}[hypo]{Proposition}

\newtheorem{thm}[hypo]{Theorem}

\newtheorem{lem}[hypo]{Lemma}

\newtheorem{rqe}[hypo]{Remark}

\newtheorem{coro}[hypo]{Corollary}

\def\B{\mathcal{B}}
\def\C{\mathcal{C}}
\def\D{\mathcal{D}}

\def\I{\mathcal{I}}

\def\O{\mathcal{O}}
\def\E{\mathcal{E}}

\def\PP{\mathbb{P}}
\def\RR{\mathbb{R}}
\def\ZZ{\mathbb{Z}}

\def\EE{\mathbb{E}}

\newcommand {\pare}[1] {\left( {#1} \right)}
\newcommand {\cro}[1] {\left[ {#1} \right]}

\newcommand {\va}[1] {\left| {#1} \right|}
\newcommand {\acc}[1] {\left\{ {#1} \right\}}

\title
       {Limit theorems for additive functionals of random walks in random scenery}

\author{Fran\c{c}oise P\`ene}
\address{Univ Brest, Universit\'e de Brest, LMBA, UMR CNRS 6205, 29238 Brest cedex, France}
\email{francoise.pene@univ-brest.fr}

\keywords{random walk in random scenery; central limit theorem; local limit theorem; local time; Brownian motion; ergodicity; infinite measure; dynamical system\\
This research was supported by the french ANR project MALIN, Projet-ANR-16-CE93-0003}

\begin{document}

\begin{abstract} 
We study the asymptotic behaviour of additive functionals of random walks in random scenery. 
We establish bounds for the moments of the local time of the Kesten and Spitzer process.
These bounds combined with a previous moment convergence result (and an ergodicity result) imply the convergence in distribution of additive observables  (with a normalization in $n^{\frac 14}$).
When the sum of the observable is null, the previous limit vanishes and we prove the convergence in the sense of moments (with a normalization in $n^{\frac 18}$).
\end{abstract}
\maketitle

\section{Introduction} 

\subsection{Description of the model and of some earlier results}
We consider two independent sequences $(X_k)_{k\ge 1}$ (the increments of the random walk) and $(\xi_y)_{y\in\mathbb Z}$
(the random scenery) of independent identically distributed $\mathbb Z$-valued random variables. 
We assume in this paper that $X_1$ is centered and admits finite moments of all orders,  
and that its support generates the group $\mathbb Z$.
We define the random walk $(S_n)_{n\ge 0}$ as follows
$$S_0:=0\quad \mbox{and}\quad S_n:=\sum_{i=1}^{n}X_i\quad  \textrm{for all } n\ge 1.$$ 
We assume that $\xi_0$ is centered, that its support generates the group $\mathbb Z$, and that it admits a finite second moment $\sigma_\xi^2:=\EE[\xi_0^2]>0$.
The random walk in random scenery (RWRS) is the process defined as follows
\begin{equation}\label{Zn}
Z_n:=\sum_{k=0}^{n-1}\xi_{S_k}=\sum_{y\in\mathbb Z}\xi_yN_n(y)\, ,
\end{equation}
where we set $N_n(y)=\#\{k=0,\dots,n-1\ :\ S_k=y\}$ for the local time of $S$ at position $y$ before time $n$. This process first studied by Borodin \cite{Borodin} and Kesten and Spitzer \cite{KestenSpitzer} describes the evolution of the total amount won until time $n$ by a particle moving with respect to the random walk $S$, starting with a null amount at time $0$ and wining the amount $\xi_\ell$ at each time the particle hits the position $\ell\in\mathbb Z$. This process is a natural example of (strongly) stationary process with long time dependence.
Due to the first works by Borodin \cite{Borodin} and by Kesten and Spitzer \cite{KestenSpitzer}, we know that
$(n^{-\frac 34}Z_{\lfloor nt\rfloor})_t$ converges in distribution, as $n$ goes to infinity, to the so-called Kesten and Spitzer process
$\left(\sigma_\xi\Delta_t,t\ge 0\right)$, where $\Delta$ is defined by
\begin{equation}\label{Delta}
\Delta_t:=\int_{-\infty}^{+\infty} L_t(x)\, {\rm d}\beta_x\, ,
\end{equation}
with $(\beta_x)_{x\in\mathbb R}$ a Brownian motion and 
$(L_t(x), \, t\ge 0,\ x\in\mathbb R)$ a jointly continuous in $t$ and $x$ version 
of the local time process of a
standard Brownian motion $(B_t)_{t\ge 0}$, where $((B_t)_t,(\beta_s)_s)$ is the limit in distribution  of
$n^{-\frac 12}((S_{\lfloor nt\rfloor})_t,(\sigma_\xi^{-1}\sum_{k=1}^{\lfloor ns\rfloor}\xi_k)_s)$ as $n\rightarrow +\infty$. 
Observe that $\Delta$ is the continuous time analog of the random walk in random scenery. To be convinced of this fact, one may compare the right hand side of \eqref{Zn} with \eqref{Delta}.
The process $\Delta$ is a classical and nice example of a (strongly) stationary process, self-similar with dependent (strongly) stationary increments and exhibiting long time dependence.

In \cite{Borodin}, Borodin established the convergence in distribution of $Z_n$ when $X$ and $\xi$ have second order moments. Kesten and Spitzer established in \cite{KestenSpitzer} a functional limit theorem when the distributions of $X$ and $\xi$  belong to the domain of attraction of stable distributions with respective parameters $\alpha\ne 1$ and $\beta\in(0,2]$.
Limit theorems have been extended by Bolthausen \cite{Bolthausen} (for the case $\alpha=\beta=2$ for random walks of dimension $d=2$), by Deligiannidis and Utev \cite{DU} ($\alpha=d\in\{1,2\}$, $\beta=2$, providing some correction to \cite{Bolthausen}) and by Castell, Guillotin-Plantard and the author \cite{FFN} (when $\alpha\le d$ and $\beta<2$), completing the picture for the convergence in the sense of distribution and for the functional limit theorem (except in the case $\alpha\le 1$ and $\beta=1$). Since the seminal works by Borodin and by Kesten and Spitzer, random walks in random scenery and the Kesten and Spitzer process $\Delta$ have been the object of various studies (let us mention for example \cite{K,YX,GPJ,BP,GPP,GPHS,GPPW,AGPP}). 

Random walks in random scenery are related to other models, such as the Matheron and de Marsily Model \cite{MdM} of transport in porous media, the transience of which has been established by Campanino and  Petritis \cite{CP} and which has many generalizations (e.g. \cite{GPLN,DP,GKP,Bremont1,Bremont2}), 
and such as the Lorentz-L\'evy process (see \cite{FPMAS} for a short presentation of some models linked to random walks in random scenery).

Random walks in random scenery constitute also a model of interest in the context of dynamical systems. They correspond indeed to Birkhoff sums of a transformation called the $T,T^{-1}$ transformation appearing in \cite[p. 682, Problem 2]{Weiss} where it was asked whether this Kolmogorov automorphism is Bernoulli or not. In \cite{Kalikow}, Kalikow answered negatively this question by proving that this transformation is not even loosely Bernoulli.

\subsection{Main results}\label{mainresults}
Before stating our main results, let us introduce some additional notations. Let $d\in\mathbb N$ be the greatest common divisor of the set
$\{x\in\mathbb Z,\ \mathbb P(\xi_0-\xi_1=x)>0\}$ and $\alpha\in\mathbb Z$ such that $\mathbb P(\xi_0=\alpha)>0$. This means that the random variables $\xi_\ell$ take almost surely their
values in $\alpha+d\mathbb Z$ and that $d$ is largest positive integer satisfying this property. Since the support of $\xi$ generates the group $\mathbb Z$, necessarily $\alpha$ and $d$ are coprime. 
Recall that the quantity $d$
can be also simply characterized using the common characteristic function $\varphi_\xi$ of the $\xi_\ell$.\footnote{Indeed $d\ge 1$ is such that $\{u\, :\, |\varphi_\xi(u)|=1\} = (2\pi/d)\ZZ$
and a.s.
$
e^{\frac{2i\pi\xi}d}=e^{\frac{2i\pi\alpha}d}$
which is a primitive $d$-th root of the unity.
}

In the present paper we are interested in the asymptotic behaviour of additive functionals of the RWRS $(Z_n)_{n\ge 1}$ that is of quantities of the following form:
\[
\mathcal Z_n:=\sum_{k=1}^{n}f(Z_k)
\]
where $f:\mathbb Z\rightarrow\mathbb R$ is absolutely summable.
This quantity is strongly related to the local time $\mathcal N_n$ of the RWRS $Z$, which is defined by
\[
\mathcal N_n(a)=\#\{k=1,...,n\, :\, Z_k=a\}\, .
\]
Indeed if $f=\mathbf 1_0$, then $\mathcal Z_n=\mathcal N_n(0)$ and if
$f=\mathbf 1_0-\mathbf 1_1$, then $\mathcal Z_n=\mathcal N_n(0)-\mathcal N_n(1)$. In the general case, $\mathcal Z_n$ can be rewritten
\[
\mathcal Z_n:=\sum_{a\in\mathbb Z}f(a)\mathcal N_n(a)\, .
\]
The asymptotic behaviour of $(\mathcal N_n(0))_n$ has been studied by Castell, Guillotin-Plantard, Schapira and the author in \cite[Corollary 6]{BFFN2}, in which it has been proved that the moments of $(n^{-\frac 14}\mathcal N_n(0))_{n\ge 1}$ converge to those of the local time $\mathcal L_1(0)$ at position 0 and until time 1 of the process $\Delta$.
The proof of this result was based on a multitime local limit theorem \cite[Theorem 5]{BFFN2}
extending a local limit theorem contained in \cite{BFFN1} and on the finiteness of the moments of $\mathcal L_1(0)$ (which was a delicate question).
We complete this previous work by establishing in Section \ref{boundmoment} the following bounds for the moments of $\mathcal L_1(0)$.
\begin{thm}\label{propboundmoment}
For any $\eta_0>0$, there exists $\mathfrak a>0$ and $C>0$ such that
\[
(Cm)^{\frac {3m}4}\le
\mathbb E[(\mathcal L_1(0))^{m}]=\mathcal O\left(\frac{\mathfrak a^m\, (m!)^{ \frac 32+\eta_0}}{\Gamma(\frac m4+1)}\right)\le \mathcal O\left(m^{m\left(\frac 54+2\eta_0\right)}\right)
\, .
\]
\end{thm}
Even if 
it uses some ideas that already existed in \cite{BFFN2}, the proof of Theorem~\ref{propboundmoment} (given in Section~\ref{boundmoment}) is different in many aspects. It requires indeed much more precise estimates which changes in the approach of the control of the moments.
The proof of Theorem~\ref{propboundmoment} relies on several auxiliary results. We summarize quickly its strategy.
We will prove (see \eqref{FFF} coming from \cite{BFFN2} and \eqref{FFF-1}) that
\[
\mathbb E[(\mathcal L_1(0))^{m}]=\frac {m!}{(2\pi \sigma^2_\xi)^{\frac m2}}\int_{0<t_1<...<t_m<1}\prod_{k=0}^{m-1}(t_{k+1}-t_k)^{-\frac 34}\mathbb E\left[\prod_{k=0}^{m-1}\left(d(L^{(k+1)},W_k)\right)^{-1}\right]\, dt_1...dt_m\, ,
\]
where we set $W_k:=Vect(L^{(1)},...,L^{(k)})$ and $L^{(k+1)}:=(L_{t_{k+1}}-L_{t_{k}})/(t_{k+1}-t_{k})^{\frac 34}$ (normalized so that $|L^{(m)}|_{L^2(\mathbb R)}$ has the same distribution as $|L_1|_{L^2(\mathbb R)}$).
We will prove, in Lemma \ref{LEMME00}, that
\[
\exists c,C>0,\quad
m!\, \int_{0<t_1<...<t_{m}<1}   \prod_{k=0}^{m-1}(t_{k+1}-t_k)^{-\frac 34}\, dt_1...dt_{m}\sim c(Cm)^{\frac {3m}4}\, ,
\]
as $m\rightarrow +\infty$
and, in Lemma~\ref{LEMME0}, that
\begin{equation}\label{encadrement}
\left(\mathbb E\left[\left|L_1
\right|_{L^2(\mathbb R)}^{-1}\right]\right)^m\le
\mathbb E\left[\prod_{k=0}^{m-1}\left(d(L^{(k+1)},W_k
)\right)^{-1}\right]\le \prod_{k=0}^{m-1}\left(\sup_{V\in \mathcal V_k}\mathbb E\left[\left( d\left(L_{1},V\right)\right)^{-1}\right]\right)\, ,
\end{equation}
where $d(\cdot,\cdot)$ is the distance associated with the $L^2$-norm on $
L^2(\mathbb R)$ and where $\mathcal V_k$ is the set of linear subspaces  of $L^2(\mathbb R)$ of dimension at most $k$.
Theorem~\ref{propboundmoment} will then follow from the next self-interesting estimate on the local time $L_1$ of the Brownian motion $B$ up to time 1.
\begin{thm}\label{EstiVk}
\begin{equation}\label{ESTIMk0}
\sup_{V\in \mathcal V_k}\mathbb E\left[\left( d\left(L_{1},V\right)\right)^{-1}\right]= k
^{\frac 12+o(1)}\, ,\quad\mbox{as }k\rightarrow +\infty\, .
\end{equation}
\end{thm}
Now we use the following classical argument for positive random variables.
The upper bound provided by Theorem~\ref{propboundmoment} allows us to prove that the Carleman's criterion is satisfied for
$\mathcal E\sqrt{\mathcal L_1(0)}$  where $\mathcal E$ is a centered Rademacher distribution independent of $\mathcal L_1(0)$ and of $Z$, indeed:
\[
\sum_{m\ge 1}\mathbb E[(\mathcal L_1(0))^{m}]^{-\frac 1{2m}}\ge
c_1\sum_{m\ge 1}m^{-\frac 58-\eta_0} =\infty\, ,
\]
for every $\eta_0\in(0,\frac 38)$.
This enables us to deduce from \cite[Corollary 6]{BFFN2} that $n^{-\frac 18}\mathcal E\sqrt{\mathcal N_n(0)}$ converges in distribution to $\mathcal E\sqrt{\sigma_\xi^{-1} \mathcal L_1(0)}$ and so that
\begin{equation}\label{cvgeloi}
n^{-\frac 14}\mathcal N_n(0)\stackrel{\mathcal L}{\longrightarrow}\sigma_\xi^{-1} \mathcal L_1(0)\, ,\quad\mbox{as}\ n\rightarrow +\infty\, ,
\end{equation}
where $\stackrel{\mathcal L}{\longrightarrow}$ means convergence in distribution.
This convergence in distribution is extended to more general observables.
\begin{thm}\label{LLN}
Let $f:\mathbb Z\rightarrow \mathbb R$ be such that $\sum_{a\in\mathbb Z}|f(a)|<\infty$.
Then $ n^{-\frac 14}\sum_{k=0}^{n-1}f(Z_k)$
converges in 
distribution
and in the sense of moments  
to $\sum_{a\in\mathbb Z}f(a)\sigma _\xi^{-1}\mathcal L_{1}(0)$.
\end{thm}
The proof of the moments convergence in Theorem~\ref{LLN} is a straigthforward adaptation of \cite{BFFN2} and is given in Appendix~\ref{Appendcvgcemoment}. 
Due to Theorem~\ref{propboundmoment} and to the above argument that lead to \eqref{cvgeloi},  
the convergence in distribution in Theorem~\ref{LLN} is a consequence of the moments convergence.
Another strategy to prove the convergence in distribution in Theorem~\ref{LLN}
consists in seing this result as a direct consequence of~\eqref{cvgeloi} combined with Proposition~\ref{invariant} stating the
ergodicity of the dynamical system $(\widetilde\Omega,\widetilde T,\widetilde\mu)$ corresponding to
\[
\widetilde T^k((X_{m+1})_{m\in\mathbb Z},(\xi_{m})_{m\in\mathbb Z},Z_{0})=((X_{k+m+1})_{m\in\mathbb Z},(\xi_{m+S_{k}})_{m\in\mathbb Z},Z_{k})\, .
\]
This dynamical system preserves the infinite measure $\widetilde\mu:=\mathbb P_{X_1}^{\otimes \mathbb Z}\otimes\mathbb P_{\xi_0}^{\otimes\mathbb Z}\otimes\lambda_{\mathbb Z}$, where $\lambda_{\mathbb Z}$ is the counting measure on $\mathbb Z$.
Actually, thanks to \eqref{cvgeloi} and to the recurrence ergodicity of  $(\widetilde\Omega,\widetilde T,\widetilde\mu)$, we prove  the following stronger version of the convergence in distribution of Theorem~\ref{LLN}.
\begin{thm}\label{LLNter}
For any $\widetilde\mu$-integrable function $\widetilde f:\widetilde\Omega\rightarrow\mathbb R$,
\[
n^{-\frac 14}\sum_{k=0}^{n-1}\widetilde f\circ\widetilde T^k
\stackrel{\mathcal L(\widetilde\mu)}{\longrightarrow}\frac{\int_{\widetilde\Omega}\widetilde f\, d\widetilde\mu}{\sigma_\xi} \mathcal L_1(0)\, ,\quad\mbox{as}\ n\rightarrow +\infty\, ,
\]
where $\stackrel{\mathcal L(\widetilde\mu)}{\longrightarrow}$ means
convergence in distribution with respect to any probability measure absolutely continuous with respect to $\widetilde\mu$.
\end{thm}

Theorem~\ref{LLN} can be seen as weak law of large numbers, with a non constant limit.
When $\sum_{a\in\mathbb Z}f(a)=0$, the limit given by Theorem \ref{LLN} vanishes, but then the next result provides a limit theorem for $\mathcal Z_n=\sum_{k=0}^{n-1}f(Z_k)$ with another normalization.
This second result corresponds to a central limit theorem for additive functionals
of RWRS.
Let us indicate that,
contrarily to the moments convergence in Theorem~\ref{LLN}, the next result is not an easy adaptation of \cite{BFFN2}, even if its proof (given in Section \ref{proofPAPATCL}) uses the same initial idea (computation of moments using the local limit theorem) and, at the begining, some estimates established in \cite{BFFN1,BFFN2}. Indeed, important technical difficulties arise from the cancellations coming from the fact that $\sum_{a\in\mathbb Z}f(a)=0$.
\begin{thm}\label{PAPATCL}
Assume moreover that there exists some $\kappa\in(0,1]$ such that $\xi_0$ admits a moment of order $2+\kappa$.
Let $f:\mathbb Z\rightarrow \mathbb R$ be such that
$\sum_{a\in\mathbb Z}(1+|a|)|f(a)|<\infty$ and that $\sum_{a\in\mathbb Z}f(a)=0$. Then
\[
\sum_{\ell\in\mathbb Z}\left|\sum_{\ell'=0}^{d
-1}\sum_{a,b\in\mathbb Z^2}f(a)f(b)\mathbb P(Z_{|\ell'+d
\ell|}=a-b)\right|<\infty\, .
\]
Moreover all the moments of
$ \left(n^{-\frac 18}\sum_{k=0}^{n-1}f(Z_k)\right)_n$
converges
to those of 
$\sqrt{
\frac{
\sigma^2_f}{\sigma_\xi 
}
\mathcal L_1(0)}\mathcal N$, where $\mathcal N$ is a standard Gaussian random variable independent of $\mathcal L_1(0)$ and where
\begin{equation}\label{eqsigma2}
\sigma^2_f:=\sum_{k\in\mathbb Z}\sum_{a,b\in\mathbb Z^2}f(a)f(b)\mathbb P(Z_{|k|}=a-b)\, .
\end{equation}
In particular, 
for any $a\in\mathbb Z$,
$ \left(n^{-\frac 18}(\mathcal N_n(a)-\mathcal N_n(0))\right)_n$
converges
in the sense of moments  
to $\sqrt{
\frac{
\sigma^2_{0,a}}{\sigma_\xi 
}
\mathcal L_1(0)}\mathcal N$,  with 
$\sigma^2_{0,a}:=\sum_{k\in\mathbb Z}\left[2\mathbb P(Z_{|k|}=0)-\mathbb P(Z_{|k|}=a)-\mathbb P(Z_{|k|}=-a)\right]$.
\end{thm}

Let us point out the similarity beween these results and the classical Law of Large Numbers and Central Limit Theorem for sums of square integrable independent and identically distributed random variables.
Indeed Theorems \ref{LLN} and \ref{PAPATCL} establish convergence results of the respective following forms
\[
\frac 1{a_n}\sum_{k=1}^nY_k\rightarrow I(Y_1)\mathcal Y\quad \mbox{and}\quad
\frac 1{\sqrt{a_n}}\sum_{k=1}^n(Y_k-I(Y_1)Y_k^0)\rightarrow \sqrt{\sigma^2_Y\mathcal Y}\, \mathcal Z
\]
as $n\rightarrow +\infty$, with $a_n\rightarrow +\infty$, $I$ an integral (with respect to the counting measure on $\mathbb Z$) and $Y_k^0$ a reference random 
variable with integral 1 (e.g. $Y_k^0=\mathbf 1_0(Z_k)$, note that we cannot take $Y_k^0=1$ since it is not integrable with respect to the counting measure on $\mathbb Z$).

The summation order in the expression  \eqref{eqsigma2} of $\sigma^2_f$ is important. Indeed recall that $\mathbb P(Z_k=0)$ has order $k^{-\frac 34}$ and so is not summable.
The sum $\sum_{k\in\mathbb Z}$ appearing in \eqref{eqsigma2} is a priori non absolutely convergent if $d\ne 1$. Indeed, considering for example that $\xi_0$ is a centered Rademacher random variable (i.e. $\mathbb P(\xi_0=1)=\mathbb P(\xi_0=-1)=\frac 12$) and that $f=\mathbf 1_0-\mathbf 1_1$, then, for any $k\ge 0$,
\[\sum_{a,b\in\mathbb Z^2}f(a)f(b)\mathbb P(Z_{|2k|}=a-b)=\mathbb P(Z_{|2k|}=0-0)+\mathbb P(Z_{|2k|}=1-1)=2\mathbb P(Z_{|2k|}=0)\]
and
\[\sum_{a,b\in\mathbb Z^2}f(a)f(b)\mathbb P(Z_{|2k+1|}=a-b)=-\mathbb P(Z_{|2k+1|}=0-1)-\mathbb P(Z_{|2k+1|}=1-0)=-\mathbb P(|Z_{|2k+1|}|=1)\, .
\]
But, $\sigma^2_f$ corresponds to the following sum of an absolutely convergent series (in $k$):
\[
\sigma^2_f=\sum_{k\in\mathbb Z}\left(\sum_{\ell'=0}^{d-1}\sum_{a,b\in\mathbb Z^2}f(a)f(b)\mathbb P(Z_{|\ell'+dk|}=a-b)\right)\, .
\]

Finally, let us point out that $\sigma^2_f$ defined in \eqref{eqsigma2}
corresponds to the Green-Kubo formula, well-known to appear in central limit theorems for probability preserving dynamical systems (see Remark~\ref{exprsigma2} at the end of Section~\ref{proofLLN}). 


Let us indicate that results similar to Theorem \ref{PAPATCL} exist for one-dimensional random walks, that is when the RWRS $(Z_n)_{n\ge 1}$ is replaced by the RW $(S_n)_{n\ge 1}$, with other normalizations and with an exponential random variable instead of $\mathcal L_1(0)$. Such results have been
obtained by Dobru\v{s}in \cite{Dobrushin}, Kesten in \cite{Kesten} and by Cs\'aki and F\"oldes in \cite{CF1,CF2}. The idea used therein was to construct a coupling using the fact that the times between successive return times of $(S_n)_{n\ge 1}$ to 0 are i.i.d., as well as the partial sum of the $f(S_k)$ between these return times to 0 and that these random variables have regularly varying tail distributions.
This idea has been adapted to dynamical contexts by Thomine \cite{ThomineThese,Thomine1}. Still in dynamical contexts, another approach based on moments has been developed in \cite{PT1,PT2} in parallel to the coupling method.
This second method based on local limit theorem is well tailored to treat non-markovian situations, such as RWRS. Indeed, recall that the RWRS $(Z_n)_{n\ge 1}$ is (strongly) stationary but far to be not markovian
(for example it has been proved in \cite{BFFN2} that $Z_{n+m}-Z_n$ is more likely to be 0 if we know that $Z_n=0$)
and even more intricate conditionally to the scenery (it has been proved in \cite{GPHS} that the RWRS does not converge knowing the scenery). Luckily local limit theorem type estimates
enables to prove moments convergence. But unfortunately Theorem~\ref{propboundmoment} is not enough to conclude the convergence in distribution
via Carleman's criterion.

The paper is organized as follows.
In Section
\ref{boundmoment}, we prove Theorem \ref{propboundmoment} (bounds on moments of the local time of the Kesten Spitzer process) and Theorem~\ref{EstiVk} (estimate on the distance in $L^2(\mathbb R)$ between the local time of a Brownian motion and a $k$-dimensional vector space).
In Section \ref{proofLLN}, we 
establish the recurrence ergodicity of the infinite measure preserving
dynamical system $(\widetilde\Omega,\widetilde T,\widetilde\mu)$
and obtain the convergence in distribution of Theorem~\ref{LLN} (Law of Large Numbers) as a byproduct of
this recurrence ergodicity combined with \eqref{cvgeloi}. Section \ref{proofLLN} is completed
by Appendix~\ref{Appendcvgcemoment} which contains the proof of the moments convergence of Theorem~\ref{LLN}.
In Section~\ref{proofPAPATCL} (completed with Appendix \ref{AppendA}), we prove  Theorem~\ref{PAPATCL} (Central Limit Theorem). 

\section{Upper bound for moments: Proof of Theorem~\ref{propboundmoment}}\label{boundmoment}
This section is devoted to the study of the behaviour of $\mathbb E[(\mathcal L_1(0))^{m}]$ as $m\rightarrow +\infty$.
It has been proved in \cite{BFFN2} that this quantity is finite, but the estimate established therein was not enough to apply the Carleman criterion. The proof of Theorem \ref{propboundmoment} requires a much more delicate study, even if it uses some estimates used in  \cite{BFFN2}.
We start by establishing bounds for $\mathbb E[(\mathcal L_1(0))^{m}]$.
\begin{lem}\label{LEMME0}
\begin{equation}\label{FFF0AA}
\left(\mathbb E\left[\left|L_1
\right|_{L^2(\mathbb R)}^{-1}\right]\right)^m \frac {m!}{(2\pi \sigma^2_\xi)^{\frac m2}}\int_{0<t_1<...<t_m<1}
 \prod_{k=0}^{m-1}(t_{k+1}-t_k)^{-\frac 34} \, dt_1...dt_m\le \mathbb E[(\mathcal L_1(0))^{m}]
\end{equation}
and
\begin{equation}\label{FFF0}
\mathbb E[(\mathcal L_1(0))^{m}]\le \prod_{j=0}^{m-1}\left(\sup_{V\in \mathcal V_k}\mathbb E\left[\left( d\left(L_{1},V\right)\right)^{-1}\right]\right)\frac {m!}{(2\pi \sigma^2_\xi)^{\frac m2}}\int_{0<t_1<...<t_m<1}
 \prod_{k=0}^{m-1}(t_{k+1}-t_k)^{-\frac 34}\, dt_1...dt_m\, ,
\end{equation}
where 
$d(f,g)=\vert f-g\vert_{L^2(\mathbb R)}$ and where $\mathcal V_k$ is the set of linear subspaces  of $L^2(\mathbb R)$ of dimension at most $k$.
\end{lem}
\begin{proof}
Recall that it has been proved  in \cite[Theorem 3]{BFFN2} that
\begin{equation}\label{FFF}
\mathbb E[(\mathcal L_1(0))^{m}]=\frac {m!}{(2\pi \sigma^2_\xi)^{\frac m2}}\int_{0<t_1<...<t_m<1}\mathbb E[(\det \mathcal D_{t_1,...,t_m})^{-\frac 12}]\, dt_1...dt_m\, ,
\end{equation}
with $ \mathcal D_{t_1,...,t_m}:=\left(\int_{\mathbb R} L_{t_i}(x)L_{t_j}(x)\, dx\right)_{i,j=1,...,m}$ where $(L_t(x))_{t\ge 0,x\in\mathbb R}$ is the local time
of the Brownian motion $B$.
Since $\det \mathcal D_{t_1,...,t_m}$ is a Gram determinant, we have the iterative relation
\[
\det \mathcal D_{t_1,...,t_{m+1}}^{\frac 12}=\det \mathcal D_{t_1,...,t_{m}}^{\frac 12}d(L_{t_{m+1}},Vect(L_{t_1},...,L_{t_m}))\, ,
\]
where $d(f,g)=\Vert f-g\Vert_{L^2(\mathbb R)}$ and where $Vect(L_{t_1},...,L_{t_m})$ is the sublinear space of $L^2(\mathbb R)$ generated
by $L_{t_1},...,L_{t_m}$. It follows that
\begin{equation}\label{FFF-1}
\det \mathcal D_{t_1,...,t_{m}}^{-\frac 12}=\prod_{k=0}^{m-1}\left(d(L_{t_{k+1}},Vect(L_{t_1},...,L_{t_k}))\right)^{-1}\, .
\end{equation}
But, for any $m\ge 1$ and any $0< t_1<...<t_{m+1}<1$ and any $k=0,...,m-1$,
\begin{align*}
\nonumber&\mathbb E\left[\left.d\left(L_{t_{k+1}},Vect(L_{t_1},...,L_{t_k})\right)^{-1}\right|(B_s)_{s\le t_k}\right]\\
\nonumber&=\mathbb E\left[\left.d\left(L_{t_{k+1}}-L_{t_k},Vect(L_{t_1},...,L_{t_k})\right)^{-1}\right|(B_s)_{s\le t_k}\right]\\
\nonumber&=\mathbb E\left[\left.d\left((L_{t_{k+1}}-L_{t_k})(B_{t_k}+\cdot),Vect(L_{t_1}(B_{t_k}+\cdot),...,L_{t_k}(B_{t_k}+\cdot))\right)^{-1}\right|(B_s)_{s\le t_k}\right]\, .
\end{align*}
Therefore
\begin{equation}
\mathbb E\left[\left|L_{t_{k+1}}-L_{t_k}
\right|_{L^2(\mathbb R)}^{-1}\right]
\le\mathbb E\left[\left.d\left(L_{t_{k+1}},Vect(L_{t_1},...,L_{t_k})\right)^{-1}\right|(B_s)_{s\le t_k}\right]
\label{FFF-2mino}
\end{equation}
and
\begin{equation}
\mathbb E\left[\left.d\left(L_{t_{k+1}},Vect(L_{t_1},...,L_{t_k})\right)^{-1}\right|(B_s)_{s\le t_k}\right]\le  \sup_{V\in\mathcal V_k}\mathbb E\left[d\left((L_{t_{k+1}}-L_{t_k})(B_{t_k}+\cdot),V\right)^{-1}\right]\, ,\label{FFF-2}
\end{equation}
where $\mathcal V_k$ is the set of linear subspaces of dimension at most $k$ of $L^2(\mathbb R)$ and 
where we used the independence of $(L_{t_{k+1}}-L_{t_k})(B_{t_k}+\cdot)$
with respect to $(B_s)_{s\le t_k}$ and the fact that $(L_{t_1}(B_{t_k}+\cdot),...,L_{t_k}(B_{t_k}+\cdot))$ is measurable with respect to $(B_s)_{s\le t_k}$.
Thus, by induction and using the fact that the increments of $B$ are (strongly) stationary, it follows from \eqref{FFF-1} and \eqref{FFF-2} that
\begin{align}
\nonumber
\prod_{k=0}^{m-1}\mathbb E\left[\left|L_{t_{k+1}}-L_{t_k}
\right|_{L^2(\mathbb R)}^{-1}\right]
\le \mathbb E\left[\det \mathcal D_{t_1,...,t_{m}}^{-\frac 12}\right]
&\le \prod_{k=0}^{m-1}\sup_{V\in \mathcal V_k}\mathbb E\left[\left( d\left((L_{t_{k+1}}-L_{t_k})(B_{t_k}+\cdot),V\right)\right)^{-1}\right]\\
&= \prod_{k=0}^{m-1}\sup_{V\in \mathcal V_k}\mathbb E\left[\left( d\left(L_{t_{k+1}-t_k},V\right)\right)^{-1}\right]\, ,\label{FFF-3}
\end{align}
with the convention $t_0=0$. Recall that $(L_u(x))_{x\in\mathbb R}$ has the same distribution $(\sqrt{u}L_1(x/\sqrt{u}))_{x\in\mathbb R}$ and so 
$
\left(d(L_u,Vect(g_1,...,g_k))\right)^2$ has the same distribution as 
\begin{align*}
\min_{a_1,...,a_k}\int_\mathbb R \left(\sqrt{u}L_1\left(\frac x{\sqrt{u}}\right)-\sum_{i=1}^ka_ig_i(x)\right)^2\, dx
&=
u\min_{a'_1,...,a'_k}
\int_\mathbb R \left(L_1\left(\frac x{\sqrt{u}}\right)-\sum_{i=1}^ka'_ig_i(x)\right)^2\, dx\\
&= u^{\frac 32}
\min_{a'_1,...,a'_k}
\int_\mathbb R \left(L_1\left(y \right)-\sum_{i=1}^ka'_ig_i(\sqrt{u}y)\right)^2\, dy\\
&=  u^{\frac 32} \left(d(L_1,Vect(h_1,...,h_k))\right)^2
\end{align*}
setting $a'_i:=a_i/\sqrt{u}$, and making the change of variable $y=x/\sqrt{u}$, with
$h_i(x)=g_i(\sqrt{u}x)$
and so \eqref{FFF-3} becomes
\begin{equation}
\prod_{k=0}^{m-1}\left((t_{k+1}-t_k)^{-\frac 34}\mathbb E\left[\left|L_1
\right|_{L^2(\mathbb R)}^{-1}\right]\right) \le
\mathbb E\left[\det \mathcal D_{t_1,...,t_{m}}^{-\frac 12}\right]
\le \prod_{k=0}^{m-1}(t_{k+1}-t_k)^{-\frac 34}\sup_{V\in \mathcal V_k}\mathbb E\left[\left( d\left(L_{1},V\right)\right)^{-1}\right]\,  ,\label{FFF0bis}
\end{equation}
which ends the proof of the lemma.
\end{proof}
We first study the behaviour, as $m\rightarrow +\infty$, of the integral appearing in Lemma~\ref{LEMME0}.
\begin{lem}\label{LEMME00}
\[
m!\, \int_{0<t_1<...<t_{m}<1}   \prod_{k=0}^{m-1}(t_{k+1}-t_k)^{-\frac 34}\, dt_1...dt_{m}=\frac{m!\, \Gamma(\frac 14)^m}{\Gamma(\frac{m}4+1)}\sim c(Cm)^{\frac {3m}4}\, ,
\]
as $m\rightarrow +\infty$.
\end{lem}
\begin{proof}
\begin{align*}
a_{m+1}&:=\int_{0<t_1<...<t_{m+1}<1}   \prod_{k=0}^{m}(t_{k+1}-t_k)^{-\frac 34}\, dt_1...dt_{m+1}\\
&=\int_{x_i>0\, :\, x_1+...+x_{m+1}<1}   \prod_{k=1}^{m+1}x_k^{-\frac 34}\, dx_1...dx_{m+1}\\
&=\int_0^1 x_{m+1}^{-\frac 34}(1-x_{m+1})^{-\frac {3m}4}\left(\int_{x_i>0\, :\, x_1+...+x_{m}<1-x_{m+1}}   \prod_{k=1}^{m}(x_k/(1-x_{m+1}))^{-\frac 34}\, dx_1...dx_m\right)dx_{m+1}\\
&=\int_0^1 x_{m+1}^{-\frac 34}(1-x_{m+1})^{\frac m4 }  \left(\int_{u_i>0\, :\, u_1+...+u_{m}<1}   \prod_{k=1}^{m}u_k^{-\frac 34}\, du_1...du_m\right)dx_{m+1}\\
&=a_m\int_0^1 x_{m+1}^{-\frac 34}(1-x_{m+1})^{\frac m4 } dx_{m+1}=a_mB\left(\frac 14,\frac m4+1\right)=a_m\frac{\Gamma(\frac 14)\Gamma(\frac m4+1)}{\Gamma(\frac{m+1}4+1)}\, ,
\end{align*}
where $B(\cdot,\cdot)$ and $\Gamma$ stand respectively for Euler's Beta and Gamma functions, and so, by induction, $ a_m=\frac{\Gamma(1/4)^m}{\Gamma(\frac m4+1)}$ proving the first point of the lemma. Moreover
\begin{align*}
m! a_m&\sim\left(\Gamma(1/4)\right)^m m^{m+\frac 12} (m+4)^{-\frac m4-\frac 12}   4^{\frac m4+\frac 12}e^{-\frac {3m}4+1}\, ,
\end{align*}
where we used the Stirling formulas $m!=\Gamma(m+1)$ and $\Gamma(z)\sim\sqrt{2\pi}z^{z-\frac 12}e^{-z}$. This ends the proof of the lemma.
\end{proof}
Observe that $\mathbb E\left[| L_{1}|_{L^2(\mathbb R)}^{-1}\right]>0$.
Thus, the proof of Theorem~\ref{propboundmoment} will be be deduced from the two previous lemmas combined with Theorem~\ref{EstiVk}, which can be rewritten as follows 
\begin{equation}\label{ESTIMk}
\forall \eta_0>0,\ \exists C>1,\ \forall k\in\mathbb N^*,\quad C^{-1}k^{\frac 12-\eta_0}\le \sup_{V\in \mathcal V_k}\mathbb E\left[\left( d\left(L_{1},V)\right)\right)^{-1}\right]\le C k
^{\frac 12+
\eta_0}\, .
\end{equation}
Due to  \cite[Cor. (1.8) of Chap. VI, Theorem (2.1) of Chap. I]{RY}, $L_1$ is almost surely H\"older continuous of order
$\frac 12-\eta_0$ and its H\"older constant admits moments of any order.
The lower bound of theorem~\ref{EstiVk} follows directly from this fact.
\begin{proof}[Proof of the lower bound of Theorem~\ref{EstiVk}]
We prove the lower bound of \eqref{ESTIMk}. Let $\eta_0\in(0,\frac 12)$.
Let $\mathcal C_1$ be the H\"older constant of order $\frac 12-\eta_0$ of $L_1$.
Let $V_k$ be the linear subspace of $L^2(\mathbb R)$ generated by the set
\[
\left\{\mathbf 1_{[m/k,(m+1)/k]},\ m=-\left\lfloor\frac{k}2\right\rfloor,...,\left\lceil\frac{k}2\right\rceil-1\right\}\, ,
\]
and consider $\widetilde L_k\in V_k$ given by
\[
\widetilde L_k:=\sum_{m=-\lfloor \frac k2\rfloor}^{\lceil \frac k2\rceil-1} L_1\left(\frac mk\right)\mathbf 1_{[\frac mk,\frac{m+1}k)}\, .
\]
Let $K_0>0$.
We will use the fact that
\[
\mathbb E\left[\left( d\left(L_{1},V_k\right)\right)^{-1}\right]\ge
\mathbb E\left[\left( d\left(L_{1},V_k\right)\right)^{-1}\mathbf 1_{\{\mathcal C_1\le
K_0
,\ \sup_{[0,1]}|B|\le \frac{k-1}{2k}\}}\right]\, .
\]
Observe that, if $\sup_{[0,1]}|B|\ge \frac{k-1}{2k}$ and $\mathcal C_1\le
K_0
$, then
\begin{align*}
 d\left(L_{1},V_k\right)^2&\le d(L_1,\widetilde L_k)^2=\sum_{m=\lfloor \frac k2\rfloor}
^{\lceil\frac k2\rceil-1}\int_{\frac mk}^{\frac{m+1}k}(L_1(u)-L_1(m/k))^2\, du\\
&\le  \sum_{m=\lfloor \frac k2\rfloor}
^{\lceil\frac k2\rceil-1}k^{-1}\left( 
K_0
k^{-\frac 1 2+ \eta_0}\right)^2
\le\left( 
K_0
k^{-\frac 1 2+ \eta_0}\right)^2\, .
\end{align*}
Thus
\begin{align*}
\mathbb E\left[\left( d\left(L_{1},V_k\right)\right)^{-1}\right]&\ge \mathbb E\left[\left( d\left(L_{1},V_k\right)\right)^{-1}\mathbf 1_{\{\mathcal C_1\le K_0,\ \sup_{[0,1]}|B|\le  \frac{k-1}{2k}\}}\right]\\
&\ge \mathbb E\left[\left(K_0 k^{-\frac 12+\eta_0}\right)^{-1}\mathbf 1_{\{\mathcal C_1\le K_0,\ \sup_{[0,1]}|B|\le  \frac{k-1}{2k}\}}\right]\\
&\ge  K_0^{-1} k^{\frac 12-\eta_0}\, \mathbb P\left(\mathcal C_1\le K_0,\ \sup_{[0,1]}|B|\le  \frac 13\right)\, .
\end{align*}

\end{proof}
The rest of this section is devoted to the proof of the upper bound of  Theorem~\ref{EstiVk} (i.e. the upper bound of \eqref{ESTIMk}), which is much more delicate to establish. To this end, we will prove a sequence of estimates.
Let us first introduce the quantities used in this proof.
We fix $\eta_0>0$ and
$
d=\frac 12+\eta_0>1/2
$.
Choose $\epsilon_0\in(0,\frac 1{10})$ such that 
\begin{equation}\label{hypd}
d>\frac {1+\epsilon_0}{2(1-\epsilon_0)}\, .
\end{equation}
Fix $a,b,\eta,\gamma\in(0,\frac 1{10})$ such that $0<\frac b8<\frac a2
$
and 
small enough so that
\begin{equation}\label{hypab}
 \frac{(1+\gamma)(1+\epsilon_0)}2+\frac a2+\frac b8 <1
\end{equation}
and
\begin{equation}\label{hypeta}
(2d(1-\epsilon_0)-1-\epsilon_0)(1-2\eta) -8\eta>0\, .
\end{equation}
Let $\theta>0$ such that
$(1-2\eta)\theta>1$ and
\begin{equation}\label{hyptheta}
1-\frac b4-\frac{(1+\gamma)(1+\epsilon_0)}2<\theta(1-2\eta)\left( 1-\frac{(1+\gamma)(1+\epsilon_0)}2-\frac a2-\frac b8 \right)
\, 
\end{equation}
and
\begin{equation}\label{hypthetabis}
(1-\epsilon_0)(1+2d)<\theta\left[(2d\left(1-\epsilon_0\right)-1-\epsilon_0)(1-2\eta) -8\eta\right] \, .
\end{equation}
The existence of such a $\theta$ is ensured by \eqref{hypab} and \eqref{hypeta}.
Fix then $K$ such that $\frac 1{4a- b}<K
$ and $v_0=\lceil 16/ b\rceil$.
We will also consider the following quantities which will depend on $k\ge 1$. 
We set $M:=\lceil \theta k\rceil$
and $M':=M^d$.
For $x>M'$, we also set: 
\begin{equation}\label{r0x0x1}
r_0:=
(x/M')^{-(1+\gamma)(1+\epsilon_0)}M^{-\frac {1+\epsilon_0}2}{M'}^{-1-\epsilon_0
},\quad x_0=(x/M')^aM,\quad
x_1=(x/M')^b\, .
\end{equation}
Let  $V$ be a linear space generated by $g_1,...,g_k\in L^2(\mathbb R)$. 
Observe that 
\begin{align}
\nonumber\mathbb E\left[\left( d\left(L_{1},V)\right)\right)^{-1}\right]&=\int_0^\infty
\mathbb P\left(\left( d\left(L_{1},V)\right)\right)^{-1}>x\right)\, dx\\
&=\mathcal O(M')+\int_{M'}^\infty
\mathbb P\left( d\left(L_{1},V)\right)<x^{-1}
\right)\, dx\, .\label{AABBCC}
\end{align}
\begin{lem}
Uniformly on $x>M'$:
\begin{align}
\nonumber&\mathbb P\left( d\left(L_{1},V)\right)<x^{-1}
\right)\\
&\ \ \ \ \ \ \le \mathcal O\left((x/M')^{-2}\right)+
\mathbb P\left(\forall \ell=-v_0,...,v_0,\ 
D\left(\left(L_{1}\left(\ell x_1^{-\frac 18}+\frac n{x_0}\right)\right)_{n=1,...,M},W^{(\ell x_1^{-\frac 18})}_V\right)< 2 x^{-1}r_0^{-\frac 12}
\right)\, .\label{FFF1}
\end{align}
where $W^{(y_0)}_V:=Span\left(\left(\int_{y_0
+n/x_0}^{y_0
+(n+1)/x_0
}g_j(y')\, dy'\right)_{n =1,...,M},\ 1,...,k
\right) \subset\mathbb R^M$ and where $D$ is the usual euclidean metric  in $\mathbb R^{M}$.
\end{lem}
\begin{proof}
We set
\[
\mathcal C_1:=\sup_{y,z\in\mathbb R\, :\, y\ne z}\frac{|L_1(y)-L_1(z)|}{|y-z|^u},\quad \mbox{with}\ u:=\frac 1{1+\epsilon_0}-\frac 12\, .
\]
Since $\mathcal C_1$ admits moments of every  order, it follows that
\[
\mathbb P\left( d\left(L_{1},V\right)<1/x\right)\le \mathbb P\left( d\left(L_{1},V)\right)<1/x,\ 
\mathcal C_1\le (x/M')^{\gamma}\right)
+\mathcal O((x/M')^{-2})
\, ,\]
Note that, if $x>M'$, then
\[
r_0x_0=(x/M')^{a-(1+\gamma)(1+\epsilon_0)}M^{\frac {1-\epsilon_0}2}{M'}^{-1-\epsilon_0
}\le 1\, ,
\]
as soon as $x>M'$, since $a<1<(1+\gamma)(1+\epsilon_0)$ and since
$M'=M^d$ with $\frac 12\le d$, and so $r_0\le x_0^{-1}$.
Assume moreover that $d(L_1,V)<1/x$ and $\mathcal C_1\le x^{\gamma}$.
Let $a_j$ such that $d\left(L_{1},\sum_{j=1}^ka_jg_j\right)<x^{-1}$.
Then, for every $\ell\in\mathbb Z$, the following estimate holds true
\begin{align*}
x^{-1}&>\left(\sum_{n=1}^M\int_{\ell x_1^{-\frac 18}+\frac n{x_0}}^{\ell x_1^{-\frac 18}+\frac n{x_0}+r_0}\left(L_{1}(y)-\sum_{j=1}^ka_jg_j(y)\right)^2\, dy\right)^{\frac 12}\\
&\ge \left(\sum_{n=1}^M
\int_{\ell x_1^{-\frac 18}+\frac n{x_0}}^{\ell x_1^{-\frac 18}+\frac n{x_0}+r_0}\left(L_{1}\left(\ell x_1^{-\frac 18}+\frac n{x_0}\right)-\sum_{j=1}^ka_jg_j(y)\right)^2\, dy\right)^{\frac 12}
- \left(M r_0 (x/M')^{2\gamma} r_0^{2u}
\right)^{\frac 12}\\
&\ge \sqrt{r_0} 
D\left(\left(L_{1}\left(\ell x_1^{-\frac 18}+\frac n{x_0}\right)\right)_{n=1,...,M},W_V^{(\ell x_1^{-\frac 18})}
\right)-\sqrt{M} (x/M')^{\gamma} r_0^{\frac 12+u}\, .
\end{align*}
Since $\frac 12+u=\frac 1{1+\epsilon_0}$ and $r_0=
(x/M')^{-(1+\gamma)(1+\epsilon_0)}M^{-\frac {1+\epsilon_0}2}{M'}^{-1-\epsilon_0
}$, we conclude that $\sqrt{M} (x/M')^{\gamma} r_0^{\frac 12+u}=x^{-1}$ and so
\begin{align*}
\nonumber&\mathbb P\left( d\left(L_{1},V)\right)<1/x,\ \mathcal C_1\le x^{\gamma}\right)\\
&\ \ \ \ \ \ \le 
\mathbb P\left(\forall \ell=-v_0,...,v_0,\ 
D\left(\left(L_{1}\left(\ell x_1^{-\frac 18}+\frac n{x_0}\right)\right)_{n},W^{(\ell x_1^{-\frac 18})}_V\right)< 2 x^{-1}r_0^{-\frac 12}
\right)\, .
\end{align*}
\end{proof}
Recall that $v_0=\lceil 16/ b\rceil$. Set
\[
E_{0,W}:=\left\{D\left(\left(
x_0\int_{\frac n{x_0}}^{\frac {n+1}{x_0}
}Y'(y)\, dy\right)_{n=1,...,M},W\right)<2 x^{-1}\sqrt{x_0}
\right\}\, .
\]
\begin{lem}\label{LEMM1}
The following estimate holds true uniformly on $x>M$:
\begin{align}
\mathbb P\left( d\left(L_{1},V)\right)<x^{-1}
\right)\le \mathcal O\left((x/M')^{-2}\right)
+\sup_W\mathbb P\left(E_{0,W}
\right)\, ,
\end{align}
where $\sup_W$ means the supremum over the set of linear subspaces $W$ of $\mathbb R^M$ of dimension at most $k$ and where $Y'$ is a squared Bessel process of dimension 0 starting from $x_1^{-\frac14}$.
\end{lem}
\begin{proof}
We adapt the proof of  \cite[Lemma 9]{BFFN2}.
Setting $\epsilon':=x_1^{-\frac 18}$ and $T_u:=\min\{s>0\, :\, |B_s|=u\}$ for the first hitting time of $\{\pm u\}$ by the
Brownian motion $B$, we observe that
\begin{equation}\label{Hitting}
\mathbb P(T_{v_0\epsilon'
}> 1)=\mathbb P\left(\sup_{s\in[0,1]}|B_s|\le v_0\epsilon'
\right)=\mathcal O(e^{-c_0(v_0\epsilon'
)^{-2}})=\mathcal O\left( (x
/M'
)
^{-bv_0/8}\right)
=\mathcal O\left( (x/M')^{-2}\right)
\, .
\end{equation}
(using e.g. \cite[Proposition 8.4, page 52]{PS}).
 Moreover, due to \cite[Exercise 4.12, Chapter VI,
p 265]{RY}, for every $n=0,...,v_0-1$,
\[
\mathbb P\left(L_{T_{(n+1)\epsilon'}}(B_{T_{n\epsilon'}})-L_{T_{n\epsilon'}}(B_{T_{n\epsilon'}})\le (\epsilon')^2|(B_u)_{u\le T_{n\epsilon'}}\right)\le \mathbb P(L_{T_{\epsilon'}}(0)\le (\epsilon')^2)\le \epsilon'
\]
and so, due to the strong Markov property,
\[
\mathbb P\left(\forall n=0,...,v_0-1,\ L_{T_{(n+1)\epsilon'}}(B_{T_{n\epsilon'}})-L_{T_{n\epsilon'}}(B_{T_{n\epsilon'}})\le (\epsilon')^2
\right)\le (\epsilon')^{v_0}\, ,
\]
and this, combined with \eqref{Hitting}, ensures that there exists $C_0>0$ such that
$\mathbb P(\forall \ell=-v_0,...,v_0,\ L_1(\ell \epsilon')\le(\epsilon')^2)\le C_0(\epsilon')^{v_0}$ and so
\begin{equation}\label{FFF2b}
\mathbb P\left(\forall \ell=-v_0,...,v_0,\ 
\mathfrak t_\ell(x_1)>1
\right)
\le C_0
 (x
/M')
^{-bv_0/8}\le C_0 (x /M') ^{-2}\ ,
\end{equation}
setting $\mathfrak t_\ell(x_1):=\inf\{s>0\, :\, L_s(\ell x_1^{-\frac 18})>x_1^{-\frac 14}\}$.
Moreover, for any $\ell=1,...,v_0$, we have
\begin{align}
\nonumber&
\sup_V\mathbb P\left(\mathfrak t_\ell(x_1)<1,\ D\left(\left(L_{1}\left(\ell x_1^{-\frac 18}+\frac n{x_0}\right)\right)_{n},W^{(\ell x_1^{-\frac 18})}_V\right)< 2 x^{-1}
\sqrt{x_0}
\right)\\
\nonumber&\ \ \ \le\sup_V\mathbb P\left(\mathfrak t_\ell(x_1)<1,\ D\left(\left(L_{\mathfrak t_\ell(x_1)}\left(\ell x_1^{-\frac 18}+\frac n{x_0}\right)\right)_{n}
,W_{V,0}^{(\ell x_1^{-\frac 18})}\right)<
2 x^{-1}
\sqrt{x_0}
\right)\\
&\ \ \ \le \sup_W\mathbb P\left(
D\left(\left(L_{\mathfrak t_\ell(x_1)}\left(\ell x_1^{-\frac 18}+\frac n{x_0}\right)\right)_{n},W\right)<2 x^{-1}
\sqrt{x_0}
\right)\, ,\label{FFF20}
\end{align}
setting $W_{V,0}^{(\ell x_1^{-\frac 18})}:=W^{(\ell x_1^{-\frac 18})}_V+\left((L_{\mathfrak t_\ell(x_1)}-L_1)\left(\ell x_1^{-\frac 18}+\frac n{x_0}\right)\right)_{n}$.
Due to the second Ray-Knight theorem (see \cite[Theorem 2.3, page 456]{RY}), $(L_{\mathfrak t_\ell(x_1)}(\ell x_1^{-\frac 18}+y))_{y\ge 0}$ has the same distribution as $Y'$.
The lemma follows from \eqref{FFF2b} and \eqref{FFF20}.
\end{proof}
Recall that $4a-b>0$.
Set 
\[
E_1:=\left\{ \sup_{s \le M/x_0} |Y'(s)-x_1^{-1/4}| <\frac{x_1^{-1/4}}2  \right\}\, .
\]

\begin{lem}\label{LEMM2}
For every $K>(4a-b)^{-1}$, the following estimate holds true uniformly on $x>M$:
\begin{equation}\label{FFF2c}
 \PP\left(E_1\right)=1-\mathcal O\left((x/M')^{-K(4a-b)}
\right)\, .
\end{equation}
\end{lem}
\begin{proof}
Using the Burkholder-Davis-Gundy inequality, combined with the fact that $Y'$
is dominated by the square of a Brownian motion starting from $x_1^{-\frac 14}$, we observe that
\begin{align}\nonumber
\mathfrak p_x &=
 \PP\cro{\sup_{s \le M/x_0} |Y'(s)-x_1^{-1/4}| \ge\frac{x_1^{-1/4}}2  } \\ 
\nonumber&  \le   C_K\, x_1^{2K}
 2^{8K}
\, \EE\cro{\left(\int_0^{M/x_0} Y'(u) \, du\right)^{4K} }\\
\label{FFF21}& \le  C'_K\,  
x_1^{2K}
(M/x_0)^{4K-1} \int_0^{M /x_0} \EE\cro{Y'(u)^{4K}} \, du
\end{align}
with $C'_K=2^{8K}C_K$, and so
\begin{align*}
\nonumber\mathfrak p_x
&\le C'_K\,  
x_1^{2K}
(M/x_0)^{4K-1}\int_0^{M /x_0} \EE\cro{ (x_1^{-1/8} + B_u)^{8K}} \, du\\ 
&\le C'_K\,  
x_1^{2K}
(M/x_0)^{4K} 2^{8K} \left(x_1^{-K} + (M/x_0)^{4K}\right)\\ 
&\le C''_K\left(
x_1^{K}
(M/x_0)^{4K}  + x_1^{2K}(M/x_0)^{8K}\right)
\end{align*}
with $C''_K=2^{8K}C'_K$ and
\begin{align*}
x_1^{K} (M/x_0)^{4K}& =(x/M')^{-K(4a-b)}
\, ,
\end{align*}
since $ x_0=(x/M')^aM$ and $x_1=(x/M')^b$.
\end{proof}

\begin{lem}\label{LEMM4}
Uniformly on $x>M'$,
\begin{align*}
\sup_W\, \mathbb P\left(E_{0,W}\cap E_1\right) \le&
{C''}^k   (x/M')^{\left[1-\frac b4-\frac{(1+\gamma)(1+\epsilon_0)}2\right]k-M(1-2\eta)\left[ 1-\frac{(1+\gamma)(1+\epsilon_0)}2-\frac a2-\frac b8 \right]}\\
&\quad \times M^{\frac {(1-\epsilon_0)k}4+\frac{(1+\epsilon_0)(1-2\eta) M}4+2\eta M}{M'}^{\left(\frac{1-\epsilon_0}2\right)(k-(1-2\eta)M)}
\, .
\end{align*}
\end{lem}
\begin{proof}
Observe that
\begin{equation}\label{interE1E2}
E_{0,W}\cap E_1
\subset\left\{\left(Y'(n/x_0)\right)_{n=1,...,M}\in \B_\infty\left(x_1^{-1/4}, \frac{x_1^{-1/4}}2\right)\cap W_x
\right\}\, ,
\end{equation}
where $\B_\infty\left(x_1^{-1/4}, \frac{x_1^{-1/4}}2\right)$ is the ball (for the supremum norm)
of radius $ \frac{x_1^{-1/4}}2 $ and centered on $(x_1^{-1/4},...,x_1^{-1/4})$, and 
 where $W_x$ is the $ \varepsilon=
2 x^{-1}r_0^{-\frac 12}$-neighbourhood of $W$ for the metric $D$. Note that
\begin{equation}\label{epsilon}
\epsilon=2(x/M')^{\frac{(1+\gamma)(1+\epsilon_0)}2-1}M^{\frac {1+\epsilon_0}4}{M'}^{\frac{1+\epsilon_0}2-1}
\end{equation}
and
\begin{equation}\label{Rx}
\mathcal R_x:=\frac{\sqrt{M} \, x_1^{-1/4}}\varepsilon
=\frac 12(x/M')^{1-\frac b4-\frac{(1+\gamma)(1+\epsilon_0)}2}M^{\frac {1-\epsilon_0}4}{M'}^{1-\frac{1+\epsilon_0}2}
\gg 1\, ,
\end{equation}
uniformly in $x>M'$, since $\frac b4+\frac{(1+\gamma)(1+\epsilon_0)}2<1$.
Observe that   
 $\B_\infty\left(x_1^{-1/4}, \frac{x_1^{-1/4}}2\right)\cap W_x$ is contained in $\B_2\left(x_1^{-1/4},\sqrt{M}  
x_1^{-1/4}
\right)\cap W_x$
where $\B_{2}\left(x_1^{-1/4},\sqrt{M}  
x_1^{-1/4}
\right )$ is the euclidean ball 
centered on $(x_1^{-1/4},...,x_1^{-1/4})$ with radius $\sqrt{M}  
x_1^{-1/4}
$.\\
Let  $z_0,z'_0\in\B_2\left(x_1^{-1/4},\sqrt{M}  x_1^{-1/4}\right)\cap W_x$ and
$z_1\in W\cap \B_2(z_0,\varepsilon)$, $z'_1\in W\cap \B_2(z'_0,\varepsilon)$. 
Then $z'_1\in
 \B_2\left(z_1,3\sqrt{M}  x_1^{-1/4}\right)$.
Due to \cite[Theorem 3, pages 157]{Rogers}, there exists $c>0$ such that $W\cap  \B_2\left(z_1,3\sqrt{M}  x_1^{-1/4}\right)$
is contained in the union of at most $(c\mathcal R_x)^k$ euclidean balls
of radius $\varepsilon$ in $W$. Thus $W_x\cap  \B_2\left(x_1^{-\frac 14},\sqrt{M}  x_1^{-1/4}\right)$ is contained in the union of at most $(c\mathcal R_x)^k$ euclidean balls
of radius $2\varepsilon$.
We conclude that $\B_\infty\left(x_1^{-1/4}, \frac{x_1^{-1/4}}2\right)\cap W_x$  is contained in the union of at most
$(c\mathcal R_x)^k$ euclidean balls of radius $4\varepsilon$ centered at a point contained in
$\B_\infty\left(x_1^{-1/4}, \frac{x_1^{-1/4}}2\right)\cap W_x$. 
It follows from this combined with \eqref{interE1E2} that  
\begin{align}
 &  \PP\cro{ E_{0,W}\cap E_1
} 
 \le
(c\mathcal R_x)^k
\, \sup_{z \in 
\B_{\infty}\left(x_1^{-1/4},\frac{x_1^{-1/4}}2\right)}
\PP\pare{
(Y'(n/x_0))_{n=1,...,M}\in \B_{2}(z,4\varepsilon)
} \, .\label{interm}
\end{align}
Note that if $z\in \B_{\infty}\left(x_1^{-1/4},\frac{x_1^{-1/4}}2\right)$ and $(Y'(n/x_0))_{n=1,...,M}\in \B_{2}(z,4\varepsilon)$, then $ \max_{n=0,...,M-1} |z_{n+1}-x_1^{-\frac 14}|<\frac{x_1^{-\frac 14}}2$ and there exist at most $\eta M$ indices $n'$ that $|Y'(n'/x_0)-z_{n'}|\ge 4 \varepsilon/\sqrt{\eta M}$, and so at least
$(1-2\eta)M$ indices $n=\{0,...,M-1\}$ such that
\begin{align*}
&
\left(\left|Y'\left(\frac n{x_0}\right)-z_n\right|,\left|Y'\left(\frac{n+1}{x_0}\right)-z_{n+1}\right|\right)
< 4\varepsilon/\sqrt{\eta M}
\, ,
\end{align*}
with $z_0=x_1^{-\frac 14}$.
Due to \cite[after Corollary 1.4, page 441]{RY},
the distribution of $Y'((n+1)/x_0)$ knowing $Y'(n/x_0)=y$ is the sum of a Dirac mass at $0$ 
and of a measure with density 
$$z\mapsto q_{x_0}(y,z) := \frac{x_0}2 \sqrt{\frac y z} 
 \exp\pare{- \frac{x_0(y+z)}{2}} I_1\left(x_0\sqrt{yz}\right),$$
where $I_1$ is the modified Bessel function of index $1$ which satisfies
$I_1(z) = \O(e^z/\sqrt z)$, as $z\to \infty$, (see \cite[(5.10.22) or (5.11.10)]{L}). 
So
\[
q_{x_0}(y,z)=\mathcal O\left(x_0^{\frac 12}x_1^{\frac 18}  
 \exp\pare{- \frac{x_0(\sqrt{y}-\sqrt{z})^2}{2}} \right)
=\mathcal O\left(x_0^{\frac 12}x_1^{\frac 18}  
\right)
\]
uniformly on
$y,z\in \left[\frac {x_1^{-\frac 14}}4,2x_1^{-\frac 14}\right]$.
We will use the expression $x_0$, $x_1$ and $\epsilon$ given in  \eqref{r0x0x1} and \eqref{epsilon}.
Thus by using the Markov property (and $\frac{M!}{(M(1-2\eta))!(2\eta M)!}\le M^{2\eta M}$),
we get by induction, that, when $x>M'$,
\begin{align}
\nonumber&\sup_{z \in \B_{\infty}\left(x_1^{-1/4},\frac{x_1^{-1/4}}2\right)}\PP\left( (Y'(n/x_0))_{n=1,...,M}\in \B_{2}(z,4\varepsilon)\right)\\
 \nonumber&\le  M^{2\eta M} \left( C'(x/M')^{-\left(1-\frac{(1+\gamma)(1+\epsilon_0)}2-\frac a2-\frac b8\right)}M^{\frac {1+\epsilon_0}4}{M'}^{-1+\frac{1+\epsilon_0}2}\right)^{(1-2\eta) M}\, .
\end{align}
Recalling that $M=\mathcal O(k)$, the previous estimate
combined with \eqref{interm} and \eqref{Rx} ensures that
\begin{align}
\nonumber \sup_W\mathbb P\left(
E_{0,W}\cap E_1
\right)
\nonumber \le &{C''}^k   (x/M')^{\left[1-\frac b4-\frac{(1+\gamma)(1+\epsilon_0)}2\right]k-M(1-2\eta)\left[ 1-\frac{(1+\gamma)(1+\epsilon_0)}2-\frac a2-\frac b8 \right]}\\
&\quad M^{\frac {(1-\epsilon_0)k}4+\frac{(1+\epsilon_0)(1-2\eta) M}4+2\eta M}{M'}^{\left(1-\frac{1+\epsilon_0}2\right)(k-(1-2\eta)M)}\, ,
\end{align}
which ends the proof of the lemma.
\end{proof}
\begin{proof}[Proof of the upper bound of Theorem~\ref{EstiVk}]
Formula \eqref{ESTIMk} follows from \eqref{AABBCC} and Lemmas~\ref{LEMM1}, \ref{LEMM2} 
and \ref{LEMM4}. We will use the fact that
\begin{equation}\label{intQ}
\forall Q>1,\quad \int_{M'}^\infty(x/M')^{-Q}\, dx=\mathcal O(M')\, .
\end{equation}
Thanks to this, the error terms in Lemmas \ref{LEMM1} and \ref{LEMM2} gives directly a term in $\mathcal O(M')=\mathcal O(k^d)$.
Let us detail the term coming from Lemma~\ref{LEMM4}.
We first observe that the exponent of $(x/M')$ is strictly smaller than -1 for $k$ large enough. Indeed this exponent is 
\[
\left[1-\frac b4-\frac{(1+\gamma)(1+\epsilon_0)}2\right]k-M(1-2\eta)\left[ 1-\frac{(1+\gamma)(1+\epsilon_0)}2-\frac a2-\frac b8 \right]
\]
which is smaller than
\[
k\left[1-\frac b4-\frac{(1+\gamma)(1+\epsilon_0)}2-\theta(1-2\eta)\left( 1-\frac{(1+\gamma)(1+\epsilon_0)}2-\frac a2-\frac b8 \right)\right]
\]
where we used the fact that $M=\lceil \theta k\rceil\ge \theta k$. The fact that this quantity is strictly smaller than -1 for any $k$ large enough comes from our conditions \eqref{hypab} and \eqref{hyptheta}.
It follows from this combined with \eqref{intQ} and Lemma \ref{LEMM4} that
\begin{align*}
&\int_{M'}^{+\infty}\sup_W\, \mathbb P\left(E_{0,W}\cap E_1\cap E_2\right)\, dx
\\
&\le {C''}^k M^{\frac {(1-\epsilon_0)k}4+\frac{(1+\epsilon_0)(1-2\eta) M}4+2\eta M}{M'}^{1+\left(\frac{1-\epsilon_0}2\right)(k-(1-2\eta)M)} \\
&\le {C''}^k M^{d+\frac {(1-\epsilon_0)(1+2d)M}{4\theta}+\frac{(1+\epsilon_0-2d(1-\epsilon_0))(1-2\eta) M}4+2\eta M} \, ,
\end{align*}
where we used the fact that $M'=M^d$ and that $k\le \lceil \theta k\rceil/\theta=M/\theta$.
Finally, we notice that $1+\epsilon_0-2d(1-\epsilon_0)<0$ (due to \eqref{hypd}) and that \eqref{hypthetabis} ensures that
\[
\frac {(1-\epsilon_0)(1+2d)}{4\theta}+\frac{(1+\epsilon_0-2d(1-\epsilon_0))(1-2\eta) }4+2\eta <0\, 
\]
and conclude that
\[
\int_{M'}^{+\infty}\sup_W\, \mathbb P\left(E_{0,W}\cap E_1\cap E_2\right)\, dx=\mathcal O(1)\, .
\]
\end{proof}
\section{Law of large numbers: Proof of Theorem~\ref{LLN}}\label{proofLLN}
We complete the sequence $(X_n)_{n\ge 1}$ into a bi-infinite sequence $(X_n)_{n\in\mathbb Z}$ of i.i.d. random variables.
Theorem~\ref{LLN} could be proved by an adaptation of the proof of
\cite[Corollary 6]{BFFN2} (combined with Theorem~\ref{propboundmoment}).
We use here another approach enabling the study of more general additive functionals.
Recall that $(\xi_{m+S_k})_{m\in\mathbb Z}$ is the scenery seen from the particle at time $k$.
\begin{prop}\label{LLNbis}
Let $\widetilde f:\mathbb Z^{\mathbb Z}\times \mathbb Z^{\mathbb Z}\times\mathbb Z\rightarrow \mathbb R$ be a measurable function such that 
\[
\sum_{\ell\in\mathbb Z}|\mathbb E[\widetilde f((X_{n+1})_{n\in\mathbb Z}, (\xi_n)_{n\in\mathbb Z},\ell)]|<\infty\, .
\]
Then
\[
\left( \frac{\sum_{k=0}^{n-1}\widetilde f((X_{m+k+1})_{m\in\mathbb Z},(\xi_{m+S_k})_{m\in\mathbb Z},
Z_{k+m})}{\mathcal N_n(0)}\right)_{n\ge 0}
\]
converges almost surely to $I(\widetilde f):=\sum_{\ell\in\mathbb Z}\mathbb E[\widetilde f((X_n)_{n\in\mathbb Z},(\xi_{n})_{n\in\mathbb Z},\ell)]$.\\
In particular, 
this combined with \eqref{cvgeloi}
ensures that

\[
\left( n^{-\frac 14}\sum_{k=0}^{n-1}\widetilde f((X_{m+k+1})_{m\in\mathbb Z},(\xi_{m+S_k})_{m\in\mathbb Z},Z_k)\right)_{n\ge 0}
\]
converges in distribution  to $I(\widetilde f)\sigma_\xi^{-1}\mathcal L_{1}(0)$.
\end{prop}
Our approach to prove Proposition~\ref{LLNbis}
uses an ergodic point of view.
Let us consider the probability preserving dynamical system $(\Omega,T,\mu)$ given by 
\[
\Omega=\mathbb Z^{\mathbb Z}\times
\mathbb Z^{\mathbb Z}
,\quad T((x_k)_{k\in\mathbb Z},(y_k)_{k\in\mathbb Z})=
((x_{k+1})_{k\in\mathbb Z},(y_{k+x_0})_{k\in\mathbb Z}),\quad  \mu=\mathbb P_{X_1}^{\otimes \mathbb Z}\otimes\mathbb P_{\xi_0}^{\otimes\mathbb Z}\, ,
\]
i.e. $T\left(\boldsymbol x,\boldsymbol y\right)=\left(\sigma \boldsymbol x,\sigma^{x_0}\boldsymbol y\right)$, where we write $\sigma:\mathbb Z^{\mathbb Z}\rightarrow\mathbb Z^{\mathbb Z}$ for the usual shift transformation given by
$\sigma\left((z_k)_{k\in\mathbb Z}\right)=(z_{k+1})_{k\in\mathbb Z}$.

This system is known to be ergodic (see \cite{Weiss,Kalikow}).
We set $\Phi(x,y):=y_0$. With these notations, $Z_k$ corresponds to the Birkhoff sum
$\sum_{k=0}^{n-1}\Phi\circ T^k$.
Consider the $\mathbb Z$-extension $(\widetilde\Omega,\widetilde T,\widetilde\mu)$
over  $(\Omega,T,\mu)$  with step function $\Phi$. This system is given by
\[
\widetilde\Omega:=\Omega\times\mathbb Z,\quad\widetilde\mu=\mu\otimes \lambda_{\mathbb Z}\, ,
\]
where $\lambda_{\mathbb Z}=\sum_{\ell\in\mathbb Z}\delta_\ell$ is the counting measure on $\mathbb Z$ and with  
\[
\widetilde T(x,y,\ell)=(T(x,y),\ell+y_0)\, .
\]
In particular
\[
\widetilde T^k\left((x_{m+1})_{m\in\mathbb Z},(y_{m})_{m\in\mathbb Z},\ell\right)
=\left((x_{m+k+1})_{m\in\mathbb Z},(y_{m+x_0+...+x_{k-1}})_{m\in\mathbb Z},
\ell+\sum_{j=0}^{k-1}y_{x_0+...+x_{j}}\right)\, .
\]
Observe that $\mathcal N_n(0)$ corresponds to the Birkhoff sum $\sum_{k=0}^{n-1}h_0\circ\widetilde T^k(\boldsymbol x,\boldsymbol y,0)$ with $h_0(\boldsymbol x,\boldsymbol y,\ell)=\mathbf 1_0(\ell)$, and the sum studied in Proposition~\ref{LLNbis} corresponds to $\sum_{k=0}^{n-1}\widetilde f\circ \widetilde T^k(\boldsymbol x,\boldsymbol y,0)$, while $I(\widetilde f)=\int_{\widetilde\Omega}\widetilde f\, d\widetilde\mu$.
\begin{prop}\label{invariant}
The system  $(\widetilde\Omega,\widetilde T,\widetilde\mu)$ is recurrent ergodic.
\end{prop}
\begin{proof}
Since $(\Omega,T,\mu)$ is ergodic and since $\Phi$ is integrable and $\mu$-centered, we know (by \cite[Corollary 3.9]{Schmidt} combined with the Birkhoff ergodic theorem) that 
$\mathbb P(Z_n=0 \ i.o.)=1$, thus  that
$(\widetilde\Omega,\widetilde T,\widetilde\mu)$
 is recurrent (i.e. conservative).
Now let us prove that this system is also ergodic. Let $g:\widetilde\Omega\rightarrow(0,+\infty)$ be a positive $\widetilde\mu$-integrable function such that $g(\boldsymbol x,\boldsymbol y,\ell)=g_0(\ell)$ does not depend on $(\boldsymbol x,\boldsymbol y)\in\Omega$ and with unit integral ($g$ is a probability density function with respect to $\widetilde\mu$). By recurrence of $(\widetilde\Omega,\widetilde T,\widetilde\mu)$, we know that
\begin{equation}\label{ginfinite}
\sum_{k\ge 1}g\circ\widetilde T^k=\infty 
\end{equation}
$\widetilde\mu$-almost everywhere.
Let $K\in\mathbb N$. Consider $f:\widetilde\Omega\rightarrow\mathbb R$ a $\widetilde\mu$-integrable
function constant on the $K$-cylinders of the first coordinate, i.e. such that  $f(\boldsymbol x,\boldsymbol y,\ell)=f_0((x_m)_{|m|\le k},\boldsymbol y,\ell)$ does not depend on $(x_k)_{|k|>K}$.

Since $(\widetilde\Omega,\widetilde T,\widetilde\mu)$ is recurrent, the Hopf-Hurewicz's theorem (see e.g. \cite[p. 56]{Aa}) ensures that
\begin{equation}\label{Hopf}
\lim_{|n|\rightarrow +\infty}\frac{\sum_{k=1}^nf\circ\widetilde T^k}{\sum_{k=1}^ng\circ\widetilde T^k}=H_{(f,g)}:=\mathbb E_{g\widetilde\mu}\left[\left.\frac fg\right|\widetilde{\mathcal I}\right]
\end{equation}
$\widetilde\mu$-almost everywhere, where $\widetilde{\mathcal I}$ is the $\sigma$-algebra
of $\widetilde T$-invariant events.
Thus the ergodicity of $(\widetilde\Omega,\widetilde T,\widetilde\mu)$ will follow from the fact that $H_{(f,g)}$ is $\widetilde\mu$-almost everywhere constant for every $f$ as above ($g$ can be fixed). 
Observe that, for $k>K$, 
\begin{align*}
f\circ \widetilde T^{k}(\boldsymbol x,\boldsymbol y,\ell)&=
f\left(\sigma^k\boldsymbol x,\sigma^{x_0+...+x_{k-1}}\boldsymbol y,\ell+
      \sum_{m=0}^{k-1}y_{x_0+...+x_{m}}\right)\\
&=
f_0\left(x_{k-K},...,x_{K+k},\sigma^{x_0+...+x_{k-1}}\boldsymbol y,\ell+
      \sum_{m=0}^{k-1}y_{x_0+...+x_{m}}\right)
\end{align*}
does not depend on $(x_k)_{k\le -1}$. Analogously,  for $k>K$, 
\begin{align*}
f\circ \widetilde T^{-k}(\boldsymbol x,\boldsymbol y,\ell)&=
f\left(\sigma^{-k}\boldsymbol x,\sigma^{-x_{-1}-...-x_{-k}}\boldsymbol y,\ell-
      \sum_{m=1}^{k}y_{-x_{-1}-...-x_{-m}}\right)\\
&=
f_0\left(x_{-K-k},...,x_{-(k-K)},\sigma^{-x_{-1}-...-x_{-k}}\boldsymbol y,\ell-
      \sum_{m=1}^{k}y_{-x_{-1}-...-x_{-m}}\right)
\end{align*}
does not depend on $(x_k)_{k\ge 0}$. Of course $g\circ\widetilde T^k$ satisfies the same property. 
Thus, due to  \eqref{ginfinite} and \eqref{Hopf}, it follows that $H(f,g)(\boldsymbol x,\boldsymbol y,\ell)$ does not depend on $\boldsymbol x$.
Thus, $H_{(f,g)}(\boldsymbol x,\boldsymbol y,\ell)=H^{(0)}_{(f,g)}(\boldsymbol y,\ell)$ for $\widetilde\mu$-almost every $(\boldsymbol x,\boldsymbol y,\ell)\in\widetilde\Omega$.

By $\widetilde T$-invariance of $H_{(f,g)}$, given two distinct points $x_0,x'_0\in\mathbb Z$ such that $\mathbb P(X_1=x_0)\mathbb P(X_1=x'_0)>0$,
the following equality holds true almost everywhere
\[
H^{(0)}_{(f,g)}(\boldsymbol y,\ell)=H^{(0)}_{(f,g)}(\sigma^{x_0}\boldsymbol y,\ell+y_0)=H^{(0)}_{(f,g)}(\sigma^{x'_0}\boldsymbol y,\ell+y_0)\, ,
\]
where we write $\sigma$ for the usual shift on $\mathbb Z^{\mathbb Z}$ given by $\sigma((y_k)_{k\in\mathbb Z})=(y_{k+1})_{k\in\mathbb Z}$.
It follows that, for every $\ell\in\mathbb Z$, $H^{(0)}_{(f,g)}(\cdot,\ell)$ is $\sigma^{x_0-x'_0}$-invariant almost everywhere. By ergodicity of $\sigma^{x_0-x'_0}$, we conclude that $H_{(f,g)}(\boldsymbol x,\boldsymbol y,\ell)=H_{f,g}^{(1)}(\ell)$ depends only on $\ell$ almost everywhere. Since it is $\widetilde T$-invariant, for every $y_0\in\mathbb Z$ such that $\mathbb P(\xi_0=y_0)>0$,
$H_{f,g}^{(1)}(\ell)=H_{f,g}^{(1)}(\ell+y_0)$. Since the support of $y_0$
generates the group $\mathbb Z$, we conclude that $H_{(f,g)}$ is $\widetilde\mu$-almost everywhere equal to a constant.
\end{proof}
Note that the system in infinite measure $(\widetilde\Omega,\widetilde T,\widetilde\mu)$
describes the evolution in time $m$ of $((X_{m+k+1})_{k\in\mathbb Z},(\xi_{S_m+k})_k,Z_{m})$. In comparison, the system corresponding 
to $((X_{m+k+1})_k,S_{m})$ is also recurent ergodic, but the analogous system
corresponding to $((X_{m+k+1})_k,(\xi_{S_m+k})_k,S_{m})$ is recurrent (since $\mathbb P(S_n=0\ i.o.)=1$) not ergodic (since the
sets of the form $\{(x,y,\ell)\, :\, (y_{n-\ell})_n\in A_0\}$ are invariant).

\begin{proof}[Proof of Proposition~\ref{LLNbis}]
Since $(\widetilde\Omega,\widetilde T,\widetilde\mu)$ is recurrent ergodic, the Hopf ergodic theorem ensures that, for any $\widetilde f\in L^{1}(\widetilde\mu)$, the sequence
$\left(\frac{\sum_{k=0}^{n-1}\widetilde f\circ \widetilde T^k}{\sum_{k=0}^{n-1}\widetilde h_0\circ \widetilde T^k}\right)_{n\ge 0}$
converges $\widetilde\mu$-almost everywhere to $\frac{\int_{\widetilde\Omega}\widetilde f\, d\widetilde\mu}{\int_{\widetilde\Omega}\widetilde h_0\, d\widetilde\mu}=I(\widetilde f)$. Thus
\[
\left(\frac{\sum_{k=0}^{n-1}\widetilde f((X_{m+k+1})_{m\in\mathbb Z},(\xi_{m+S_k})_{m\in\mathbb Z},
Z_{k+m})}{\mathcal N_n(0)}=\frac{\sum_{k=0}^{n-1}\widetilde f\circ \widetilde T^k}{\sum_{k=0}^{n-1}\widetilde h_0\circ \widetilde T^k}((X_m)_{m\in\mathbb Z},(\xi_{m})_{m\in\mathbb Z},0)\right)_{n\ge 0}
\]
converges almost surely to $I(\widetilde f)$,
and we have proved 
the first part of the proposition.
The second part comes from the first part combined with \eqref{cvgeloi} and the Slustky theorem.
\end{proof}
\begin{proof}[Proof of Theorem~\ref{LLNter}]
Proposition~\ref{LLNbis} states that $n^{-\frac 14}\sum_{k=0}^{n-1}\widetilde f\circ\widetilde T^k$ converges in distribution, with respect to $\mu\otimes\delta_0\ll\widetilde\mu$, to
 $\int_{\widetilde\Omega}\widetilde f\, d\widetilde\mu\,\sigma_\xi^{-1}\mathcal L_{1}(0)$.
Thus, Theorem~\ref{LLNter} follows from Proposition~\ref{LLNbis} combined with
\cite[Theorem 1]{Zwei}.
\end{proof}
We end this section with an interpretation of  $\sigma^2_f$ in terms of the
famous Green-Kubo formula.
\begin{rqe}\label{exprsigma2}
Assume the assumptions of Theorem~\ref{PAPATCL}.
consider the function $\widetilde f:\widetilde\Omega\rightarrow \mathbb Z$
given by $\widetilde f(\boldsymbol x,\boldsymbol y,\ell):=f(\ell)$. Then $\sigma^2_f$ can be rewritten
\[
\sigma^2_f=\sum_{k\in\mathbb Z}\int_{\widetilde\Omega}\widetilde f.\widetilde f\circ \widetilde T^{|k|}\, d\widetilde \mu\, .
\]
\end{rqe}

\section{Proof of the central limit theorem: proof of Theorem~\ref{PAPATCL}}\label{proofPAPATCL}
We start by stating key intermediate results.
We recall that $d$
and $\alpha$ have been introduced in the beginning of Section \ref{mainresults}.
\begin{prop}\label{1:Point1}
Assume the assumptions of Theorem~\ref{PAPATCL}.
Let $M\in\mathbb N^*$ and $\eta>0$.
There exists $L\in(0,1)$ such that for every $\theta\in(0,1)$
the following holds true with the notations $n_j:=k_j-k_{j-1}$, with the convention $k_0=0$.

First,
\begin{equation}\label{majoproba}
\sum_{k'_j=0,...,d-1,\ \forall j\in\mathcal J}
 \mathbb E\left[  \prod_{j=1}^m \left(f(Z_{k_j+k'_j})\prod_{s=1}^{s_j}f(Z_{k_j+\ell_{j,s}})\right)\right]
=\mathcal O\left(\left(
\prod_{i=1}^mn_i^{-\frac 34}\right)\mathfrak E_{\boldsymbol k}
\right)
\, ,
\end{equation}
uniformly over the $\boldsymbol k=(k_1,...,k_m)$ and
$\boldsymbol \ell=(\ell_{j,s})_{j=1,...,m;s=1,...,s_j}$ such that
$n>k_j>k_{j-1}+n^\theta$ (with convention $k_0:=0$) and $\ell_{j,s}\in\{0,...,\lfloor n^{L\theta}\rfloor\}$  with $M=\sum_{j=1}^m(s_j+1)$, where we set $\mathcal J:=\{j=1,...,m\, : s_j=0\}$ and $k'_j=0$ if $j\not\in \mathcal J'$, and with
\[
\mathfrak E_{\boldsymbol k}=\mathcal O\left(
\sum_{\mathcal J'\subset\{1,...,m\}\, :\, \#\mathcal J'\ge \#\mathcal J/2}
\left(\prod_{j\in\mathcal J'}n_j^{-\frac 12+\eta}\right)\right)\, .
\]

Second, if $s_j= 1$ for all $j$, then
\[\mathbb E\left[   \prod_{j=1}^m\left(f(Z_{k_j})f(Z_{k_j+\ell_{j}})\right)\right]=
\frac{ d^m E_{\boldsymbol k}}{
(2\pi\sigma_\xi^2)^{\frac m2}
}\prod_{j=1}^m\mathcal A_{k_j,\ell_{j}}+\mathcal O\left(n^{-L(M+1)\theta}\prod_{j=1}^mn_j^{-\frac {3}4}\right)\,  ,\]
uniformly on $\boldsymbol k,\boldsymbol \ell$ as above,
with $E_{\boldsymbol k}$ depending on $\boldsymbol k$ but not on $\boldsymbol\ell$ and such that $E_{\boldsymbol k}=\mathcal O\left(\prod_{j=1}^mn_j^{-\frac 34}\right)$ uniformly on $\boldsymbol k$ as above, and
$E_{\boldsymbol k}\sim n^{-\frac{3m}4}\mathbb E\left[\det \mathcal D_{t_1,...,t_m}^{-\frac 12}\right]$ as $k_j/n\rightarrow t_j$ and $n\rightarrow +\infty$, 
with $\mathcal D_{t_1,...,t_m}=(\int_{\mathbb R} L_{t_i}(x)L_{t_j}(x)\, dx)_{i,j=1,...,m}$ where $L$ is the local time of the brownian motion $B$,
 limit of $(S_{\lfloor nt\rfloor}/\sqrt{n})_t$ as $n$ goes to infinity,
and where
\[
\mathcal A_{k,\ell}:=
\sum_{a\in k\alpha+d\mathbb Z,\ b\in\mathbb Z}\left(f(a)\prod_{s=1}^{m}f(b)\right)\mathbb P(Z_{\ell}=b-a)\, .
\]
Third, 
\begin{align*}
\sum_{k'_1,...,k'_m=0}^{d-1}
\sum_{\ell_1,...,\ell_m=0}^{n^{\frac{\kappa\theta\eta}{10 M}}}
2^{\#\{j:\ell_j>0\}}
\prod_{j=1}^m\mathcal A_{k_j+k'_j,\ell_j}
=\sigma^{2m}_f+o(1)
\, ,
\end{align*}
as $(k_1/n,...,k_m/n)\rightarrow (t_1,...,t_m)$ and $n\rightarrow +\infty$.
\end{prop}
\begin{proof}
The proof of Proposition~\ref{1:Point1} is based on several technical lemmas. For
reader convenience, the most technical points are proved in Appendix \ref{AppendA}.
Let $M\ge 1$, $\theta\in(0,1)$
and $\eta \in \left(0,\frac 1{100}\right)$.
Choose $L=\frac{\kappa\eta}{10M}$.
Assume $n^\theta<n_j<n$
and let $\ell_{j,1},...,\ell_{j,s_j}=0,..., \lfloor n^{
L\theta}\rfloor$ 
with $\sum_{j=1}^m(1+s_j)=M$.
We set $N'_j(y):=
\#\{s=0,...,n_j-1\, :\, S_{k_{j-1}+s}=y\}$,
$ N_{j}^* := \sup_y N'_j$ and $R'_{j} := \#\{y\in\mathbb Z\, :\, N'_j
(y)>0\}$.
Analogously, we set $N'_{j,s}=\#\{s=0,...,\ell_{j,s}-1\, :\, S_{k_{j}+s}=y\}$.
The left hand side of \eqref{majoproba} can be written
\begin{equation}\label{CCCC1}
B_{\boldsymbol k,\boldsymbol\ell}=
\sum_{\boldsymbol a,\boldsymbol b}\left(\prod_{j=1}^m\left( 
f(a_j)
\prod_{s=1}^{s_j}f(b_{j,s})
\right)
\right)p_{\boldsymbol k,\boldsymbol\ell}({\boldsymbol a,\boldsymbol b})\, ,
\end{equation}
where $\sum_{\boldsymbol a,\boldsymbol b}$ means the sum over $(\boldsymbol a,\boldsymbol b)\in\mathbb Z^M$ with $\boldsymbol a=(a_1,...,a_m)
$ and
$\boldsymbol b=(b_{j,s})_{j=1,...,m;s=1,...,s_j}$,
with the convention $a_0=0$ and
\begin{align}
\nonumber&p_{\mathbf{k},\boldsymbol{\ell}}(\boldsymbol{a},\boldsymbol{b})=\mathbb P(\forall j=1,...,m,\ Z_{k_j}=a_j,\ \forall s=1,...,s_j,\ Z_{k_j+\ell_{j,s}}=b_{j,s})\, .
\end{align}
An classical computation (detailed in Appendix \ref{AppendA}) ensures  the following.
\begin{lem}\label{LEM1}
\begin{align}
p_{\mathbf{k},\boldsymbol{\ell}}(\boldsymbol{a},\boldsymbol{b})
=\mathbf 1_{\{\forall i,\, a_i=k_i\alpha+d\mathbb Z\}}\frac {d^m}{(2\pi)^M}
   \int_{[-\frac{\pi}d,\frac\pi d]^{m}\times[-\pi,\pi]^{M-m}} e^{-i\sum_{j=1}^m[(a_j-a_{j-1})\theta_j+\sum_{s=1}^{s_j}(b_{j,s}-a_j)\theta'_{j,s}]}\varphi_{\boldsymbol{k},\boldsymbol{\ell}}(\boldsymbol\theta,\boldsymbol\theta')
   \, d(\boldsymbol\theta,\boldsymbol \theta')\, .\label{1:EqFourier}
\end{align}
with  $\boldsymbol\theta=(\theta_j)_{j=1,...,m}$ and $\boldsymbol\theta'=(\theta'_{j,s})_{j=1,...,m;s=1,...,s_j}$ and
\begin{eqnarray}
\label{1:phini0}
\varphi_{\boldsymbol{k},\boldsymbol{\ell}}(\boldsymbol\theta,\boldsymbol\theta')
 = \EE \left[\prod_{y\in \ZZ}
  \varphi_{\xi}\left( \sum_{j=1}^m\left(\theta_jN'_{j}(y)+\sum_{s=1}^{s_j}\theta'_{j,s}N'_{j,s}(y))\right)\right)\right]\, .
\end{eqnarray}
\end{lem}
For any event $E$ and any $I\subset [-\frac{\pi}d,\frac\pi d]^{m}\times[-\pi,\pi]^{M-m}$, we also set
\begin{align}
\label{1:phini000}
\varphi_{\boldsymbol{k},\boldsymbol{\ell}}(\boldsymbol\theta,\boldsymbol\theta',E)
 = \EE \left[\mathbf 1_E\prod_{y\in \ZZ}
  \varphi_{\xi}\left( \sum_{j=1}^m\left(\theta_jN'_{j}(y)+\sum_{s=1}^{s_j}\theta'_{j,s}N'_{j,s}(y))\right)\right)\right]\, ,
\end{align}
\begin{align}
p_{\mathbf{k},\boldsymbol{\ell}}(\boldsymbol{a},\boldsymbol{b},I,E)
=\mathbf 1_{\{\forall i,\, a_i=k_i\alpha+d\mathbb Z\}}\frac {d^m}{(2\pi)^M}   \int_{I} e^{-i\sum_{j=1}^m[(a_j-a_{j-1})\theta_j+\sum_{s=1}^{s_j}(b_{j,s}-a_j)\theta'_{j,s}]}\varphi_{\boldsymbol{k},\boldsymbol{\ell}}(\boldsymbol\theta,\boldsymbol\theta',E')
   \, d(\boldsymbol\theta,\boldsymbol \theta')\, .\label{1:EqFourierbisbis}
\end{align}
and
\begin{equation}\label{BIE}
B_{\boldsymbol k,\boldsymbol\ell,I,E}=
\sum_{\boldsymbol a,\boldsymbol b}\left(\prod_{j=1}^m\left( 
f(a_j)
\prod_{s=1}^{s_j}f(b_{j,s})
\right)
\right)p_{\boldsymbol k,\boldsymbol\ell}({\boldsymbol a,\boldsymbol b},I,E)\, .
\end{equation}

Let $\gamma<\min(L\theta,\frac{\eta\theta} {2M})$. 
Let $\theta'\in(0,\frac {\theta\eta}2)$ such that $\theta'\le\frac \theta 2-2M
L\theta$.
We consider the set
\begin{equation}\label{defOmegak}
\Omega_{\boldsymbol k} := \left\{\det D_{\boldsymbol k}\ge n^{-\theta'}\prod_{i=1}^mn_i^{\frac 32}\right\}\cap\bigcap_{j=1}^m\Omega^{(j)}_{\boldsymbol k}\, ,
\end{equation}
with
\[
\Omega^{(j)}_{\boldsymbol k}:=
\left\{   
\sup_{r=0,...,n_j}|S_{r+k_{j-1}}-S_{k_{j-1}}| \le \frac {n_j^{\frac 1 2+\gamma}}3,\  \quad \sup_{y \ne z} 
\frac{\vert N'_{j}(y)-N'_{j}(z)\vert}{|y-z|^{\frac{1}{2}}} 
\leq n_j^{\frac 1 4 +\frac\gamma 2} 
\right\}\, ,
\]
and with $D_{\boldsymbol k}=\left(\sum_{y\in\mathbb Z}N'_{i}(y)N'_{j}(y)\right)_{i,j}$.
The following lemma follows from \cite{BFFN2} (see appendix \ref{AppendA} for details).
\begin{lem}\label{LEM2}
For any $p>1$, $\PP(\Omega_{\boldsymbol k})=1-o(n^{-p})$,
and so
$B_{\mathbf{k},\boldsymbol{\ell},
\left[-\frac{\pi}d,\frac\pi d\right]^{M},\Omega_{\boldsymbol k} ^c}
=o(n^{-p})$.
\end{lem}
Note that, on $\Omega_{\boldsymbol k}$,
\begin{align}
&R'_j\le n_j^{\frac 12+\gamma}\, ,\\
&N_j
^*=|N_j^*-0|\le n_j^{\frac 14+\frac \gamma 2}((n_j)^{\frac 12+\gamma})^{\frac 12}
\ll n_j^{\frac 12+\frac\eta 2}\, ,\\
&V_{j}:=\sum_{z\in\mathbb Z} (N'_{j}(z))^2 \ge \frac{\left(\sum_{z\in\mathbb Z}N'_j(z)\right)^2}{R'_j}\ge\frac{n_j^2}{n_j^{\frac 12+\gamma}}\ge n_j^{\frac 32 - \frac\eta 2}\, ,\label{1:minVn}\\
&V_{j}\le R'_j(N_j^*)^2\le n_j^{\frac{3(1+\eta)}2}\, .\label{1:maxVn}
\end{align} 
It will be useful to notice that
\begin{equation}\label{majovarphi}
\left|\varphi_{\boldsymbol{k},\boldsymbol{\ell}}(\boldsymbol\theta,\boldsymbol\theta', E)\right|\le
\EE \cro{\mathbf 1_E\prod_{y\in\mathcal F} \left|\varphi_{\xi}\left(\sum_{j=1}^m\theta_jN'_{j}(y)\right)\right|}
\end{equation}
with
\[
\mathcal F:=\left\{y\in\mathbb Z\, :\, \forall (j,s),\ N'_{j,s}(y)=0\right\}\, ,
\]
and that
\begin{equation}\label{Z-F}
\#(\mathbb Z\setminus\mathcal F)\le \sum_{j=1}^m\sum_{s=1}^{s_j}\ell_{j,s}\le M n^{L\theta}=o(n^{\frac 14})\, .
\end{equation}
Using a straighforward adaptation of the proof of \cite[Proposition 10]{BFFN1}, we prove (see Appendix \ref{AppendA}) that
\begin{lem}\label{LEM3}
\begin{equation}\label{1:Neglec00}
B_{\boldsymbol k,\boldsymbol\ell,I^{(1)}_{\boldsymbol k},\Omega_{\boldsymbol k}}=o\left(e^{- n^c}\right) \, ,
\end{equation}
uniformly on $\boldsymbol k,\boldsymbol\ell$ as in Proposition~\ref{1:Point1},
where $I^{(1)}_{\boldsymbol k}$ is the set of $(\boldsymbol\theta,\boldsymbol\theta')\in\left[-\frac\pi d,\frac \pi d\right]^m\times[-\pi,\pi]^{M-m}$ such that there exists
$j=1,...,m$ so that $n_j^{-\frac 12+{\eta}}<|\theta_j|$.
\end{lem}

\begin{lem}\label{LEM4}
\begin{equation}\label{1:Neglec11}
B_{\boldsymbol k,\boldsymbol\ell,I^{(2)}_{\boldsymbol k},\Omega_{\boldsymbol k}}=
\mathcal O\left(\prod_{j=1}^mn_j^{-\frac 54+\eta}
\right)
 \, ,
\end{equation}
uniformly on $\boldsymbol k,\boldsymbol\ell$ as in Proposition~\ref{1:Point1},
where $I^{(2)}_{\boldsymbol k}$ is the set of $(\boldsymbol\theta,\boldsymbol\theta')\in\left[-\frac\pi d,\frac \pi d\right]^m\times[-\pi,\pi]^{M-m}$ such that for all
$j=1,...,m$, $|\theta_j|<n_j^{-\frac 12+{\eta}}$ and there exists
$j'=1,...,M$ such that $n_{j'}^{-\frac 12-{\eta}}<|\theta_{j'}|$.\\
\end{lem}
It remains to estimate the integral over  $I^{(3)}_{\boldsymbol k}$, the set of $(\boldsymbol\theta,\boldsymbol\theta')\in\left[-\frac\pi d,\frac \pi d\right]^m\times[-\pi,\pi]^{M-m}$ such that for all
$j=1,...,m$, $|\theta_j|<n_j^{-\frac 12-{\eta}}$.\\
We set $\mathcal J:=\{j=1,...,m\, : s_j=0\}=\{j(1),...,j(J)\}$.
\begin{lem}\label{LEM5}
Under the assumptions of Theorem~\ref{PAPATCL} with  $\sum_{a\in\mathbb Z} f(b+ad)=0$ for all $b\in\mathbb Z$.
Let $\mathcal J'\subset\mathcal J$, then
\begin{equation}\label{1:Neglec22}
\sum_{k'_j=0,...,d-1,\ \forall j\in\mathcal J'}
B_{\boldsymbol k+\boldsymbol k',\boldsymbol\ell,I^{(3)}_{\boldsymbol k},\Omega_{\boldsymbol k}}=
 \mathcal O\left(
\left(\prod_{j=1}^mn_j^{-\frac {3}4}\right)
\sum_{\mathcal J''\subset \mathcal J'\cup\mathcal (\mathcal J'+1)\, :\, \#\mathcal J''\ge \mathcal J/2}
\left(\prod_{j\in\mathcal J''}n_j^{-\frac 12+\eta}\right)\right)\, ,
\end{equation}
uniformly on $\boldsymbol k,\boldsymbol\ell$ as in Proposition~\ref{1:Point1},
and where we set $\boldsymbol k'=(k'_1,...,k'_m)$ with $k'_j=0$ if $j\not\in \mathcal J'$.

Moreover, if $s_j=1$ for all $j$ (and $\mathcal J'=\emptyset$), then,
\begin{align*}
B_{\boldsymbol k+\boldsymbol k',\boldsymbol \ell,I^{(3)}_{\boldsymbol k},\Omega_{\boldsymbol k}}&=\left(\frac d{\sqrt{2\pi}\sigma_\xi}\right)^{ m}\sum_{a_1,...,a_m\in\mathbb Z}\mathbf 1_{\{\forall i,\, a_i=k_i\alpha+d\mathbb Z\}}\mathbb E\left[ (\det D_{\boldsymbol k})^{-\frac 12}\mathbf 1_{\Omega_{\boldsymbol k}}\right]\\
&\ \ \ \ \prod_{j=1}^mf(a_j)\mathbb E\left[ f\left(a_j+Z_{\ell_j}\right) \right]+\mathcal O\left(n^{-(M+1)L\theta}\prod_{j=1}^mn_j^{-\frac {3}4}\right)\, ,
\end{align*}
uniformly on $\boldsymbol k,\boldsymbol\ell$ as above,
with
\[
\mathbb E\left[ (\det D_{\boldsymbol k})^{-\frac 12}\mathbf 1_{\Omega_{\boldsymbol k}}\right]=\mathcal O\left(\prod_{j=1}^mn_j^{-\frac 34}\right)\, ,
\]
uniformly on $\boldsymbol k$ as above,
and
\[
\mathbb E\left[ (\det D_{\boldsymbol k})^{-\frac 12}\mathbf 1_{\Omega_{\boldsymbol k}}\right]\sim n^{-\frac{3m}4}
\mathbb E\left[ \det \mathcal D_{t_1,...,t_m}^{-\frac 12}\right]\, .
\]
 as $k_j/n\rightarrow t_j$ and $n\rightarrow +\infty$.
\end{lem}
We can now complete the proof of Proposition ~\ref{1:Point1}.
The two first points of Proposition ~\ref{1:Point1} comes from the upper bounds provided by Lemmas \ref{LEM1}, \ref{LEM2}, \ref{LEM3}, \ref{LEM4} and \ref{LEM5}, with $E_{\boldsymbol k}:=\mathbb E\left[ (\det D_{\boldsymbol k})^{-\frac 12}\mathbf 1_{\Omega_{\boldsymbol k}}\right]$.
It remains to prove the last point of Proposition~\ref{1:Point1}. We assume that $s_j=1$ for all $j$ and that $k_j/n\rightarrow t_j$ and $n\rightarrow +\infty$.
Recall that $d_0=\min\{n\ge 1\, :\, n\xi_0\in d\mathbb Z\}=\min\{n\ge 1\, :\, n\alpha\in d\mathbb Z\}$. 
Observe that, for every $a_j\in\mathbb Z$ there is a unique $k'\in\{0,...,d_0-1\}$ such that $a_j\in (k_j+k'_j)\alpha+d\mathbb Z$. Thus
\begin{align*}
\sum_{k'_1,...,k'_m=0}^{d-1}
\sum_{\ell_1,...,\ell_m=0}^{n^{\frac{\kappa\theta\eta}{10 M}}}
2^{\#\{j:\ell_j>0\}}
\prod_{j=1}^m\mathcal A_{k_j+k'_j,\ell_j}
=
 \sum_{\ell_1,...,\ell_m=0}^{n^{\frac{\kappa\theta\eta}{10 M}}}
2^{\#\{j:\ell_j>0\}}
\sum_{a_j,\, b_j\in\mathbb Z}
\prod_{j=1}^mf(a_j)f(b_j)\mathbb P\left( Z_{\ell_j}=b_j-a_j\right)\, .
\end{align*}
Finally, due  to the last point of Lemma~\ref{LEM5} and to the next lemma, this quantity is equivalent to
\begin{align*}
\sum_{\ell_1,...,\ell_m \ge 0}
2^{\#\{j:\ell_j>0\}}
\sum_{a_j,\, b_j\in\mathbb Z}
\prod_{j=1}^mf(a_j)f(b_j)\mathbb P\left( Z_{\ell_j}=b_j-a_j\right)\, ,
\end{align*}
as $k_j/n\rightarrow t_j$ and $n\rightarrow +\infty$.
\end{proof}
\begin{lem}\label{1:Point2}
Under the assumptions\footnote{Our proof is valid in a more general context. The assumptions on $f$ and $S$ can be relaxed in
$\sum_{a\in\mathbb Z}|af(a)|<\infty$, $\sum_{a\in\mathbb Z}f(a)=0$,
 and $\Vert S_n\Vert_{L^{\frac 83}}=O(\sqrt{n})$.} of Theorem~\ref{PAPATCL},
\[
\sum_{\ell\ge 1}\left|\sum_{\ell'=0}^{d
-1}\sum_{a,b\in\mathbb Z}f(a)f(b)\mathbb P(Z_{\ell'+\ell d}=b-a)\right|<\infty
\]
\end{lem}
\begin{proof}
The proof of this lemma only uses estimates established in \cite{BFFN1}.
Since $\sum_{a,b}|f(a)f(b)|<\infty$ and using Lemma \ref{LEM1}, we observe that
\begin{align*}
&\left|\sum_{\ell'=0}^{d
-1}\sum_{a,b\in\mathbb Z}f(a)f(b)\mathbb P(Z_{\ell'+\ell d}=b-a)\right|\\
&=
\left|\sum_{\ell'=0}^{d
-1}\sum_{a\in\mathbb Z}\sum_{b\in a+(\ell d+\ell')\alpha+d\mathbb Z}f(a)f(b)\frac {d}{2\pi} 
   \int_{[-\frac{\pi}d,\frac\pi d]} e^{-it(b-a)}\mathbb E\left[\prod_{y\in\mathbb Z}\varphi_\xi\left(tN_{\ell d+\ell'}(y)\right)\right]
   \, dt\right|
\, .
\end{align*}
Moreover, due to \cite[Propositions 8,9,10]{BFFN1},
$\mathbb P(\Omega_k)=1-o(k^{-1-\eta_0})$
(due to \cite[Lemma 16]{BFFN2}), and due the fact that $|\varphi_\xi(tN_k(y))|\le e^{-\frac{\sigma^2_\xi (tN_k(y))^2}4}$
on $\Omega_{k}$ when $|t|\le k^{-\frac 34+\eta}$, we have
\begin{align*}
&\int_{[-\frac{\pi}d,\frac\pi d]} e^{-it(b-a)}\, \mathbb E\left[\prod_{y\in\mathbb Z}\varphi_\xi\left(tN_{\ell d+\ell'}(y)\right)\right]\\
&=\int_{|t|\le \ell^{-\frac 34+\eta}} e^{-it(b-a)}\,\mathbb E\left[\prod_{y\in\mathbb Z}\varphi_\xi\left(tN_{\ell d+\ell'}(y)\right)\mathbf 1_{\Omega_{\ell d+\ell'}}\right]
   \, dt+o(\ell^{-1-\eta_0})\\
&=\int_{|t|\le \ell^{-\frac 34+\eta}} e^{-it(b-a)}\,\mathbb E\left[\prod_{y\in\mathbb Z}\varphi_\xi\left(tN_{\ell d}(y)\right)\mathbf 1_{\Omega_{\ell d}}\right]
   \, dt+o(\ell^{-1-\eta_0})\, ,
\end{align*}
using also the fact that $\#\{y\in\mathbb Z\, :\, N_{\ell d}(y)\ne N_{\ell d+\ell'}(y)\}\le d$. 
It follows that
\begin{align*}
&\left|\sum_{\ell'=0}^{d-1}\sum_{a,b\in\mathbb Z}f(a)f(b)\mathbb P(Z_{\ell'+\ell d}=b-a)\right|\\
&=\left| \frac {d}{2\pi}   \int_{|t|\le \ell^{-\frac 34+\eta}}\sum_{a,b\in\mathbb Z}f(a)f(b) \left(e^{-it(b-a)}-1\right)\mathbb E\left[\prod_{y\in\mathbb Z}\varphi_\xi\left(tN_{\ell d}(y)\right)\mathbf 1_{\Omega_{\ell d}}\right]   \, dt\right|+o(\ell^{-1-\eta_0})\\
&\le \frac {d}{2\pi}   \int_{|t|\le \ell^{-\frac 34+\eta}}\sum_{a,b}\left|f(a)f(b) t(b-a) \right|\, \mathbb E\left[ e^{-\frac{\sigma^2_\xi t^2V_{\ell d}}4}\mathbf 1_{\Omega_{\ell d}}\right]   \, dt+o(\ell^{-1-\eta_0})\\
&\le C   \mathbb E\left[V_{\ell d}^{-1} \mathbf 1_{\Omega_{\ell d}} \right]+o(\ell^{-1-\eta_0})\, ,
\end{align*}
since $\sum_{a,b\in\mathbb Z}f(a)f(b)=0$, $\sum_{a\in\mathbb Z}|af(a)|<\infty$ and
using the change of variable $v=tV_{\ell d}^{\frac 12}$. Now, due to  \eqref{1:minVn},
$V_{\ell d}^{-1}\mathbf 1_{\Omega_{\ell d}}\le \ell^{-\frac 32-2\gamma}=\mathcal O(\ell^{-1-\eta_0})$ up to take $\eta_0$ small enough, which ends the proof of the lemma. 
\end{proof}

Theorem~\ref{PAPATCL} follows directly from the following corollary of Proposition~\ref{1:Point1} and Lemma~\ref{1:Point2}, since 
$\mathbb E[\mathcal N^{2N}]=\frac{(2N)!}{ N!2^N}$ and $\mathbb E[(\mathcal L_1(0))^{N}]=\int_{[0,1]^N}\frac{\mathbb E\left[\det \mathcal D_{t_1,...,t_N}^{-\frac 12}\right]}{(2\pi)^{\frac N2}}\, dt_1...dt_N$ (due to \cite{BFFN2}).
\begin{coro}\label{cvgcemoments}[A rewritting of Theorem~\ref{PAPATCL}]Under the assumptions of Theorem~\ref{PAPATCL},
\[
\mathbb E\left[\left(\sum_{k=1}^nf(Z_k)\right)^{2N+1}\right]=o\left(n^{\frac {2N+1}8}\right)\, ,
\]
and
\[
\mathbb E\left[\left(\sum_{k=1}^nf(Z_k)\right)^{2N}\right]= 
\frac{(2N)!}{ N!2^N}n^{\frac {2N}8}
\frac{\sigma^{2N}_f}{(2\pi \sigma_\xi^2)^{\frac N 2}}
\int_{[0,1]^N}
\mathbb E\left[\det \mathcal D_{t_1,...,t_N}^{-\frac 12}\right]\, dx_1...dx_N\, .
\]
\end{coro}
\begin{proof}
Since $f$ is bounded, it is enough to prove the result for $n=n'd$.
We start by writing
\begin{equation}\label{HHH1}
\mathbb E\left[\left(\sum_{k=1}^nf(Z_k)\right)^{M}\right]=\sum_{1\le m_1\le...\le m_{M}\le n} c_{\boldsymbol m}\mathbb  E\left[   \prod_{j=1}^{M}f(Z_{m_j})\right]
\, ,
\end{equation}
where $c_{\boldsymbol m}$ is the number of $(r_1,...,r_M)\in\{1,...,n\}^M$
such that $r_1,...,r_M$ and $m_1,...,m_M$ contain the same values  with same multiplicities. \\
Let $\theta_0\in\left(0,\frac 1{M+1}\right)$.
Given a sequence $1\le m_1\le ...\le m_M\le n$ with convention $m_0=0$, we consider $p\in\{0,...,M\}$ such that
no $m_j-m_{j-1}$ (for $j=1,...,M$) is in $(n^{L^{p+1}\theta_0},n^{L^p\theta_0}]$. Set $\theta=L^p\theta_0$. We write $k_1=m_1$ and, inductively, if $k_j=m_{u(j)}$,  we set $k_{j+1}=m_{u(j+1)}$ for the smallest integer $m_r $ such that $m_r>k_j+n^{\theta}$,
$s_j=u(j+1)-u(j)-1$ and then $\ell_{j,s}=m_{u(j)+s}$.\\
Thus
each $\boldsymbol m=(m_1,...,m_M)$ with $1\le m_1\le...\le m_M\le n$
can be represented by at least one 
\begin{equation}
(\boldsymbol k,\boldsymbol\ell)\in \bigcup_{p=0}^M\bigcup_{m=1}^M\bigcup_{s_j\ge 0\, :\, M=\sum_{j=1}^m(1+s_j)}F_{n,L^{p}\theta_0,m,s_1,...,s_m}\, ,
\end{equation}
with $F_{n,\theta,m,s_1,...,s_m}$ the set of $M$-uple $(\boldsymbol k,\boldsymbol\ell)$  of nonnegative integers with $\boldsymbol k=(k_j)_{j=1,...,m}$, $\boldsymbol \ell=(\ell_{j,s})_{j=1,...,m;s=1,...,s_j}$ such that, for all $j=1,...,m$, $k_{j}\ge k_{j-1}+n^\theta$ (with convention $k_0=0$) and, for all $j=1,...,m$ and all $s=1,...,s_j$, $0\le\ell_{j,s}\le n^{L\theta}$ and, with this representation,
\begin{equation}
\mathbb  E\left[   \prod_{j=1}^{M}f(Z_{m_j})\right]=\mathbb  E\left[   \prod_{j=1}^m \left(f(Z_{k_j})\prod_{s=1}^{s_j}f(Z_{k_j+\ell_{j,s}})\right)\right]\, .
\end{equation}
We first study separately the following sums
\[
\sum_{(m,\boldsymbol s)\in G_M}
\sum_{(\boldsymbol k,\boldsymbol \ell)\in F_{n,\theta,m,s_1,...,s_m}}c_{(\boldsymbol k,\boldsymbol \ell)}
\mathbb E\left[  \prod_{j=1}^m \left(f(Z_{k_j})\prod_{s=1}^{s_j}f(Z_{k_j+\ell_{j,s}})\right)\right]\, ,
\]
with $G_M$ the set of $(m,\boldsymbol s)$ with $m\in\{1,...,M\}$ and $\boldsymbol s=(s_1,...,s_m)$ with $s_j\ge 0$ for all $j=1,...,m$ and such that $M=\sum_{j=1}^m(s_j+1)$.

Let us fix for the moment $(m,\boldsymbol s)\in G_M$.
With the notation \eqref{CCCC1}, we wish to study
\begin{equation}\label{quantiteF}
\sum_{(\boldsymbol k,\boldsymbol \ell)\in F_{n,\theta,m,s_1,...,s_m}}
\mathbb E\left[  \prod_{j=1}^m \left(f(Z_{k_j})\prod_{s=1}^{s_j}f(Z_{k_j+\ell_{j,s}})\right)\right]=\sum_{(\boldsymbol k,\boldsymbol \ell)\in F_{n,\theta,m,s_1,...,s_m}}B_{\boldsymbol k,\boldsymbol \ell}\, .
\end{equation}
We say that $(\boldsymbol k,\boldsymbol\ell)$ and $(\boldsymbol k',\boldsymbol\ell')$ belong to a same block if 
\[
\forall r\not\in \mathcal J,\ k_r=k'_r,\quad \forall j\in\mathcal J,\quad \lfloor k_j/d\rfloor=\lfloor k'_j/d\rfloor,\quad \boldsymbol\ell=\boldsymbol \ell'\, .
\]
A block is an equivalence class for this equivalence relation.
We write $F'_{n,\theta,m,s_1,...,s_m}$ for the set of $(\boldsymbol k,\boldsymbol\ell)$ such that their block is contained in $F_{n,\theta,m,s_1,...,s_m}$.
We will see that the contribution of the sum over $F_{n,\theta,m,s_1,...,s_m}\setminus F'_{n,\theta,m,s_1,...,s_m}$ is neglectable in \eqref{quantiteF}.
Indeed, observe that if $(\boldsymbol k,\boldsymbol \ell)\in F_{n,\theta,m,s_1,...,s_m}\setminus F'_{n,\theta,m,s_1,...,s_m}$, then at least one of the following condition holds true
\begin{itemize}
\item[(a)] $\lfloor k_j/d\rfloor d-k_{j-1}<n^\theta\le (\lfloor k_j/d\rfloor+1)d-1-k_{j-1}$ if $j-1\not\in\mathcal J$ (or $\lfloor k_j/d\rfloor d-(\lfloor k_{j-1}/d\rfloor+1)d-d<n^\theta\le (\lfloor k_j/d\rfloor+1)d-1-\lfloor k_{j-1}/d\rfloor  d$
if $j-1\in\mathcal J$)
\item[(b)] 
$m\in\mathcal J$ and $d\lfloor k_m/d\rfloor+\max_s\ell_{m,s}<n\le 
d(\lfloor k_j/d\rfloor+1)+\max_s\ell_{m,s}
$
\end{itemize}
Let us fix $\mathcal J''\subset \mathcal J$.
Due to the first point of Lemma \ref{LEM5}, the contribution to \eqref{quantiteF} of blocks having a type (a) or (b) problem at indices $\mathcal J''$ is in
\begin{align*}
&\sum_{(k_j)_{j\not\in\mathcal J''},\boldsymbol \ell}\mathcal O\left(\left(\prod_{j=1}^mn_j^{-\frac 34}\right)\sum_{\mathcal J'\in\{1,...,m\}\, :\, \#\mathcal J'\ge\#(\mathcal J\setminus\mathcal J'')/2}\prod_{j\in\mathcal J'} n_j^{-\frac 12+\eta}
\right)
\\
&=\mathcal O\left(n^{LM\theta}\sum_{(n_j)_{j\not\in\mathcal J''}=n^\theta}^n\left(\prod_{j=1}^mn_j^{-\frac 34}\right)
\sum_{\mathcal J'\in\{1,...,m\}\, :\, \#\mathcal J'\ge\#(\mathcal J\setminus\mathcal J'')/2}\prod_{j\in\mathcal J'} n_j^{-\frac 12+\eta}\right)\, .
\end{align*}
The study of this quantity corresponds to \eqref{TOTO} up to replace $m$ par $m-\#\mathcal J''$ and to delete indices $\mathcal J''$, which thus will be in $o(n^{-\frac M8})$, as proved below.

Now, using the $d$-block structure of $F'_{n,\theta,m,s_1,...,s_m}$, It follows from \eqref{majoproba} that
\begin{align}
&
\sum_{(\boldsymbol k,\boldsymbol \ell)\in F'_{n,\theta,m,s_1,...,s_m}}
B_{(\boldsymbol k,\boldsymbol \ell)}
=\mathcal O\left( n^{LM\theta} \sum_{n_1,...,n_m=n^\theta}^n \left(\prod_{i=1}^mn_i^{-\frac 34}\right)\left(
\sum_{\mathcal J'\in\{1,...,m\}\, :\, \#\mathcal J'\ge(\#\mathcal J)/2}\prod_{j\in\mathcal J'} n_j^{-\frac 12+\eta}
\right)\right)\, .\label{TOTO}
\end{align}
The above quantity is in
\begin{align*}
&\mathcal O\left( n^{LM\theta} \sum_{\mathcal J'\, :\, \#\mathcal J'\ge\#(\mathcal J)/2}
\sum_{n_1,...,n_m= n^\theta}^n \left(\prod_{i=1}^mn_i^{-\frac 34}\right)\prod_{r\in\mathcal J'} n_{r}^{-\frac 12-\eta
}\right)\\
&=\mathcal O\left(\sum_{\mathcal J'\, :\, \#\mathcal J'\ge\#(\mathcal J)/2} n^{LM\theta+\frac 14\left(m-\lceil\#(\mathcal J)/2\rceil\right)-\left(\frac 1 4-\eta \right)\theta\lceil\#(\mathcal J)/2\rceil}\right)\\
&=\mathcal O\left( n^{LM\theta+\frac 14\left(m-\lceil\#\mathcal J/2\rceil\right)-\frac \theta 4\lceil\#\mathcal J/2\rceil+\theta J\gamma}\right)\, ,
\end{align*}
where we used the fact that $\sum_{r=1}^nr^{-\frac 34}=\mathcal O\left(n^{\frac 14}\right)$ and that  $\sum_{r\ge n^\theta}r^{-\frac 54
}=\mathcal O\left(n^{\frac\theta 4
}\right)$.
Observe moreover that $M=\sum_{j=1}^m(s_j+1)\ge 2(m-\#\mathcal J)+\#\mathcal J=2m- \#\mathcal J$, 
with equality if and only if $s_j\in\{0,1\}$ for all $j=1,...,m$.
It follows that
\begin{align*}
\sum_{(\boldsymbol k,\boldsymbol \ell)\in F_{n,\theta,m,s_1,...,s_m}}
\left|\mathbb E\left[  \prod_{j=1}^m \left(f(Z_{k_j})\prod_{s=1}^{s_j}f(Z_{k_j+\ell_{j,s}})\right)\right]\right|=\mathcal O\left( n^{LM\theta+\frac {M}8-\left[\frac{M-(2m-\#\mathcal J)}8
+\theta \left(\frac {\lceil \#\mathcal J/2\rceil}4- \#\mathcal J\eta\right)\right]}\right)
\end{align*}
In particular this is in $o(n^{\frac M8})$ as soon as 
$M>2m-\#\mathcal J$ or $\mathcal J\ne\emptyset$.

This ends the proof of the first point of Corollary~\ref{cvgcemoments} (since, when $M$ is odd, we cannot have $M=2m-\#\mathcal J$ and $\mathcal J=\emptyset $) and ensures that, for $M$ even,
\[
n^{-\frac M8}\mathbb E\left[\left(\sum_{k=1}^nf(Z_k)\right)^{M}\right]=n^{-\frac M8}\sum_{(\boldsymbol k,\boldsymbol \ell)\in\bigcup_{p=0}^M F_{n,L^p\theta_0,M/2,1,...,1}}c_{(\boldsymbol k,\boldsymbol \ell)}
\mathbb E\left[  \prod_{j=1}^m \left(f(Z_{k_j})f(Z_{k_j+\ell_{j,1}})\right)\right]
\, .
\]
Assume from now on that $\theta=\theta_0$ and that $M$ is even, $\mathcal J=\emptyset$ and $M=2m$, which means that $s_j=1$ for every $j=1,...,m$ and let us estimate the following quantity
\[
\mathcal E_{n,M,\theta}=\sum_{(\boldsymbol k,\boldsymbol \ell)\in F_{n,\theta,M/2,1,...,1}}c_{(\boldsymbol k,\boldsymbol \ell)}
\mathbb E\left[  \prod_{j=1}^m \left(f(Z_{k_j})f(Z_{k_j+\ell_{j,1}})\right)\right]\, .
\]
Note that, when $(\boldsymbol k,\boldsymbol \ell)\in F_{n,\theta,M/2,1,...,1}$, then $c_ {(\boldsymbol k,\boldsymbol \ell)}=\frac{(2m)!}{2^{\#\{ j:\ell_j=0\}}}$.
Using this and applying 
Proposition~\ref{1:Point1} combined with the dominated convergence theorem, we obtain that
\begin{align*}
&n^{-\frac m4}\mathcal E_{n,M,\theta}
=\frac{(2m)!}{2^m}n^{-\frac m4} \sum_{0\le k_1<...<k_m\le n:k_{i+1}-k_i>n^{\theta}}\sum_{\ell_1,...,\ell_m=0}^{n^{L\theta}}2^{\#\{ j:\ell_j>0\}} \mathbb E\left[   \prod_{j=1}^mf(Z_{k_j})f(Z_{k_j+\ell_{j}})\right]\\
&=\frac{(2m)!}{2^m}n^{-m} \sum_{0\le k_1<...<k_m\le n/d
:k_{i+1}-k_i>n^{\theta}}n^{\frac {3m}4}\sum_{k'_1,...,k'_m=0}^{d
-1}\sum_{\ell_1,...,\ell_m=0}^{n^{L\theta}}2^{\#\{ j:\ell_j>0\}} \mathbb E\left[   \prod_{j=1}^mf(Z_{d
k_j+k'_j})f(Z_{d
k_j+k'_j+\ell_{j}})\right]+o(1)\\
&=\frac{(2m)!}{2^m}
\int_{0\le t_1<...<t_m\le 1/d}
\frac{ d^m \sigma_f^{2m}\mathbb E\left[\det \mathcal D_{d_0t_1,...,d_0t_m}^{-\frac 12}\right]}{
(2\pi\sigma_\xi^2)^{\frac m2}}\, dt_1...dt_m+ o(1)\, .
\end{align*}
Therefore
\begin{align*}
\lim_{n\rightarrow +\infty}n^{-\frac m4}\mathcal E_{n,M,\theta}
&=\frac{(2m)!
}{2^m
}
\int_{0\le s_1<...<s_m\le 1}
\frac{
\sigma_f^{2m}\mathbb E\left[\det \mathcal D_{s_1,...,s_m}^{-\frac 12}\right]}{
(2\pi\sigma_\xi^2)^{\frac m2}}\, ds_1...ds_m\\
&=\frac{(2m)!
 \sigma_f^{2m}}{m!2^m
(2\pi\sigma_\xi^2)^{\frac m2}}
\int_{[0,1]^m}
\mathbb E\left[\det \mathcal D_{s_1,...,s_m}^{-\frac 12}\right]\, ds_1...ds_m\, .
\end{align*}
It remains now to prove that we can neglect the contribution of the
$(\boldsymbol k,\boldsymbol \ell)\in \bigcup_{p=1}^MF_{n,L^{p}\theta_0,M/2,1,...,1}\setminus F_{n,\theta_0,M/2,1,...,1}$. Fix some $p=1,...,M$. It follows from \eqref{majoproba} that
\begin{align*}
&n^{-\frac m4}\sum_{(\boldsymbol k,\boldsymbol \ell)\in F_{n,L^{p}\theta_0,M/2,1,...,1}\setminus F_{n,\theta_0,M/2,1,...,1}}c_{(\boldsymbol k,\boldsymbol \ell)}
\mathbb E\left[  \prod_{j=1}^m \left(f(Z_{k_j})f(Z_{k_j+\ell_{j,1}})\right)\right]\\
&=\mathcal O\left(   n^{-\frac m4} \sum_{n_1,...,n_{m-1}=n^{L^p\theta_0}}^n \left(\prod_{i=1}^{m-1}n_i^{-\frac 34}\right)
\sum_{n_m=1}^{n^{\theta_0}}n_m^{-\frac 34} n^{mL^{p+1}\theta_0}\right)
=\mathcal O\left(n^{-\frac 14+\frac{\theta_0} 4+mL^{p+1}\theta_0} 
\right)=o(1) \, .
\end{align*}
\end{proof}
The last part of Theorem \ref{PAPATCL} corresponds to the particular case $f=\delta_0-\delta_a$. In this case
\[
\sigma^2_f=\sigma^2_{0,a}=\sum_{k\in\mathbb Z}\left[2\mathbb P(Z_{|k|}=0)-\mathbb P(Z_{|k|}=a)-\mathbb P(Z_{|k|}=-a)\right]\, .
\]

\begin{appendix}
\section{Proofs of technical lemmas for Theorem~\ref{PAPATCL}}\label{AppendA}
Recall the context.
Let $M\ge 1$, $\theta\in(0,1)$, $\eta \in \left(0,\frac 1{100}\right)$, $L=\frac{\kappa\eta}{10M}$.
Recall that $n_j=k_j-k_{j-1}$ (with convention $k_0=0$).
Assume $n^\theta<n_j<n$
and let $\ell_{j,1},...,\ell_{j,s_j}=0,...,\lfloor n^{
L\theta}\rfloor$ 
with $\sum_{j=1}^m(1+s_j)=M$.

\begin{proof}[Proof of Lemma \ref{LEM1}]
We start by writing
\begin{align}
p_{\mathbf{k},\boldsymbol{\ell}}(\boldsymbol{a},\boldsymbol{b})
&\ \ \ =\frac {1}{(2\pi)^M} 
  \int_{[-\pi,\pi]^{M}} e^{-i\sum_{j=1}^m[(a_j-a_{j-1})\theta_j+\sum_{s=1}^{s_j}(b_{j,s}-a_j)\theta'_{j,s}]}\varphi_{\boldsymbol{k},\boldsymbol{\ell}}(\boldsymbol\theta,\boldsymbol\theta')
   \, d(\boldsymbol\theta,\boldsymbol \theta')\, .\label{1:EqFourier0}
\end{align}
But, due to the definition of $d$, for any $u,v\in\mathbb Z$, $\varphi_\xi(u+\frac{2\pi v}d)=(\varphi_\xi(\frac{2\pi}d))^v\varphi_\xi(u)$ and so, for any $\boldsymbol u\in\mathbb R^M$ and $\boldsymbol v\in\mathbb Z^M$,
\begin{align*}
&\varphi_{\boldsymbol{k},\boldsymbol{\ell}}(\boldsymbol u+\frac{2\pi}d\boldsymbol v)= \EE \left[\prod_{y\in \ZZ}  \varphi_{\xi}\left( \sum_{j=1}^m\left[\left(u_j+\frac {2\pi v_j}d\right)N'_{j}(y)+\sum_{s=1}^{s_j}\left(u_{j,s}+\frac {2\pi v_{j,s}}d\right)N'_{j,s}(y)\right]\right)\right]\\
&= \EE \left[\prod_{y\in \ZZ} \left(\varphi_\xi\left(\frac{2\pi}d\right)\right)^{\sum_{j=1}^m [v_jN'_j(y)+\sum_{s=1}^{s_j}v_{j,s}N'_{j,s}(y)]} \varphi_{\xi}\left( \sum_{j=1}^m\left[u_jN'_{j}(y)+\sum_{s=1}^{s_j}u_{j,s}N'_{j,s}(y)\right]\right)\right]\\
&=\left(\varphi_\xi\left(\frac{2\pi }d\right)\right)^{\sum_{j=1}^m\left[v_jn_j+\sum_{s=1}^{s_j}\ell_{j,s}v_{j,s}\right]}\varphi_{\boldsymbol{k},\boldsymbol{\ell}}(\boldsymbol u)\, .
\end{align*}
 and so
\begin{align*}
\nonumber&p_{\mathbf{k},\boldsymbol{\ell}}(\boldsymbol{a},\boldsymbol{b})
=\frac {1}{(2\pi)^M} 
  \int_{[-\frac \pi d,\frac \pi d]^{m}\times[-\pi,\pi]^{M-m}} \sum_{r_j
=0}^{d-1}e^{-i\sum_{j=1}^m[(a_j-a_{j-1})(\theta_j+\frac{2\pi r_j}{d})+\sum_{s=1}^{s_j}(b_{j,s}-a_j)\theta'_{j,s}
}\\
&\quad\quad\quad\quad\quad\quad\quad\quad\quad\quad\quad\quad \left(\varphi_\xi\left(\frac{2\pi }d\right)\right)^{\sum_{j=1}^m
r_jn_j
}\varphi_{\boldsymbol{k},\boldsymbol{\ell}}(\boldsymbol\theta,\boldsymbol\theta')
   \, d(\boldsymbol\theta,\boldsymbol \theta')\, .
\end{align*}
Moreover, for any $a\in\mathbb Z$, then
$\sum_{r=0}^{d-1} e^{-\frac{2ia\pi r}d}\left(\varphi_\xi\left(\frac{2\pi }d\right)\right)^{vr}=0$ except if $e^{-\frac{2ia\pi}d}\left(\varphi_\xi\left(\frac{2\pi }d\right)\right)^{v}=1$ (i.e. if $v\alpha-a\in d\mathbb Z$) and then this sum is equal to $d$.
This ends the proof of Lemma~\ref{LEM1}.
\end{proof}

\begin{proof}[Proof of Lemma~\ref{LEM2}]
Due to \cite[Lemma 16]{BFFN2}, for any $\gamma>0$, 
satisfies $\PP(\Omega^{(j)}_{\boldsymbol k}) =1 - o(n_j^{-p})$ for any $p>1$ and so, since $n_j>n^\theta$, it follows that
for all $p>1$, $\PP(\Omega^{(j)}_{\boldsymbol k})=1-o(n^{-p})$.
Moreover, since $\theta'\in(0,\frac\theta 4)$, due to \cite[Lemma 21]{BFFN2},
\[
\forall p>1,\quad \mathbb P\left(\det D_{n_1,...,n_m}< n^{-\theta'}\prod_{i=1}^mn_i^{\frac 32}\right)=o(n^{-p})\, ,
\]
uniformly on $\boldsymbol k$ as above.
\end{proof}
\begin{proof}[Proof of Lemma~\ref{LEM3}]
Recall that $\mathcal F=\left\{y\in\mathbb Z\, :\, \forall (j,s),\ N'_{j,s}(y)=0\right\}$.
Due to \eqref{majovarphi}, 
Lemma~\ref{LEM3} follows from the following estimate
\begin{equation}
\exists c>0,\quad
 \int_{\{ \exists j,n_j^{-\frac 12+{\eta}}<|\theta_j|\}}
 \EE \cro{\prod_{y\in\mathcal F} \left|\varphi_{\xi}\left(\sum_{j=1}^m\theta_jN'_{j}(y)\right)\right|
   {\bf 1}_{\Omega_{\boldsymbol k}} } \, d\boldsymbol\theta
= o\left(e^{- n^c}\right) \, ,
\end{equation}
uniformly on $\boldsymbol k,\boldsymbol\ell$ as in Proposition~\ref{1:Point1}.
To this end, we follow and slightly adapt the proof of \cite[Proposition 10]{BFFN1} as explained below.
Observe that, up to conditioning with respect to $(S_{k+1}-S_{k})_{k\not\in\{k_{j-1},...,k_j-1\}}$, this will be a consequence of
\begin{equation}\label{1:Neglecb}
\forall j=1,...,m,\quad \forall u\in\mathbb R,\quad \int_{n_j^{-\frac 12+{\eta}}<|\theta|<\frac\pi d}
 \EE \cro{\prod_{y\in \mathcal F} \left|\varphi_{\xi}\left(u+\theta N'_{j}(y)\right)\right|
   {\bf 1}_{\Omega_{\boldsymbol k}} } \, d\theta
= o\left(e^{- n^c}\right) \, ,
\end{equation}
uniformly on $k_j,\ell_{j,s}$ as above.
Recall that $\#(\mathbb Z\setminus \mathcal F)\le\sum_{j=1}^m\sum_{s=1}^{s_j}\ell_{j,s}\le M n^{L\theta}$.
As in \cite[after Lemma 16]{BFFN1}, we observe that, for $n$ large enough,
\begin{eqnarray}
\label{majodist}
\prod_{y\in\mathcal F} \va{\varphi_{\xi}(u+\theta  N'_j(y)) }
\leq \exp \pare{- \frac{\sigma^2_\xi}4 n^{-\frac 1 2 + 4\gamma}\# \acc{y\ :\
 d\left(u+\theta N'_j(y),\frac{2\pi}d{\mathbb Z}\right) \geq n^{-\frac{1}{4}+2\gamma}
}} \, ,
\end{eqnarray}
and that
\begin{eqnarray}
\label{Gt}
d\left(u+\theta  N'_j(y),\frac{2\pi{\mathbb Z}}d\right)\ge
n^{-\frac{1}{4}+2\gamma}\ \Longleftrightarrow\ \frac u\theta +N'_j(y)\in  \I:=\bigcup_{k\in{\mathbb Z}}I_{k}\, ,
\end{eqnarray}
where, for all $k\in \ZZ$,
$$I_{k}:=\left[
\frac{2k\pi}{d\theta}+\frac{n^{-\frac{1}{4}+2\gamma}}
{\theta},
\frac{2(k+1)\pi}{d\theta}-\frac{n^{-\frac{1}{4}+2\gamma}}
{\theta}\right]. $$
In particular ${\mathbb R}\setminus\I=\bigcup_{k\in{\mathbb Z}}J_{k}$, where for all $k\in \ZZ$,
$$J_{k}:=\left(
\frac{2k\pi}{d\theta}-\frac{n^{-\frac{1}{4}+2\gamma}}
{\theta},
\frac{2 k\pi}{d \theta}+\frac{n^{-\frac{1}{4}+2\gamma}}{\theta}\right). $$
Let $N_\pm$ be two positive integers such that
${\mathbb P}(X_1=N_+){\mathbb P}(X_1=-N_-)>0$.
Let ${\mathcal C}^\pm=(\mathcal C^\pm_k)_{k=1,...,T}\in\mathbb Z^{T}$ 
with $T=N_++N_-$ and $\mathcal C^+_k= N_+$ for $k\le N_-$ and $\mathcal C^+_k=- N_-$ otherwise, and symetrically
and $\mathcal C^-_k=- N_-$ for $k\le N_+$ and $\mathcal C^-_k= N_+$
otherwise.
It has been proved in \cite{BFFN1} (see Lemma 15 therein combined with the estimate ${\mathbb P}(\D_n) =1-o(e^{-cn})$ in Section 2.8 therein) 
that, for $n$ large enough,
\begin{equation}\label{controlEj}
\mathbb P(\Omega_{\boldsymbol k}\setminus\mathcal E_{j})=o(e^{-cn_j})\, ,
\end{equation}
with
\[
\mathcal E_{j}=\left\{ \#\{y\in {\mathbb Z}\ :\
C_j(y)
\ge  n_j^{\frac{1}{2}-2\gamma}\}  \ge 3N_+N_- n_j^{\frac 1 2- 2\gamma}\right\}\, ,
\]
and
where, for any $y\in {\mathbb Z}$,
$$
C_{j}(y):=\#\left\{k=0,\dots,\left\lfloor \frac {n_j} T\right\rfloor-1 \ :\ 
S_{k_{j-1}+kT}-S_{k_{j-1}}=y \textrm{ and }(X_{k_{j-1}+kT},\dots,X_{k_{j-1}+(k+1)T-1}) = \C^\pm\right\}\, .
$$
Now, on  $\E_j$, we define
$Y_i$  for $i=1,\dots,\left\lfloor{n_j^{\frac 1 2-2\gamma}}
   \right\rfloor$, by 
$$Y_1:=\min\left\{y\in \mathbb Z
\ :\
C_j(y)\ge n_j^{\frac 1 2 -2 \gamma} \right\},$$
and
$$Y_{i+1}:= \min\left\{y\ge Y_i+3N_-N_+\ :\ C_j(y)\ge n_j^{\frac 1 2 -2 \gamma} \right\}\quad \textrm{for }i\ge 1 .$$
For every $i=1,\dots,\left\lfloor{n_j^{\frac 1 2-2\gamma}}\right\rfloor$, 
let
$t_i^{1},\dots,t_i^{\left\lfloor n_j^{\frac 1 2 -2 \gamma}\right\rfloor}$ be the 
$\left\lfloor n_j^{\frac 1 2 -2 \gamma}\right\rfloor$
first times (which are multiples of $T$) when a peak of the form $\mathcal C^{\pm}$ is based on the site $Y_i$. 
We also define $N_j^{0}(Y_i+N_+N_-)$ as the
 number of visits of $(S_{k_{j-1}+k}-S_{k_{j-1}})_{k\ge 0}$ before time $n_j$ to $Y_i+N_+N_-$, which do not occur during the time intervals $[t_i^u,t_i^u+T]$, 
for $u\le \left\lfloor n_j^{\frac 1 2 -2 \gamma}\right\rfloor$.   
We proved in \cite[Lemma 16]{BFFN1} that, for any $H\ge 0$,
\begin{eqnarray}
\label{H}
{\mathbb P}\left(\frac u\theta+N'_j(Y_i+N_+N_-)\in\I \left| \E_n,\ N_j^0(Y_i+N_+N_-)=H\right.\right)={\mathbb P}\left(H+\frac u\theta+b_j\in\I\right),
\end{eqnarray}
where $b_j$ is a random variable with
binomial distribution
${\mathcal B}\left(\left\lfloor {n_j^{\frac 12 -2\gamma}}\right\rfloor; \frac 12\right)$
and finally we proved in \cite[Lemmas 17 and 18]{BFFN1} (see in particular the last formula in the proof of Lemma 17) that
\[\forall H'\in\mathbb R,\quad \PP \pare{H'+b_n \in \I} 
\geq \frac 1 3 \, .
\]
Thus, conditionally to  $(S_{k+1}-S_{k})_{k\not\in\{k_{j-1},...,k_j-1\}}$,
$\E_j$ and $((N_j^0(Y_i+N_+N_-),i\ge 1)$, the events 
$\{\frac u\theta+N_{j}(Y_i+N_+N_-)\in\I\}$, $i\ge 1$, are independent of each other, and all happen
with probability at least $1/3$. We conclude that
\begin{equation}\label{controlBinom}
{\mathbb P}\left(\E_j\cap\left\{
\ \#\left\{i\ :\ \frac u\theta+N'_{j}(Y_i+N_+N_-)\in{\mathcal I}\right\}\le
\frac{n_j^{\frac 1 2-2\gamma}} 4\right\}\right)
\le {\mathbb P}\left(B_j\le \frac{n_j^{\frac 1 2-2\gamma}}4\right) 
=o(e^{-c'n_j})\, ,
\end{equation}
where $B_j$ has binomial distribution
${\mathcal B}\left(\left\lfloor {n_j^{\frac 12 -2\gamma}}\right\rfloor;
\frac 13\right)$.

\noindent But if $\#\{y\in\mathbb Z\ :\ N'_j(z)\in\I\}\ge
n_j^{\frac 1 2-2\gamma}/4 
$, then, by \eqref{majodist} and \eqref{Gt} there exists a constant $c''>0$, such that,
for any $n$ large enough,
$$\prod_{y\in\mathcal F}\vert\varphi_\xi(u+\theta N'_j(y))\vert
\le \exp\left(-c'' n_j^{\frac 1 2-2\gamma}n_j^{-\frac 1 2+4\gamma}\right),
$$
since $\#(\mathbb Z\setminus\mathcal F)\ll n_j^{\frac 1 2-2\gamma}/4 $.
This, combined with \eqref{controlEj} and \eqref{controlBinom}, ends the  proof of \eqref{1:Neglecb} and so of Lemma~\ref{LEM3}.
\end{proof}

\begin{proof}[Proof of Lemma~\ref{LEM4}]
We have to estimate
$B_{\boldsymbol k,\boldsymbol\ell,I^{(2)}_{\boldsymbol k},\Omega_{\boldsymbol k}}$
uniformly on $\boldsymbol k,\boldsymbol\ell$ as in Proposition~\ref{1:Point1},
where $I^{(2)}_{\boldsymbol k}=V_{\boldsymbol k}\times[-\pi,\pi]^{M-m}$
and where $V_{\boldsymbol k}$ is the set of $\boldsymbol\theta\in \mathbb R^m$ such that for all
$j=1,...,m$, $|\theta_j|<n_j^{-\frac 12+{\eta}}$ and such that there exists
some $j_0=1,...,m$ satisfying $n_{j_0}^{-\frac 12-{\eta}}<|\theta_{j_0}|$.
Let $\varepsilon_0>0$ be such that
\begin{equation}\label{phixi}
\forall u\in[-\varepsilon_0,\varepsilon_0],\quad
|\varphi_\xi(u)|\le e^{-\frac{\sigma_\xi^2u^2}4}\, .
\end{equation}
We define the events $H_{\boldsymbol k}=\Omega_{\boldsymbol k}\cap \{\forall y\in\mathbb Z,\, |\sum_{j=1}^m\theta_jN'_{j}(y)|\le\varepsilon_0/2\}$ and
\[
H'_{\boldsymbol k}:=\left\{ 
\#\left\{y\in\mathbb Z\, :\, \left|\sum_{j=1}^m\theta_jN'_{j}(y)\right|\in\left[\frac{\varepsilon_0}4,\frac{\varepsilon_0}2\right]\right\}> n^{\frac 14}\right\}\, .
\]
Due to \cite[Lemma 21 and last formula of p. 2446]{BFFN2},
\begin{equation*}
\exists c'>0,\quad 
\mathbb P\left(\Omega_{\boldsymbol k}\setminus (H_{\boldsymbol k}\cup H'_{\boldsymbol k})\right)
=\mathcal O\left(\prod_{j=1}^mn_j^{-\frac 34}\right)\, ,
\end{equation*}
uniformly on $\boldsymbol k$ as above and 
uniformly on $\boldsymbol\theta\in V_{\boldsymbol k}$.
Thus,
\begin{equation}\label{1:Neglec00bb}
B_{\boldsymbol k,\boldsymbol\ell,I^{(2)}_{\boldsymbol k},\Omega_{\boldsymbol k}\setminus (H_{\boldsymbol k}\cup H'_{\boldsymbol k})}=\mathcal O\left(\prod_{j=1}^mn_j^{-\frac 54+\eta}
\right)\, ,
\end{equation}
where we used the fact that $\int_{V_{\boldsymbol k}}\, d\boldsymbol\theta\le \prod_{j=1}^mn_j^{-\frac 12+\eta}$.
Moreover, for $n$ large enough, it follows from the definition of $H'_k$, from
\eqref{Z-F} and \eqref{phixi} that
\begin{align}\label{1:Neglec00cc}
B_{\boldsymbol k,\boldsymbol\ell,I^{(2)}_{\boldsymbol k},\Omega_{\boldsymbol k}\cap H'_{\boldsymbol k})}=\mathcal O\left(
\int_{V_{\boldsymbol k}}\mathbb E\left[ \prod_{y\in\mathcal F}\left|\varphi_{\xi}\left(\sum_{j=1}^m\theta_jN'_{j}(y)\right)\right|\mathbf 1_{\Omega_{\boldsymbol k}\cap H'_{\boldsymbol k}}\right]\, d\boldsymbol\theta\right)\le e^{-\frac{\sigma_\xi^2\varepsilon_0^2n^{\frac 14}}{64}}\, .
\end{align}
Finally, it remains to estimate $B_{\boldsymbol k,\boldsymbol\ell,I^{(2)}_{\boldsymbol k},\Omega_{\boldsymbol k}\cap H_{\boldsymbol k}}$. To this end we write
\begin{align}
\nonumber\int_{V_{\boldsymbol k}}\mathbb E&
\left[
\prod_{y\in\mathcal F} \left|\varphi_{\xi}\left(\sum_{j=1}^m\theta_jN'_{j}(y)\right)\right|\mathbf 1_{\Omega_{\boldsymbol k}\cap H_{\boldsymbol k}}\right]\, d\boldsymbol\theta\\
\nonumber&\le \int_{V_{\boldsymbol k}} 
\mathbb E\left[e^{-\frac{\sigma^2_\xi}4\sum_{y\in\mathcal F}\left(\sum_{j=1}^m\theta_jN'_{j}(y)\right)^2}\mathbf 1_{\Omega_{\boldsymbol k}
}\right]\, d\boldsymbol\theta\\
\nonumber&\le \int_{V''_{\boldsymbol k}} 
\left(\prod_{j=1}^mn_j^{-\frac 34}\right)
\mathbb E\left[e^{-\frac{\sigma^2_\xi}4\sum_{y\in\mathcal F}\left(\sum_{j=1}^m\theta''_jn_j^{-\frac 34}N'_{j}(y)\right)^2}\mathbf 1_{\Omega_{\boldsymbol k}
}\right]\, d\boldsymbol\theta''\\
&\le
\left(\prod_{j=1}^mn_j^{-\frac 34}\right)
\mathbb E\left[\int_{(\widetilde D'_{\boldsymbol k})^{\frac 12}V''_{\boldsymbol k}}(\det\widetilde D'_{\boldsymbol k})^{-\frac 12}e^{-\frac{\sigma^2_\xi|\boldsymbol v|^2}4}
\mathbf 1_{\Omega_{\boldsymbol k}
}\, d\boldsymbol v\right]\, ,\label{Etape31}
\end{align}
with the successive changes of variable $\theta''_j=n_j^{\frac 34}\theta_j$ and 
$\boldsymbol v=(\widetilde D'_{\boldsymbol k})^{\frac 12}\boldsymbol\theta''$, with
\[
\widetilde D'_{\boldsymbol k}=\left((n_in_j)^{-\frac 34}\sum_{y\in\mathcal F}N'_{i}(y)N'_{j}(y)\right)_{i,j}\quad\mbox{and}\quad V''_{\boldsymbol k}=Diag(n_i^{\frac 34})V_{\boldsymbol k}\, .
\]
Note that $V''_{\boldsymbol k}$ is the set of $(\theta''_1,...,\theta''_m)$ such that $|\theta''_j|\le n_j^{\frac 14+\eta}$ and such that there exists $j_0=1,...,m$ such that $|\theta''_{j_0}|\ge n_{j_0}^{\frac 14-\eta}$.

Let us prove that, in the above formula, we can approximate the determinant of $\widetilde D'_{\boldsymbol k}$ by the one of $\widetilde D_{\boldsymbol k}:=\left((n_in_j)^{-\frac 34}\sum_{y\in\mathbb Z}N'_{i}(y)N'_{j}(y)\right)_{i,j}$.
To this end, writing $\Sigma_m$ for the set of permutations of the set $\{1,...,m\}$
and $\varkappa(\sigma)$ for the signature of $\sigma\in \Sigma_m$, we observe that, on $\Omega_{\boldsymbol k}$,
\begin{align*}
&\left|\det\widetilde D'_{\boldsymbol k}-\det \widetilde D_{\boldsymbol k}\right|\\
&=\left(\prod_{j=1}^mn_j^{-\frac 32}\right)\left|\sum_{\sigma\in\Sigma_m}(-1)^{\varkappa(\sigma)}\prod_{j=1}^m
\left(\sum_{y\in\mathcal F}N'_{j}(y)N'_{\sigma(j)}(y)\right)
-\sum_{\sigma\in\Sigma_m}(-1)^{\varkappa(\sigma)}\prod_{j=1}^m
\left(\sum_{y\in\mathbb Z}N'_{j}(y)N'_{\sigma(j)}(y)\right)\right|\\
&\le \left(\prod_{j=1}^mn_j^{-\frac 32}\right)\sum_{\sigma\in\Sigma_m}\sum_{j=1}^m
\sum_{z\in\mathbb Z\setminus\mathcal F}N'_j(z)N'_{\sigma(j)}(z)
\prod_{j'\ne j}\left(\sum_{y\in\mathbb Z}N'_{j'}(y)N'_{\sigma(j')}(y)\right)
\\
&\le \left(\prod_{j=1}^mn_j^{-\frac 32}\right) \sum_{\sigma\in\Sigma_m}\sum_{j=1}^m\#(\mathbb Z\setminus\mathcal F)n_j^{\frac {1+2\gamma}2}n_{\sigma(j)}^{\frac {1+2\gamma}2}
\prod_{j'\ne j}\sqrt{V_{j'}V_{\sigma(j')}}\, ,
\end{align*}
where we used the Cauchy-Schwarz inequality together with the notations and estimates given after Lemma~\ref{LEM2}. Using \eqref{defOmegak} and \eqref{Z-F}, it follows that, on $\Omega_{\boldsymbol k}$,
\begin{align*}
\left|\det \widetilde D'_{\boldsymbol k}-\det \widetilde D_{\boldsymbol k}\right|
&\ll   \left(\prod_{j=1}^mn_j^{-\frac 32}\right)n^{L\theta} \sum_{j=1}^mn_j^{\frac {1+2\gamma}2}n_{\sigma(j)}^{\frac {1+2\gamma}2}
\prod_{j'\ne j}n_{j'}^{\frac {3(1+2\gamma)}4}n_{\sigma(j')}^{\frac {3(1+2\gamma)}4}\\
&\ll \frac 12n^{m\gamma-\frac \theta 2+L\theta}
\ll n^{-\theta'
-(M+1)L\theta
} 
 \le \frac {
n^{-(M+1)L\theta}}2\det \widetilde D_{\boldsymbol k}\, ,
\end{align*}
since $\theta'
\le \frac \theta 2-
2ML\theta
<\frac \theta 2  -m\gamma-ML\theta$ and where we used the fact that $\det \widetilde D_{\boldsymbol k}=\det D_{\boldsymbol k}\prod_{j=1}^mn_j^{-\frac 32}$ together with the definition of $\Omega_{\boldsymbol k}$.
Therefore, on $\Omega_{\boldsymbol k}$, $\det \widetilde D'_{\boldsymbol k}\ge \frac 12 \det \widetilde D_{\boldsymbol k}$.
Thus, due to \eqref{Etape31},
\begin{align}
\nonumber\int_{V_{\boldsymbol k}}\mathbb E&\left[
\prod_{y\in\mathcal F} \left|\varphi_{\xi}\left(\sum_{j=1}^m\theta_jN'_{j}(y)\right)\right|\mathbf 1_{\Omega_{\boldsymbol k}\cap H_{\boldsymbol k}}\right]\, d\boldsymbol\theta\\
\nonumber&\le \mathcal O\left(
\left(\prod_{j=1}^mn_j^{-\frac 34}\right)
 \mathbb E\left[\int_{(\widetilde D'_{\boldsymbol k})^{\frac 12}V''_{\boldsymbol k}}(\det \widetilde D'_{\boldsymbol k})^{-\frac 12}\mathbf 1_{\Omega_{\boldsymbol k}}e^{-\frac{\sigma^2_\xi|\boldsymbol v|^2}4}\, d\boldsymbol v\right]\right)\\
&=\mathcal O\left(
\left(\prod_{j=1}^mn_j^{-\frac 34}\right)
 \mathbb E\left[(\det \widetilde D_{\boldsymbol k})^{-\frac 12}\mathbf 1_{\Omega_{\boldsymbol k}}\int_{(\widetilde D'_{\boldsymbol k})^{\frac 12}V''_{\boldsymbol k}}e^{-\frac{\sigma^2_\xi|\boldsymbol v|^2}4}\, d\boldsymbol v\right]\right)\, ,
\label{1:Neglec00dd}
\end{align}
By definition of $V''_{\boldsymbol k}$,
for any $\boldsymbol v\in(\widetilde D'_{\boldsymbol k})^{\frac 12}V''_{\boldsymbol k}$, $|\boldsymbol v|_2\ge (\widetilde \lambda'_{\boldsymbol k})^{\frac 12}n^{(\frac 14-\eta)\theta}$, where $\widetilde\lambda'_{\boldsymbol k}$ is the smallest eigenvalue of $\widetilde D'_{\boldsymbol k}$. Since all the eigenvalues of $\tilde D'_{\boldsymbol k}$
are nonnegative ($\widetilde D'_{\boldsymbol k}$ being symmetric and nonnegative), it follows that all the  eigenvalues of $\widetilde D'_{\boldsymbol k}$ are smaller than $trace(\widetilde D'_{\boldsymbol k})\le \sum_{j=1}^m \frac{V_j}{n_j^{\frac 32}}\le mn^{3\gamma}$ (on $\Omega_{\boldsymbol k}$).
Thus, on $\Omega_{\boldsymbol k}$,
\begin{align}\label{minolambda}
(\widetilde \lambda'_{\boldsymbol k})^{\frac 12} n^{(\frac 14-\eta)\theta}&\ge \frac{\det (\widetilde D'_{\boldsymbol k})^{\frac 12}}{(m^{\frac 12}n^{\frac {3\gamma}2})^{m-1}}  n^{(\frac 14-\eta)\theta}
\ge 
 \frac{n^{(\frac 14-\eta)\theta-\frac{\theta'}2-\frac{3\gamma(m-1)}2}}{2m^{\frac {m-1}2}}   \gg  n^{\frac{\theta}{16}}\, ,
\end{align}
since $\eta \theta$, $\frac{\theta'}2$, and $\frac{3\gamma(m-1)}2$ are all strictly smaller $\frac \theta {16}$.
Hence
\begin{align*}
&\mathbb E\left[( \det \widetilde D_{\boldsymbol k})^{-\frac 12}\mathbf 1_{\Omega_{\boldsymbol k}}\int_{(\widetilde D'_{\boldsymbol k})^{\frac 12}V''_{\boldsymbol k}}e^{-\frac{\sigma^2_\xi|\boldsymbol v|^2}4}\, d\boldsymbol v\right]\\
&\le \mathbb E\left[(\det \widetilde D_{\boldsymbol k})^{-\frac 12}\mathbf 1_{\Omega_{\boldsymbol k}}\right]\int_{|\boldsymbol v|_2>n^{\frac \theta{16}}}e^{-\frac{\sigma^2_\xi|\boldsymbol v|^2}4}\, d\boldsymbol v
=\mathcal O\left(  n^{-p}  \right)\, ,
\end{align*}
for any $p>0$.
This combined with  \eqref{1:Neglec00bb}, \eqref{1:Neglec00cc} and \eqref{1:Neglec00dd} ends the proof of the lemma.
It will be worthwhile to note that the previous estimate also holds true when $\widetilde \lambda'_{\boldsymbol k}$ is replaced by the smallest eigenvalue $\widetilde \lambda_{\boldsymbol k}$ of $\widetilde D_{\boldsymbol k}$.
\end{proof}

Before proving Lemma~\ref{LEM5}, we state a useful coupling lemma
allowing us to replace $\det D_{\boldsymbol k}$
by a copy independent of $(N'_{j,s})_{j,s}$.\\
Up to enlarging the probability space if necessary, we consider $X'=(X'_k)_{k\ge 1}$ an independent copy of the increments $X=(X_k)_{k\ge 0}$ of the random walk $S$.
We then define the random walk $S''$ as follows:
$S''_m=\sum_{k=1}^mX''_k$ with $X''_k=X_k$ if $k_{j-1}+\ell_{j-1}\le k<k_j$ and $X''_k=X'_k$ if $k_{j}\le k<k_j+\ell_{j}$, with $\ell_j:=\max_{s=1,...,s_j}\ell_{j,s}$.
We define $\Omega''_{\boldsymbol k}$, $N''_j$ and $D''_{\boldsymbol k}$ for the space 
as we have defined $\Omega_{\boldsymbol k}$, $N'_j$, $D_{\boldsymbol k}$ (up to replace $S$ by $S''$).
\begin{lem}\label{lemD-D''}
There exists $\Omega'_{\boldsymbol k}\subset \Omega_{\boldsymbol k}\cap \Omega''_{\boldsymbol k}$ such that 
\begin{equation}\label{POmega'}
\forall p>0,\quad\mathbb P\left((\Omega_{\boldsymbol k}\cap\Omega''_{\boldsymbol k})\setminus \Omega'_{\boldsymbol k}\right)=\mathcal O(n^{-p})\, 
\end{equation}
and such that, on $\Omega'_{\boldsymbol k}$, 
\begin{align*}
\left|(\det D_{\boldsymbol k})^{-\frac 12}-(\det D''_{\boldsymbol k})^{-\frac 12}\right|
&\le n^{-\frac \theta 8-L\theta}
(\det D_{\boldsymbol k}^{-\frac 32}+(\det D''_{\boldsymbol k})^{-\frac 32})\, .
\end{align*}
Moreover
\begin{equation}\label{D-D''}
\mathbb E\left[((\det D_{\boldsymbol k})^{-\frac 12}-(\det D''_{\boldsymbol k})^{-\frac 12})\mathbf 1_{\Omega'_{\boldsymbol k}}\right]
\le   n^{-\frac \theta 8-L\theta}\prod_{j=1}^m n_j^{-\frac 94}\, .
\end{equation}
\end{lem}

\begin{proof}[Proof of Lemma~\ref{lemD-D''}]
Observe that
\[S''_{k_j}-S_{k_j}= h_j=\sum_{j'<j}\left(S'_{k_{j'}+\ell_{j'}}-S'_{k_{j'}}-(S_{k_{j'}+\ell_{j'}}-S_{k_{j'}})\right)\] and, on $\Omega_{\boldsymbol k}\cap\Omega''_{\boldsymbol k}$,
\[|N'_{j}(z)-N''_{j}(z)|= |N'_{j}(z)-N'_{j}(z+h_j)|
\le 
n_j^{\frac 14+\frac\gamma 2}|h_j|^{\frac 12},\]
for all $z\in\mathbb Z\setminus\bigcup_{m=k_{j-1}}^{k_{j-1}+\ell_j}\{S_{m}, S''_m\}$.\\
We will prove that $\det D_{\boldsymbol k}$ is close enough to $\det D''_{\boldsymbol k}=\det(\sum_{y\in\mathbb Z}N_i''(y)N_j''(y))$.
Due to the Markov inequality,
\[
\forall p>0,\quad
\mathbb P\left( |S_{\ell_j}|>h\right)\le\mathcal O\left(\frac{\ell_j^{\frac p2}}{h^p}\right)=\mathcal O\left(n^{-\gamma'p}\right)\, ,
\]
where we set $h=n^{\gamma'+\frac{\kappa\theta\eta}{20M}}\ge n^{\gamma'}\ell_j^{\frac 12}$. Thus we set 
\[
\Omega'_{\boldsymbol k}:=\Omega_{\boldsymbol k}\cap\Omega''_{\boldsymbol k}\cap\{\forall j=1,...,m,\ |h_j|\le h\}
\]
and we observe that
$\mathbb P\left((\Omega_{\boldsymbol k}\cap\Omega''_{\boldsymbol k})\setminus \Omega'_{\boldsymbol k}\right)=\mathcal O(n^{-p})$ for all $p>0$.
Moreover, on $\Omega'_{\boldsymbol k}$,
\[|N'_{j}(z)-N''_{j}(z)|\le 2\ell_j+n_j^{\frac 14+\frac\gamma 2}h^{\frac 12}\le 3n_j^{\frac 14+\frac\gamma 2}n^{\frac {\gamma'}2+\frac{\kappa\theta\eta}{40M}}\, .\]
Moreover 
\[
V''_j:=\sum_{y\in\mathbb Z}(N''_j(y))^2\le\sum_{y\in\mathbb Z}(N'_j(y))^2+2\ell_j ^3\le n_j^{\frac 32+3\gamma}\, .
\]
This allows us to observe that, on $\Omega'_{\boldsymbol k}$,
\begin{align*}
&\left|\det D_{\boldsymbol k}-\det D''_{\boldsymbol k}\right|\\
&=\left|\sum_{\sigma\in\Sigma_m}(-1)^{\varkappa(\sigma)}\prod_{j=1}^m
\left(\sum_{y\in\mathbb Z}N'_{j}(y)(N'_{\sigma(j)}(y)\right)
-\sum_{\sigma\in\Sigma_m}(-1)^{\varkappa(\sigma)}\prod_{j=1}^m
\left(\sum_{y\in\mathbb Z}N''_{j}(y)N''_{\sigma(j)}(y)\right)\right|\\
&\le \sum_{\sigma\in\Sigma_m}\sum_{j=1}^m
\sum_{z\in\mathbb Z}[N'_j(z)N'_{\sigma(j)}(z)-N''_{j}(z)N''_{\sigma(j)}(z)]
\prod_{j'\ne j}\max\left(\sum_{y\in\mathbb Z}N'_{j'}(y)N'_{\sigma(j')}(y),\sum_{y\in\mathbb Z}N''_{j'}(y)N''_{\sigma(j')}(y)\right)
\\
&\le    3
n^{\frac {\gamma'}2+\frac{\kappa\theta\eta}{40M}}
 \sum_{\sigma\in\Sigma_m}
\sum_{j=1}^m 
\left[
V_j^{\frac 12} n_{\sigma(j)}^{\frac 12+ \gamma }+(V''_{\sigma(j)})^{\frac 12} n_{j}^{\frac 12+ \gamma }\right]
\prod_{j'\ne j}\max\left(V_{j'}V_{\sigma(j')},V''_{j'}V''_{\sigma(j')}\right)^{\frac 12}
\\
&\le    3
n^{\frac {\gamma'}2+\frac{\kappa\theta\eta}{40M}}
m!\prod_{j'=1}^m n_{j'}^{\frac 32+3\gamma}\sum_{j=1}^m n_j^{-\frac 14-\frac\gamma 2}\ll  \prod_{j'=1}^mn_{j'} ^{\frac 32}\sum_{j=1}^mn_j^{-\frac 18}n^{-L\theta}
\, ,
\end{align*}
since $L\theta+3m\gamma-\frac \theta 4+\frac {\gamma'
} 2<-\frac \theta 8-L\theta$, and so, on $\Omega'_{\boldsymbol k}$,
\begin{align*}
\left|(\det D_{\boldsymbol k})^{-\frac 12}-(\det D''_{\boldsymbol k})^{-\frac 12}\right|
&\le  n^{-\frac \theta 8-L\theta}
(\det D_{\boldsymbol k}^{-\frac 32}+(\det D''_{\boldsymbol k})^{-\frac 32})\, .
\end{align*}
We conclude thanks to  \cite[Lemma 21]{BFFN2} which ensures that
$\mathbb E\left[(\det D_{\boldsymbol k})^{-\frac 32} \mathbf 1_{\Omega_{\boldsymbol k}}\right]=\mathcal O\left(\prod_{j=1}^mn_j^{-\frac 94}\right)$.
\end{proof}
The proof of Lemma~\ref{LEM5} will also use the following result. Recall that
we set $\mathcal J=\{j=1,...,m\, : s_j=0\}$ and that $\mathcal J'\subset\mathcal J$.

\begin{lem}\label{LEMsum0}
Under the assumptions of Lemma~\ref{LEM5},
\[
\sum_{k'_{j}=0,...,d-1,\ \forall j\in\mathcal J'}
B_{\boldsymbol k+\boldsymbol k',\boldsymbol\ell,I_{\boldsymbol k}^{(3)},\Omega_{\boldsymbol k}}
=
\frac {d^m}{(2\pi)^M}\int_{I_{\boldsymbol k}^{(3)}} \mathbb E\left[ \mathbf 1_{\Omega_{\boldsymbol k}}
 F
(\boldsymbol\theta,\boldsymbol\theta')
 G(\boldsymbol\theta,\boldsymbol\theta')
\right]
  \, d(\boldsymbol\theta
,\boldsymbol \theta')\, ,
\]
with $\boldsymbol k'\in \mathbb Z^m$ such that $k'_{j}=0$ for all $j\not\in\mathcal J'$, and with
\begin{align*}
G(\boldsymbol\theta,\boldsymbol\theta')&:=\prod_{j\not\in \mathcal J'}\left(
\sum_{a_j\in \alpha k_j+d
\mathbb Z
}\sum_{b_{j,s},...,b_{j,s_j}\in\mathbb Z}f(a_j)\left(\prod_{v=1}^{s_j}f(b_{j,v})\right)
e^{i a_j(\theta_{j+1}-\theta_j)-i\sum_{s=1}^{s_j}(b_{j,s}-a_j)\theta'_{j,s}}\right)\\
&\quad \quad \times\prod_{y \in\mathbb Z\setminus \mathcal S'} \varphi_{\xi}\left( \sum_{j=1}^m\left(\theta_jN'_{j}(y)+\sum_{s=1}^{s_j}\theta'_{j,s}N'_{j,s}(y))\right)\right)\, ,
\end{align*}
with
$\mathcal S'=\bigcup_{j\in\mathcal J'}\mathcal S'_j$, $\mathcal S'_j:=\{S_{k_j},...,S_{k_j+d-1}\}$, so that 
$\{S_{k_j},...,S_{k_j+d-1}\}$ and,  uniformly on $\boldsymbol k,\boldsymbol \ell$ and on $\Omega_{\boldsymbol k}$,
$ F
(\boldsymbol\theta,\boldsymbol\theta')
=\mathcal O\left(\sum_{\mathcal J''\subset\mathcal J'}\prod_{j\in\mathcal J'\setminus\mathcal J''}(|\theta_j|+|\theta_{j+1}|)\mathbf 1_{\bigcap_{j\in\mathcal J''}\mathcal B_j})\right)$
with $\mathcal B_j=\mathbf 1_{\mathcal S'_j\cap\bigcup_{j'\in\mathcal J''\setminus\{ j\}}\mathcal S'_{j'}\ne\emptyset}$.\\
If $\sum_{a\in\mathbb Z}f(b+d\mathbb Z)=0$ for all $b\in\mathbb Z$ (true if $d=1$), then $ F
(\boldsymbol\theta,\boldsymbol\theta')
=\mathcal O\left(\prod_{j\in\mathcal J'}(|\theta_j|+|\theta_{j+1}|)\right)$ (with convention $\theta_{m+1}=0$).
\end{lem}
\begin{proof}
We start by writing
\begin{equation}\label{CCCC3}
\sum_{k'_{j}=0,...,d-1,\ \forall j\in\mathcal J'}
B_{\boldsymbol k+\boldsymbol k',\boldsymbol\ell,I_{\boldsymbol k}^{(3)},\Omega_{\boldsymbol k}}
=\frac {d^m}{(2\pi)^M}\int_{I_{\boldsymbol k}^{(3)}} \mathbb E\left[ \mathbf 1_{\Omega_{\boldsymbol k}}
F
(\boldsymbol\theta,\boldsymbol\theta')
 G(\boldsymbol\theta,\boldsymbol\theta')
\right]
  \, d(\boldsymbol\theta
,\boldsymbol \theta')\, ,
\end{equation}
where we set
\begin{align*}
F
(\boldsymbol\theta,\boldsymbol\theta'):=\sum_{k'_j=0,...,d-1,\, \forall j\in\mathcal J'}&
\prod_{j\in\mathcal J'}\left(\sum_{a_j\in (k_j+k'_j)\alpha+d\mathbb Z}\left(f(a_j)e^{-ia_j(\theta_j-\theta_{j+1})} \right)\right)\\
&\times\prod_{y\in \mathcal S'} \varphi_{\xi}\left( \sum_{r=1}^m\left(\theta_r\widetilde N'_{r,\boldsymbol k'}(y)+\sum_{s=1}^{s_r}\theta'_{r,s} \widetilde N'_{r,s
}(y))\right)\right)\, ,
\end{align*}
with
\[
\widetilde N'_{r,\boldsymbol k'}(y)=\#\{u=k_{r-1}+k'_{r-1},...,k_r+k'_{r}-1\, :\, S_{u}=y\}\, .
\]
If we had $\sum_{a\in u+d\mathbb Z}f(a)=0$ for all $u\in\mathbb Z$, the proof of Lemma~\ref{LEMsum0} 
will be ended by noticing that
\[
\sum_{a_j\in (k_j+k'_j)\alpha+d\mathbb Z}\left(f(a_j)e^{-ia_j(\theta_j-\theta_{j+1})} \right)
=\sum_{a_j\in (k_j+k'_j)\alpha+d\mathbb Z}\left(f(a_j)\left(e^{-ia_j(\theta_j-\theta_{j+1})}-1\right) \right)\, ,
\]
which is in $\mathcal O\left(|\theta_j|+|\theta_{j+1}|\right)$ since $\sum_{a\in\mathbb Z}|af(a)|<\infty$.
Since we just assume here that $\sum_{a\in\mathbb Z}f(a)=0$, we need a more delicate approach.
We rewrite $F$ as follows
\begin{align*}
F
(\boldsymbol\theta,\boldsymbol\theta'):=\sum_{k'_j=0,...,d-1,\, \forall j\in\mathcal J'}
\sum_{a_j\in (k_j+k'_j)\alpha+d\mathbb Z\, \forall j\in\mathcal J'}
\left(\prod_{j\in\mathcal J'}H_{j,k'_j}(\theta_j-\theta_{j+1})\right)
\Psi(\boldsymbol k')
\end{align*}
with
\[
H_{j,k'_j}(\theta):=\sum_{a_j\in (k_j+k'_j)\alpha+d\mathbb Z}\left(f(a_j)e^{-ia_j\theta
} \right)\, ,
\]
\begin{align*}
\Psi(\boldsymbol k')=
\prod_{y\in \mathcal S'}
 \varphi_{\xi}\left( \sum_{r=1}^m\left(\theta_r\widetilde N'_{r,\boldsymbol k'}(y)+\sum_{s=1}^{s_r}\theta'_{r,s} \widetilde N'_{r,s}(y))\right)\right)\, ,
\end{align*}
and with
$
\widetilde N'_{r,\boldsymbol k'}(y)=\#\{u=k_{r-1}+k'_{r-1},...,k_r+k'_{r}-1\, :\, S_{u}=y\}$. 
Note that $\widetilde N'_{r,\boldsymbol k'}(y)=N'_{r}(y)$ except maybe if 
$r\in\mathcal J'$ and $y\in\mathcal S'_r$ or if $r-1\in\mathcal J'$ and $y\in\mathcal S'_{r-1}$.
We order the elements of $\mathcal J'$ as follows: $j'_1<...<j'_{J'}$
and write
\begin{align*}
F(\boldsymbol\theta,\boldsymbol\theta'
)=F_{0}(\boldsymbol\theta,\boldsymbol\theta')+F_{1}(\boldsymbol\theta,\boldsymbol\theta')
\end{align*}
with
\begin{align*}
F_{1}(\boldsymbol\theta,\boldsymbol\theta')=\sum_{k_{j'_1}=0}^{d-1}
H_{j'_1,k'_{j'_1}}(0)&\sum_{k_{j'_2},...,k_{j'_{J'}}=0}^{d-1}\left(\prod_{j\in\mathcal J'\setminus\{1\}}H_{j,k'_j}(\theta_j-\theta_{j+1})\right)
\Psi(\boldsymbol k')
\end{align*}
and
\begin{align*}
F_{0}(\boldsymbol\theta,\boldsymbol\theta')=&
\sum_{k_{j'_1}=0}^{d-1}
\left(H_{j'_1,k'_{j'_1}}(\theta_{j'_1}-\theta_{j'_1+1})-H_{j'_1,k'_{j'_1}}(0)\right)
\sum_{k_{j'_2},...,k_{j'_{J'}}=0}^{d-1} \left(\prod_{j\in\mathcal J'\setminus\{1\}}H_{j,k'_j}(\theta_j-\theta_{j+1})\right)
\Psi(\boldsymbol k')\, .
\end{align*}
Note that $H_{j'_1,k'_{j'_1}}(\theta_{j'_1}-\theta_{j'_1+1})-H_{j'_1,k'_{j'_1}}(0)$ is in $\mathcal O(|\theta_j|+|\theta_{j+1}|)$.
Since $\sum_af(a)=0$, $F_{1}$ satisfies
\begin{align*}
&F_{1}(\boldsymbol\theta,\boldsymbol\theta'
)=\sum_{k_{j'_1}=0}^{d-1}H_{j'_1,k'_{j'_1}}(0)
\sum_{k_{j'_2},...,k_{j'_{J'}}=0}^{d-1}\left(\prod_{j\in\mathcal J'\setminus\{j'_1\}}H_{j,k'_j}(\theta_j-\theta_{j+1})\right)
\Delta_{j'_1} \Psi(\boldsymbol k')
\end{align*}
with $\Delta_j\phi(\boldsymbol k')=\phi(\boldsymbol k')-\phi(\boldsymbol k'_j)$,
where $\boldsymbol k'_j\in\mathbb N^{m}$ is such that $(\boldsymbol k'_j)_i=k'_i$ for $i\ne  j$, and $(\boldsymbol k'_j)_j=0$.
Proceding iteratively on $\mathcal J'$, we obtain
\begin{equation}\label{FormulaF}
F(\boldsymbol\theta,\boldsymbol\theta')=\sum_{\epsilon_1,...,\epsilon_{J'}\in\{0,1\}}
F_{\epsilon_1,...,\epsilon_{J'}}(\boldsymbol\theta,\boldsymbol\theta')\, ,
\end{equation}
with
\[
F_{\epsilon_1,...,\epsilon_{J'}}(\boldsymbol\theta,\boldsymbol\theta')
=    \left(\prod_{j':\epsilon_{j'}=0} \left(H_{j'_1,k'_{j'_1}}(\theta_{j'_1}-\theta_{j'_1+1})-H_{j'_1,k'_{j'_1}}(0)\right)\right)\left(\prod_{j:\epsilon_j=1}H_{j,k'_{j}}(0)\right)  \Delta^{\epsilon_{J'}}_{j'_{J'}}\cdots\Delta^{\epsilon_1}_{j'_1}\Psi(\boldsymbol k')\, ,
\]
with convention $\Delta_{j'}^0=Id$.
The first part will be easily dominated by $\mathcal O\left(\prod_{j':\epsilon_{j'}=0}(|\theta_{j'}|+|\theta_{j'+1}|)\right)$. Let us study the second part of the formula
exploiting the fact that $\sum_{a\in\mathbb Z}f(a)=0$. The difficulty here is that $\boldsymbol k'$ appears both in $\left(\prod_{j:\epsilon_j=1}H_{j,k'_{j}}(0)\right) $ 	and in $\Delta_{...}\psi(\boldsymbol k')$.
The value of $(\epsilon_1,...,\epsilon_{J'})$ being fixed, we consider the set $\mathcal J''$ of the $j'\in \mathcal J'$ such that $\epsilon_{j'}=1$.
Observe that, if $\mathcal S'_{j'}\cap\mathcal S'_{j}=\emptyset$, then
\[
\Delta_{j'}\Delta_j\Psi(\boldsymbol k')=\left(\Delta_{j'}\Psi_{\mathcal S'\setminus\mathcal S'_j}(\boldsymbol k'_j)\right)\left(\Delta_j\Psi_{\mathcal S'_j}(\widehat{\boldsymbol k}'_j)\right)
\]
with
\[
\Psi_{\mathcal S_0}(\boldsymbol k')=
\prod_{y\in \mathcal S_0}
 \varphi_{\xi}\left( \sum_{r=1}^m\left(\theta_r\widetilde N'_{r,\boldsymbol k'}(y)+\sum_{s=1}^{s_r}\theta'_{r,s} \widetilde N'_{r,s}(y))\right)\right)\, ,
\]
and where we set $\widehat{\boldsymbol k}'_j$ for the vector of $\mathbb Z^m$ with $j$-th coordinate equal to $k'_j$, all the other coordinates being null.
Let $\mathcal J''_0$ be the set of $j\in\mathcal J''$ such that $\mathcal S'_j\cap\bigcup_{j''\in\mathcal S''\setminus\{j\}}\mathcal S'_{j''}=\emptyset$.
Then
\begin{align*}
&\sum_{k_{j}=0,...,d-1,\ \forall j\in\mathcal J''_0}\left(\prod_{j\in\mathcal J''_0}H_{j,k'_{j}}(0)\right)\Delta^{\epsilon_{J'}}_{j'_{J'}}\cdots\Delta^{\epsilon_1}_{j'_1}\Psi(\boldsymbol k')\\
&=\prod_{j\in\mathcal J''_0}\left(\sum_{k'_j=0}^{d-1}H_{j,k'_j}(0)\Delta_j\Psi_{\mathcal S'_j}\left(\widehat{\boldsymbol k}'_j\right)\right)
\Delta_{\mathcal J''\setminus\mathcal J''_0}\Psi(\boldsymbol k'_{\mathcal S''_0})
\end{align*}
with $\boldsymbol k'_{\mathcal S''_0}\in\mathbb N^{m}$ such that $(\boldsymbol k'_j)_i=k'_i$ for $i\not\in  \mathcal S''_0$, the other coordinates being null, the notation $\Delta_{\mathcal J''\setminus \mathcal J''_0}$
standing for the composition of all the operators $\Delta_j$ for $j\in\mathcal  J''\setminus \mathcal J''_0$.
We conclude by using \eqref{FormulaF} and by noticing that 
\[
\left(\prod_{j'\in\mathcal S\setminus\mathcal S''} \left(H_{j'_1,k'_{j'_1}}(\theta_{j'_1}-\theta_{j'_1+1})-H_{j'_1,k'_{j'_1}}(0)\right)\right)=\mathcal O\left(\prod_{j'\in\mathcal S\setminus\mathcal S''}(|\theta_{j'}|+|\theta_{j'+1}|)\right)\, ,
\]
\[
\prod_{j\in\mathcal J''_0}\left(\sum_{k'_j=0}^{d-1}H_{j,k'_j}(0)\Delta_j\Psi_{\mathcal S'_j}\left(\widehat{\boldsymbol k}'_j\right)\right)=
\mathcal O\left(\prod_{j'\in\mathcal S\setminus\mathcal S''}(|\theta_{j'}|+|\theta_{j'+1}|)\right)
\]
and that 
\[
j\in\mathcal J''\setminus\mathcal J''_0\quad\Longrightarrow\quad 
\mathcal S'_j\cap\bigcup_{j'\in\mathcal S''_0\setminus\{ j\}}\mathcal S'_{j'}\, .
\]
\end{proof}
The following lemma will be useful to estimate the term $F$ appearing in Lemma~\ref{LEMsum0}. It is not needed when  $\sum_{a\in\mathbb Z}f(b+ad)=0$ for all $b\in\mathbb Z$.
\begin{lem}\label{LEMsum0bis}
For any $\mathcal J'\subset\mathcal J$,
\[
\mathbb P\left(\Omega_{\boldsymbol k}\cap\bigcap_{j\in\mathcal J'}\mathcal B_j\right)=\mathcal O\left( \sum_{\mathcal J''\subset \mathcal J'\setminus\{\min\mathcal J'\},\ \mathcal J''\ge \#\mathcal J'/2}n^{J\gamma}\prod_{j\in\mathcal J''}(k_{j}-k^-_{j})^{-\frac 12}\right)\, ,
\]
where $k^-_{j}=\max\{k_s\le k_j,\ s\in\mathcal J'\}$.
\end{lem}
\begin{proof}
It is enough to study
\[
\mathbb P\left(\Omega_{\boldsymbol k}\cap\bigcap_{j\in\mathcal J'}\left\{ S_{k_j+r_j}=S_{k_{m(j)}+s_j}\right\}\right)
\]
for any $m(j)\in\mathcal J'\setminus\{j\}$, $r_j,s_j\in\{0,...,d-1\}$.
This probability is dominated by
\[
\mathbb P\left(\Omega_{\boldsymbol k}\cap\left\{\forall j\in\mathcal J',\ |S_{k_j}-S_{k_{m(j)}}|\le n^v\right\}\right)+o(n^{-p})\, ,
\]
for all $p,v>0$.
We partition the set $\mathcal J'$ by the equivalence relation generated by the relation $j\sim m_{j}$. We write $\mathcal R(j)$ for the class of $j$ and $\mathcal R$ for the set of these equivalence classes.
Observe that the number of equivalent classes is at most $\lfloor \#\mathcal J'/2\rfloor$.
We order the set $\mathcal J'$ in $j'_1<...<j'_{J'}$.
We wish to estimate
\[
\sum_{A_r,\ r\in \mathcal R}\mathbb P\left(\Omega_{\boldsymbol k},\, \forall i=1,...,J'-1,\ S_{k_{j'_{i+1}}-k_{j'_{i}}}=A_{\mathcal R(j'_{i+1})}-A_{\mathcal R(j''_i)}+\mathcal O(n^v)\right)\, ,
\]
where the sum is over $(A_r)_{r\in\mathcal R}\in\mathbb Z^{\mathcal R}$ such that $A_{\mathcal R(1)}=0$, $A_{\mathcal R(j'_{i+1})}-A_{\mathcal R(j'_{i})}=\mathcal O((k_{j'_{i+1}}-k_{j'_i})^{\frac 12+\frac\gamma 2})$.
Due to the local limit theorem and the independence of the increments of $S$, the above probability is in
\[
\sum_{A_r,\ r\in \mathcal R} \prod_{i=1}^{J'-1}n^v\left(
O\left((k_{j'_{i+1}}-k_{j'_{i}})^{-\frac 12}\right)
\right)\, .
\]
Now let us control the cardinal of the admissible $(A_r,\ r\in\mathcal R)$.
To this end, consider the set $\overline {\mathcal J'}$
of the smallest representants of $\mathcal R$.
Then the above quantity is smaller than
\[
n^{J'(v+\frac\gamma 2)}\prod_{j\in\mathcal J'\setminus\overline{\mathcal J'}}(k_{j}-k_{j}^-)^{-\frac 12}\, .
\]

\end{proof}

\begin{proof}[Proof of Lemma~\ref{LEM5}]
All the estimates below are uniformly in $\boldsymbol k$.
For the first estimate, we have to estimate the following integral
\begin{align}
\nonumber \int_{ \forall j,|\theta_j|<n_j^{-\frac 12-{\eta}}}&
\left(\prod_{j\not\in\mathcal J'}\left(\sum_{a_j\in\alpha k_j+d\mathbb Z}f(a_j)e^{ia_j(\theta_{j+1}-\theta_j)}
\prod_{s=1}^{s_j}\sum_{b_{j,s}\in\mathbb Z}\left(f(b_{j,s}) e^{-i(b_{j,s}-a_j)\theta'_{j,s}}\right)\right)\right)\\
&\quad\times \mathbb E\left[ \mathbf 1_{\Omega_{\boldsymbol k}}
 F(\boldsymbol\theta,\boldsymbol\theta')
\prod_{y\in\mathbb Z\setminus\mathcal S'}\mathfrak A_y
\right]
 \, d\boldsymbol\theta\, ,\label{1:Neglec}
\end{align}
where we set
\[
\mathfrak A_y:=\varphi_\xi\left(\sum_{j=1}^m\left(\theta_jN'_{j}(y)+\sum_{s=1}^{s_j}\theta'_{j,s}N'_{j,s}(y)\right)\right)\, .
\]
Let us study
\begin{equation}\label{1:defEkl}
E_{\boldsymbol k,\boldsymbol\ell}(\boldsymbol\theta,\boldsymbol\theta'):=\prod_{y\in\mathbb Z\setminus \mathcal S'}\mathfrak A_y-\prod_{y\in\mathbb Z\setminus \mathcal S'}\mathfrak B_y\, ,
\end{equation}
with

\[
\mathfrak B_y
:=\exp\left(-\frac{\sigma^2_\xi}2\left(\sum_{j=1}^m\theta_j N'_{j} (y)\right)^2\right)\varphi_\xi\left(\sum_{j=1}^m\sum_{s=1}^{s_j}\theta'_{j,s}N'_{j,s}(y)\right)\, . 
\]
But, on $\Omega_{\boldsymbol k}$, if $|\theta_j|\le n_j^{-\frac 12-\eta}$ for all $j=1,...,m$, and so
\begin{eqnarray}
\label{1:tnnz}
\forall y\in\mathbb Z,\quad \left|\sum_{j=1}^m\theta_j N'_{j}(y)\right|\le \sum_{j=1}^m|\theta_j| N_{j}^*  \le \sum_{j=1}^mn_j^{-\frac\eta 2}\le m n^{-\frac{\theta\eta}2}<\varepsilon_0\, ,
\end{eqnarray} 
as soon as $n$ is large enough (uniformly on $n_j\in[n^\theta,n]$).
Thus
$|E_{\boldsymbol k,\boldsymbol \ell}(\boldsymbol\theta,\boldsymbol\theta')|$ is dominated by
\begin{align}
\nonumber 
 &  \sum_{y\in\mathbb Z
}\left\vert \mathfrak A_y-\mathfrak B_y\right|
e^{-\frac{\sigma^2_\xi}4 \sum_{z\in\mathcal F\setminus(\mathcal S'\cup\{ y\})}\left(\sum_{j=1}^m\theta_jN'_{j}(z)\right)^2}
\end{align}
for $n$ large enough.
Now, on $\Omega_{\boldsymbol k}$, according to 
\eqref{Z-F},
\begin{equation}\label{minoexpo}
\forall y\in\mathbb Z,\quad \sum_{z\in\mathcal F\setminus(\mathcal S'\cup\{ y\})}\left(\sum_{j=1}^m\theta_jN'_{j}(z)\right)^2\ge \sum_{z'\in\mathbb Z}\left(\sum_{j=1}^m\theta_jN'_{j}(z')\right)^2-M(d+n^{\frac{\eta\theta}{10M}})n^{-\theta\eta}
\, .
\end{equation}
It follows that
\begin{align}
|E_{\boldsymbol k,\boldsymbol \ell}(\boldsymbol\theta,\boldsymbol\theta')|
&\le (A+B)\exp\left(-\frac{\sigma^2_\xi}4  \sum_{z'\in\mathbb Z}\left(\sum_{j=1}^m\theta_jN'_{j}(z')\right)^2-\mathcal O\left(n^{-\frac{9\theta}{10\eta}} \right)\right)\, ,\label{2:majoEkl}
\end{align}
with 
\begin{align}\label{majoquantiteA}
A&:=  \sum_{y\in\mathcal F\setminus\mathcal S'}\left\vert \varphi_\xi\left(\sum_{j=1}^m\theta_jN'_{j}(y)\right)-
    e^{-\frac{\sigma_\xi^2}2\left(\sum_{j=1}^m\theta_j N'_{j}(y)\right)^2 }\right\vert
\le   \sum_{y\in\mathbb Z}\left|\sum_{j=1}^m\theta_j N'_{j}(y)\right|^2\, 
C' n^{-\frac {\kappa\theta\eta} 2}
\end{align}
where we used the fact that
$$\left\vert \varphi_\xi(u)-\exp\left(-\frac{ \sigma^2_\xi|u|^2}2\right)\right\vert\le
|u|^{2+ \kappa}
 \quad \textrm{for all } u\in \RR,$$ 
since $\xi$ admits a moment of order $2+\kappa$ and
there exists $C_0>0$ such that
\begin{align}
\nonumber B&:=  \sum_{y\in\mathbb Z\setminus\mathcal F}\left\vert\varphi_\xi\left(\sum_{j=1}^m\left(\theta_jN'_{j}(y)+\sum_{s=1}^{s_j}\theta'_{j,s}N'_{j,s}(y)\right)\right)-
e^{-\frac{\sigma_\xi^2}2\left(\sum_{j=1}^m\theta_j N'_{j}(y)\right)^2  }\varphi_\xi\left(\sum_{j=1}^m\sum_{s=1}^{s_j}\theta'_{j,s}N'_{j,s}(y)\right)\right\vert\\
&\le C_0  \sum_{y\in\mathbb Z\setminus\mathcal F} \left|\sum_{j=1}^m\theta_jN'_{j}(y)\right|\le   C_0\sum_{j=1}^m\sum_{s=1}^{s_j}\ell_{j,s} n^{-\frac {\theta\eta} 2}=\mathcal O\left(n^{\frac{\theta\eta }{10M} -\frac{\theta\eta} 2}\right)=\mathcal O\left(n^{-\frac{\theta\eta} 4}\right)\, ,\label{majoquantiteB}
\end{align}
since $\varphi_\xi$ and $u\mapsto e^{-\frac{u^2}2}$ are Lipschitz continuous.
Recall that it has been proved in \cite[Lemma 21]{BFFN2} that
\begin{equation}
\label{majoespDet}
\mathbb E
\left[|\det D_{\boldsymbol k}|^{-\frac 12} \mathbf 1_{\Omega_{\boldsymbol k}}\right]=
\mathcal O\left(\prod_{j=1}^mn_j^{-\frac 34}\right)\, ,
\end{equation}
uniformly on $\boldsymbol k$.

Combining Lemmas \ref{LEMsum0} and \eqref{LEMsum0bis}, 
\eqref{2:majoEkl}, \eqref{majoquantiteA}, 
\eqref{majoquantiteB},
\eqref{majoespDet}
and using the change of variable
 $\boldsymbol v=(D_{\boldsymbol k})^{\frac 12}\boldsymbol\theta$
with
$D_{\boldsymbol k}=\left(\sum_{y\in\mathbb Z}N'_{i}(y)N'_{j}(y)\right)_{i,j}$,
it follows that there exists $C_1>0$ such that
\begin{align}
\nonumber&\int_{\forall j,|\theta_j| \le n_j^{-\frac 12-\eta}}\mathbb{E}\left[\left|F(\boldsymbol\theta,\boldsymbol\theta') E_{\boldsymbol k,\boldsymbol \ell}(\boldsymbol\theta,\boldsymbol\theta')\right|{\bf 1}_{\Omega_{\boldsymbol k}}\right]\, d(\boldsymbol \theta,\boldsymbol\theta')\\
\nonumber&\leq C_1
\int_{\mathbb{R}^m} 
\left( n^{-\frac {\kappa\theta\eta} 2}
\vert \boldsymbol v\vert_2^2+\mathcal O(n^{-\frac{\theta\eta} 4})\right)
e^{-\frac{\sigma^2_\xi \vert \boldsymbol v\vert^2}4}\, d\boldsymbol v\\
\nonumber&\quad 
 \sum_{\mathcal J_0\subset \mathcal J'}
\prod_{j\in\mathcal J'\setminus\mathcal J_0}\left(n_j^{-\frac 12-\eta}+n_{j+1}^{-\frac 12-\eta}\right)
\mathbb{E}\left[  
|\det D_{\boldsymbol k}|^{-\frac 12}{\bf 1}_{\Omega_{\boldsymbol k}\cap
\bigcap_{j\in\mathcal J_0}\mathcal B_j}\right] \\
&= \mathcal O\left(
n^{-\frac {\kappa\theta\eta} 4}
\left(\prod_{j=1}^mn_j^{-\frac {3}4}\right)
\mathfrak E_{\boldsymbol k}(\mathcal J')
\right)\, ,\label{MajoEkl}
\end{align}
with
\begin{align*}
\mathfrak E_{\boldsymbol k}(\mathcal J')&=
\sum_{\mathcal J''\subset \mathcal J'}
\prod_{j\in\mathcal J'\setminus\mathcal J''}\left(n_j^{-\frac 12-\eta}+n_{j+1}^{-\frac 12-\eta}\right)
\sum_{\mathcal J_0\subset \mathcal J''\setminus\{\min\mathcal J''\},\ \mathcal J_0\ge \#\mathcal J''/2}n^{J\gamma+\frac{\theta'}2}\left(\prod_{j\in\mathcal J_0}(k_j-k_j^-)^{-\frac 12}\right)\\
&= \mathcal O\left(\sum_{\mathcal J''\subset \mathcal J'\cup(\mathcal J'+1)\, :\, \#\mathcal J''\ge \#\mathcal J'/2}
\left(\prod_{j\in\mathcal J''}n_j^{-\frac 12+\eta}\right)\right)\, .
\end{align*}
where $k^-_{j}=\max\{k_s\le k_j,\ s\in\mathcal J''\}$.
Combining this last estimate with \eqref{1:Neglec} and Lemmas~\ref{LEMsum0}
and~\ref{LEMsum0bis},
\begin{align}
\nonumber\sum_{k'_{j}=0,...,d-1,\ \forall j\in\mathcal J'}B_{\boldsymbol k+\boldsymbol k',\boldsymbol \ell,I^{(3)}_{\boldsymbol k},\Omega_{\boldsymbol k}}&=\frac {d^m}{(2\pi)^M} 
\sum_{(a_j)_{j\not\in\mathcal J'},(b_{j,s})_{j,s}}\mathbf 1_{\{\forall i\not\in\mathcal J',\, a_i\in k_i\alpha+d\mathbb Z\}}
\int_{[-\pi,\pi]^{M-m}
}
\mathbb E\left[I_1(\boldsymbol a)\, I_2(\boldsymbol a,\boldsymbol b)\mathbf 1_{\Omega_{\boldsymbol k}}\right]
 \, d
\boldsymbol\theta'
\\
&\quad \quad\quad+ \mathcal O\left(
 n^{
-\frac {\kappa\theta\eta} 4}\prod_{j=1}^mn_j^{-\frac {3}4}
\mathfrak E_{\boldsymbol k}(\mathcal J')
\right)\, ,
\label{lastestimate}
\end{align}
with
\begin{align}
\nonumber I_1(\boldsymbol a)&:=
\int_{\forall j,\, |\theta_j| \le n_j^{-\frac 12-\eta}}
\left(\prod_{j\not\in\mathcal J'}
 e^{-i\sum_{j=}^m(a_{j}-a_{j-1})\theta_j}
\right)
F(\boldsymbol\theta,\boldsymbol\theta')
 e^{-\frac{\sigma^2_\xi}2\sum_{y\in\mathbb Z\setminus \mathcal S'}(\sum_{j=1}^m\theta_j N'_{j} (y))^2}\, d\boldsymbol \theta\\
\nonumber&=\mathcal O\left(
\int_{\forall j,\, |\theta_j| \le n_j^{-\frac 12-\eta}}F(\boldsymbol\theta,\boldsymbol\theta')
 e^{-\frac{\sigma^2_\xi}2\left(\sum_{y\in\mathbb Z}(\sum_{j=1}^m\theta_j N'_{j} (y))^2-Mdn_j^{-\eta}\right)}\, d\boldsymbol \theta\right)\\
&=\mathcal O\left(\det D_{\boldsymbol k}^{-\frac 12}\sup_{\boldsymbol\theta\in V_{\boldsymbol k}}F(\boldsymbol\theta,\boldsymbol\theta')
\int_{\mathbb R^m
}
 e^{-\frac{\sigma^2_\xi|\boldsymbol v|_2^2}2}
\, d\boldsymbol v\right)\, ,\label{majoI1}
\end{align}
with the change of variable $\boldsymbol v=D_{\boldsymbol k}^{\frac 12}\boldsymbol\theta$
and
\begin{align}
\nonumber I_2(\boldsymbol a,\boldsymbol b)&:=
 \left(\prod_{j\not\in \mathcal J'}\left(f(a_j)\prod_{s=1}^{s_j}f(b_{j,s})e^{-i\sum_{j,s}(b_{j,s}-a_j)\theta'_{j,s}}\right)\right)
\prod_{y\in\mathbb Z\setminus\mathcal S'}\varphi_\xi\left(\sum_{j,s}(\theta'_{j,s}N'_{j,s}(y))\right)
\\
&=\mathcal O\left(\prod_{j\not\in \mathcal J'}\left(f(a_j)\prod_{s=1}^{s_j}f(b_{j,s})\right)\right)\, .\label{majoI2}
\end{align}
Since $\sum_{a\in\mathbb Z}|f(a)|<\infty$, it follows from\eqref{majoespDet}, \eqref{lastestimate},
  \eqref{majoI1} and \eqref{majoI2} that
\begin{equation*}
\sum_{k'_{j}=0,...,d-1,\ \forall j\in\mathcal J'}B_{\boldsymbol k+\boldsymbol k',\boldsymbol \ell,I^{(3)}_{\boldsymbol k},\Omega_{\boldsymbol k}}=\mathcal O\left(
\mathfrak E_{\boldsymbol k}(\mathcal J')
\left(\prod_{j=1}^mn_j^{-\frac 34}\right)\right)\, .
\end{equation*}
This ends the proof of the first point of Lemma \ref{LEM5}.\\

Assume now that $s_j=1$ for all $j=1,...,m$ (in particular $\mathcal J=\emptyset$). Then
\begin{align*}
 I_1(\boldsymbol a)&=
\int_{\forall j,\, |\theta_j| \le n_j^{-\frac 12-\eta}}
 e^{-i\sum_{j=1}^m(a_{j}-a_{j-1})\theta_j}
 e^{-\frac{\sigma^2_\xi}2\sum_{y\in\mathbb Z\setminus \mathcal S'}(\sum_{j=1}^m\theta_j N'_{j} (y))^2}
\, d\boldsymbol\theta\, \\
&=\left(\prod_{j=1}^mn_j^{-\frac 34}\right)
\int_{\forall j,|\theta''_j| \le n_j^{\frac 14-\eta}} e^{-i\sum_{j=1}^m n_j^{-\frac 34}(a_{j}-a_{j-1})\theta''_j}
 e^{-\frac{\sigma^2_\xi}2\sum_{y\in\mathbb Z}(\sum_{j=1}^m\theta''_j n_j^{-\frac 34}N'_{j} (y))^2}
\, d\boldsymbol\theta''\, \\
&=
 \left(\prod_{j=1}^mn_j^{-\frac 34}\right)
\int_{\widetilde D_{\boldsymbol k}^{\frac 12}U_{\boldsymbol k}}(\det \widetilde D_{\boldsymbol k})^{-\frac 12} e^{-i\langle \widetilde D_{\boldsymbol k}^{-\frac 12}(n_j^{-\frac 34} (a_{j}-a_{j-1}))_j,\boldsymbol v\rangle}
e^{-\frac{\sigma^2_\xi|\boldsymbol v|_2^2}2}
\, d\boldsymbol v\, ,
\end{align*}
where $U_{\boldsymbol k}$ is the set of $\boldsymbol\theta''=(\theta''_1,...,\theta''_m)$ such that $|\theta''_j|\le n_j^{\frac 14-\eta}$ for all $j=1,...,m$ and with $\widetilde D_{\boldsymbol k}=\left((n_in_j)^{-\frac 34}\sum_{y\in\mathbb Z} N'_i(y)N'_j(y)\right)_{i,j}$. Moreover
\begin{align*}
I_2(\boldsymbol a,\boldsymbol b)
&=(2\pi)^{\sum_{j=1}^ms_j}
\left(\prod_{j=1}^m\left(f(a_j)
f(b_{j,1})
\right)\right)\mathbb P\left(\left. \forall j,\  \sum_{y\in\mathbb Z\setminus\mathcal S'} N'_{j,1}(y)\xi_y=b_{j,1}-a_j    \right|(N'_{j,1})_j\right)\\
&=(2\pi)^{M-m}
\left(\prod_{j=1}^mf(a_j)\right)\mathbb E\left[\left.  f\left(a_j+\sum_{y\in\mathbb Z} N'_{j,1}(y)\xi_y\right) \mathbf 1_{\{a_j+\sum_{y\in\mathbb Z} N'_{j,1}(y)\xi_y=b_{j,1}\}}  \right|(N'_{j,1})_j\right]\, .
\end{align*}
Thus, it follows that, uniformly in $\boldsymbol k$ and on $\Omega_{\boldsymbol k}$,
\begin{align*}
\frac {d^m}{(2\pi)^{M}}\sum_{b_{1,1},...,b_{m,1}\in\mathbb Z}&I_1(\boldsymbol a)\, I_2(\boldsymbol a,\boldsymbol b)
=\left(\frac d{2\pi}\right)^{m}
\left(\prod_{j=1}^mf(a_j)\right)\\
&\quad \quad (\det D_{\boldsymbol k})^{-\frac 12}
\left(\int_{\mathbb R^m
} e^{-i\langle \widetilde D_{\boldsymbol k}^{-\frac 12} (n_j^{-\frac 34} (a_{j}-a_{j-1}))_j,\boldsymbol v\rangle}
 e^{-\frac{\sigma^2_\xi|\boldsymbol v|_2^2}2}
\, d\boldsymbol v+\mathcal O(n^{-p})\right)\\
&\quad\quad\mathbb E\left[\left.  f\left(a_j+\sum_{y\in\mathbb Z} N'_{j,1}(y)\xi_y\right) 
 \right|(N'_{j,1})_j\right]\, 
\end{align*}
for all $p>0$,
as seen at the end of the proof of Lemma \ref{LEM4} (applied with $\widetilde D_{\boldsymbol k}$)
and so
\begin{align*}
&\frac {d^m}{(2\pi)^{M}}\sum_{b_{1,1},...,b_{m,1}\in\mathbb Z}
I_1(\boldsymbol a)\, I_2(\boldsymbol a,\boldsymbol b)
=\left(\frac d{2\pi}\right)^{m}
\left(\prod_{j=1}^mf(a_j)\right) (\det D_{\boldsymbol k})^{-\frac 12}\\
&\quad\times\left(\int_{\mathbb R^m
} \left(1+\mathcal O\left(
\left(\langle \widetilde D_{\boldsymbol k}^{-\frac 12} (n_j^{-\frac 34} (a_{j}-a_{j-1}))_j,\boldsymbol v\rangle\right)^{2}\right)\right)
 e^{-\frac{\sigma^2_\xi|\boldsymbol v|_2^2}2}
\, d\boldsymbol v+\mathcal O(n^{-p})\right)\\
&\quad \times\mathbb E\left[\left.  f\left(a_j+\sum_{y\in\mathbb Z} N'_{j,s}(y)\xi_y\right) 
\right|(N'_{j,1})\right]\, ,
\end{align*}
for all $p$.
Due to \eqref{lastestimate}, we obtain that
\begin{align}
B_{\boldsymbol k,\boldsymbol \ell,I^{(3)}_{\boldsymbol k},\Omega_{\boldsymbol k}}&=\left(\frac d{2\pi } \right)^m\sum_{a_1,...,a_m\in\mathbb Z}\mathbf 1_{\{\forall i,\, a_i=k_i\alpha+d\mathbb Z\}}\\
&\quad\times\mathbb E\left[ (\det D_{\boldsymbol k})^{-\frac 12}\mathbf 1_{\Omega_{\boldsymbol k}} \prod_{j=1}^mf(a_j) f\left(a_j+\sum_{y\in\mathbb Z} N'_{j,s}(y)\xi_y\right) \right]\left(\frac{\sqrt{2\pi}}{\sigma_\xi}\right)^m\\
&\quad+\mathcal O\left(\prod_{j=1}^mn_j^{-\frac 34}\right)n^{-\frac {\kappa\theta\eta} 4}+\mathbb E\left[(\det D_{\boldsymbol k})^{-\frac 12} (\min_j n_j)^{-\frac 32} \widetilde\lambda_{\boldsymbol k}^{-1}\mathbf 1_{\Omega_{\boldsymbol k}}\right]\, ,\label{ESTII}
\end{align}
where $\widetilde\lambda_{\boldsymbol k}$ is the smallest eigenvalue of $\widetilde D_{\boldsymbol k}$.
For the last term, we use \eqref{minolambda} (applied for $\widetilde D_{\boldsymbol k}$), which ensures that
on $\Omega_{\boldsymbol k}$,
\[
\widetilde \lambda_{\boldsymbol k}\ge \frac{\det \widetilde D_{\boldsymbol k}}{(mn^{3\gamma})^{m-1}}  
\]
and so
\begin{align}
\nonumber (\min_j n_j)^{-\frac 32} \mathbb E\left[(\det D_{\boldsymbol k})^{-\frac 12} \lambda_{\boldsymbol k}^{-1}\mathbf 1_{\Omega_{\boldsymbol k}}\right]&\le (mn^{3\gamma})^{m-1}\nonumber (\min_j n_j)^{-\frac 32} \left(\prod_{j=1}^m n_j^{-\frac 34}\right)
\mathbb E\left[(\det \widetilde D_{\boldsymbol k})^{-\frac 32} \mathbf 1_{\Omega_{\boldsymbol k}}\right]\\
&= \mathcal O \left( \left(\prod_{j=1}^m n_j^{-\frac 34}\right) n^{-\frac {3\theta}2+3(m-1)\gamma}\right )\, ,\label{majoDlambda}
\end{align}
where we used \cite[Lemma 21]{BFFN2} which ensures that
$\mathbb E\left[(\det \widetilde D_{\boldsymbol k})^{-\frac 32} \mathbf 1_{\Omega_{\boldsymbol k}}\right]=\mathcal O\left(1\right)$ uniformly in $\boldsymbol k$. This combined with \eqref{ESTII} implies that
\begin{align}
&
B_{\boldsymbol k,\boldsymbol \ell,I^{(3)}_{\boldsymbol k},\Omega_{\boldsymbol k}}=\mathcal O\left( \left(\prod_{j=1}^m n_j^{-\frac 34}\right) n^{-(M+1)L\theta}\right)
\\
&+\left(\frac d{\sqrt{2\pi}\sigma_\xi n^{\frac 34}}\right)^m\sum_{a_1,...,a_m\in\mathbb Z}\mathbf 1_{\{\forall i,\, a_i=k_i\alpha+d\mathbb Z\}}\mathbb E\left[ (\det D_{\boldsymbol k})^{-\frac 12}\mathbf 1_{\Omega_{\boldsymbol k}} \prod_{j=1}^mf(a_j) f\left(a_j+\sum_{y\in\mathbb Z} N'_{j,s}(y)\xi_y\right) \right]\, ,\label{ESTIII}
\end{align}
since $L<\min\left(\frac {3m}{4M},\frac{\kappa\eta}4\right)$ and since
$L(M+1)\theta<\frac {3\theta}2-3(m-1)\gamma$.

The last step of the proof of the lemma consists in studying the following quantity
\begin{align}\label{defG}
G_{\boldsymbol k}:=&\mathbb E\left[ (\det D_{\boldsymbol k})^{-\frac 12}\mathbf 1_{\Omega_{\boldsymbol k}} \prod_{j=1}^mf(a_j) f\left(a_j+\sum_{y\in\mathbb Z} N'_{j,s}(y)\xi_y\right) \right]\, .
\end{align}
Due to Lemma~\ref{lemD-D''},
\begin{align*}
G_{\boldsymbol k}=&\mathbb E\left[ (\det D''_{\boldsymbol k})^{-\frac 12}\mathbf 1_{\Omega'_{\boldsymbol k}} \prod_{j=1}^mf(a_j) f\left(a_j+\sum_{y\in\mathbb Z} N'_{j,s}(y)\xi_y\right) \right]+\mathcal O\left(n^{-\frac\theta 8-L\theta}\prod_{j=1}^m n_j^{-\frac 34}\right)\\
&= \mathbb E\left[ (\det D_{\boldsymbol k})^{-\frac 12}\mathbf 1_{\Omega_{\boldsymbol k}}\right] \prod_{j=1}^mf(a_j) \mathbb E\left[f\left(a_j+\sum_{y\in\mathbb Z} N'_{j,s}(y)\xi_y\right) \right]+\mathcal O\left(n^{-\frac\theta 8-L\theta}\prod_{j=1}^m n_j^{-\frac 34}\right)\, ,\nonumber
\end{align*}
where we used the fact that $D''_{\boldsymbol k}$ has the same distribution as $D_{\boldsymbol k}$ and 
is independent of $N'_{j,s}$.
This combined with  \eqref{ESTIII}, \eqref{majoDlambda}, \eqref{POmega'} and  \eqref{D-D''} ensures that
\begin{align*}
B_{\boldsymbol k,\boldsymbol \ell,I^{(3)}_{\boldsymbol k},\Omega_{\boldsymbol k}}&=\left(\frac d{\sqrt{2\pi}\sigma_\xi 
}\right)^{ m}\sum_{a_1,...,a_m\in\mathbb Z}\mathbf 1_{\{\forall i,\, a_i=k_i\alpha+d\mathbb Z\}}\mathbb E\left[ (\det D_{\boldsymbol k})^{-\frac 12}\mathbf 1_{\Omega_{\boldsymbol k}}\right]\\
&\ \ \ \ \prod_{j=1}^mf(a_j)\mathbb E\left[ f\left(a_j+Z_{\ell_j}\right) \right]+\mathcal O\left(n^{-L(M+1)\theta}\prod_{j=1}^m n_j^{-\frac 34}\right)
\end{align*}
Moreover \cite[Lemmas 21 and 23]{BFFN2} ensure that
\[
\mathbb E\left[ (\det D_{\boldsymbol k})^{-\frac 12}\mathbf 1_{\Omega_{\boldsymbol k}}\right]=\mathcal O\left(\prod_{j=1}^mn_j^{-\frac 34}\right)\, ,
\]
and that
\[
\mathbb E\left[ (\det D_{\boldsymbol k})^{-\frac 12}\mathbf 1_{\Omega_{\boldsymbol k}}\right]\sim n^{-\frac{3m}4}
\mathbb E\left[ \det \mathcal D_{t_1,...,t_m}^{-\frac 12}\right]
\]
as $k_j/n\rightarrow t_j$ and $n\rightarrow +\infty$.
This ends the proof of the lemma.
\end{proof}

\section{Moment convergence in Theorem~\ref{LLN}}\label{Appendcvgcemoment}
Let $f:\mathbb Z\rightarrow \mathbb R$ be such that $\sum_{a\in\mathbb Z}|f(a)|<\infty$.
In this appendix we prove that all the moments of $ n^{-\frac 14}\sum_{k=0}^{n-1}f(Z_k)$
converge to those of $\sum_{a\in\mathbb Z}f(a)\sigma _\xi^{-1}\mathcal L_{1}(0)$, as $n\rightarrow +\infty$.
\\

Due to Theorem~\ref{propboundmoment}, it is enough to prove the convergence
of every moment.
The key result is the following proposition.
\begin{prop}\label{LemmeLLN}
For all $a_1,...,a_k\in\mathbb Z$,
\begin{eqnarray*}
{\mathbb P} \pare{Z_{n_1}=a_1, \dots,Z_{n_k}=a_k} \sim
\mathbf 1_{\{\forall i,\ a_i\in n_i\alpha+d\mathbb Z\}}
\left(\frac{d} {\sqrt{2\pi}\sigma_\xi}\right)^k\, \mathbb E[\det \mathcal D_{T_1,...,T_k}^{-\frac 12}]
\ n^{-3k/4}\, ,
\end{eqnarray*}  
as $n\rightarrow +\infty$ and $n_i/n\rightarrow T_i$,
where $\mathcal D_{t_1,...,t_k}=(\int_{\mathbb R} L_{t_i}(x)L_{t_j}(x)\, dx)_{i,j=1,...,k}$ where $L$ is the local time of the brownian motion $B$,
 limit of $(S_{\lfloor nt\rfloor}/\sqrt{n})_t$ as $n$ goes to infinity.\\
Moreover, for every $k\ge 1$ and every $\vartheta\in(0,1)$, there exists $C=C(k,\theta)>0$, such that 
$$\PP\left[Z_{n_1}=a_1,\dots ,
   Z_{n_1+\dots+n_k}=a_k\right]\le C\prod_{j=1}^k n_j^{-3/4},$$
for all $n\ge 1$, all $a_1,...,a_k\in\mathbb Z$ and all $n_1,\dots,n_k\in [n^{\vartheta},n]$.
\end{prop}
\begin{proof}
The lemma has been proved for $a_i\equiv 0$ in \cite[Theorem 5]{BFFN2}. The proof in the general case is the straighforward adaptation of \cite[Section 5]{BFFN2}. For completness, we explain the required adaptations. 
The proof of the present result follows line by line the same proof with the adjonction of a term $e^{-i\sum_{j=1}^k(a_j-a_{j-1})\theta_j}$ (with convention $a_0=0$) in the integrals appearing in \cite[Lemma 15]{BFFN2} (see Lemma~\ref{LEM1} with $M=m=k$ and $s_j\equiv 0$). Lemma 16 (definition of the good set) and Propositions 18 and 19 (estimates of the integral of the absolute values) of \cite{BFFN2} are unchanged. The only difference in the proof concern \cite[Proposition 17]{BFFN2} and more specifically \cite[Lemma 23]{BFFN2} for which the there is a multiplication by
$e^{-i\sum_{j=1}^k(a_j-a_{j-1})\theta_j}$ in the integral. The only difference
in the proof of \cite[Lemma 23]{BFFN2} is that the quantity $I_{n_1,...,n_k}$ considered therein ($n_i$ corresponding to $\lfloor nT_i\rfloor-\lfloor nT_{i-1}\rfloor$) is slightly modified with the multiplication in the integral by a quantity converging in probability to 1 (with the notations of the proof of \cite[Lemma 23]{BFFN2}. Indeed, considering the real part of the integral, this quantity is $\cos(\sum_{j=0}^k(a_j-a_{j-1})(A_{n_1,...,n_k}^{-\frac 12}r)_j)$ (with the notations of \cite[Lemma 23]{BFFN2}) which is equal to 1 up to an error in $\mathcal O\left(\min\left(1, \mu_{n_1,...,n_k}^{-1}|r|^2\right)\right)$ where $\mu_{n_1,...,n_k}$
is the smallest eigenvalue of $A_{n_1,...,n_k}$, which is proved to converges to 0 
in \cite[Lemma 23]{BFFN2}, and so the asymptotic behaviour of $I_{n_1,...,n_k}$
is the same as when $a_j\equiv 0$.
\end{proof}
\begin{proof}[Proof of the convergence of moments in Theorem~\ref{LLN}]
Take $\vartheta<\frac 14$.
Note that the last point of the lemma ensures that
$$\PP\left[Z_{n_1}=a_1,\dots 
   Z_{n_1+\dots+n_k}=a_k\right]\le C\, \left(\prod_{i:n_i>n^\vartheta}n_i\right)^{-3/4}.$$
Let $\alpha_0$ be such that $\alpha\alpha_0\in 1+d\mathbb Z$. Then
$a_i=q_i\alpha+d\mathbb Z$ is equivalent to $q_i\in a_i\alpha_0+d\mathbb Z$. 
Thus
\begin{align*}
&\mathbb E
\left[\left(\sum_{q=0}^{n-1}f(Z_q)\right)^k\right]
=\sum_{q_1,...,q_k=0}^{n-1}\mathbb E
\left[f(Z_{q_1})...f(Z_{q_k})\right]=\sum_{a_1,...,a_k\in\mathbb Z}f(a_1)...f(a_k)\sum_{q_1,...,q_k=0}^{n-1}\mathbb P(Z_{q_1}=a_1,...,Z_{q_k}=a_k)\\
&=O(n^{\frac {k-1}4})+
\sum_{r_1,...,r_k=0}^{d-1}\sum_{a_1,...a_k
\in\mathbb Z}f(a_1)...f(a_k)\sum_{q_1,...,q_k=0}^{\lfloor\frac n{d}\rfloor-1}\mathbb P(Z_{r_1+q_1d}=a_1,...,Z_{r_k+q_kd}=a_k)\\
&=O(n^{\frac {k-1}4})+\sum_{a_1,...a_k
\in\mathbb Z}f(a_1)...f(a_k)\sum_{q_1,...,q_k=0}^{\lfloor\frac n{d}\rfloor-1}\mathbb P(Z_{\overline{a_1\alpha_0}+q_1d}=a_1,...,Z_{\overline{a_k\alpha_0}+q_kd}=a_k)\, ,
\end{align*}
with $\overline{x}$ the representant of $x+d\mathbb Z$ belonging to $\{0,...,d-1\}$.
It follows that
\begin{align*}
\mathbb E
\left[\left(\sum_{q=0}^{n-1}f(Z_q)\right)^k\right]
&= o(n^{\frac k 4})+\sum_{a_1,...,a_k 
\in \mathbb Z} f(a_1)...f(a_k) n^k H_k\\
&= o(n^{\frac k 4})+\sum_{a_1,...,a_k 
\in \mathbb Z} f(a_1)...f(a_k) n^k H'_k\, ,
\end{align*}
with
\begin{align*}
H_k&:=\int_{[0,1/d]^k}\mathbb P\left(Z_{\overline{a_1\alpha_0}+\lfloor t_1 n\rfloor d}=a_1,...,Z_{\overline{a_k\alpha_0}+\lfloor t_k n\rfloor d}=a_k\right)\, dt_1...dt_k\\
H'_k&=   \int_{[0,1/d]^k} n^{\frac {3k}4}\mathbb P\left(Z_{\overline{a_1\alpha_0}+\lfloor t_1 n\rfloor d}=a_1,...,Z_{\overline{a_k\alpha_0}+\lfloor t_k n\rfloor d}=a_k\right)\mathbf 1_{\min_{i,j}|\lfloor t_{i}n\rfloor-\lfloor t_{j}n\rfloor|> 2n^\vartheta}\, dt_1...dt_k\, .
\end{align*}
Due to the dominated convergence theorem, we conclude that
\begin{align*}
&\mathbb E
\left[\left(\sum_{q=0}^{n-1}f(Z_q)\right)^k\right]\\
&=  o(n^{\frac k 4})+n^{\frac k4}
\sum_{a_1,...,a_k 
\in \mathbb Z} f(a_1)...f(a_k) \int_{[0,1/d]^k}
\left(\frac{d} {\sqrt{2\pi}\sigma_\xi}\right)^k\, \mathbb E[\det \mathcal D_{t_1d,...,t_kd}^{-\frac 12}]
 dt_1...dt_k
\\
&=  o(n^{\frac k 4})+n^{\frac k4}
\left(\sum_{a\in \mathbb Z} f(a)\right)^k \int_{[0,1]^k}
\left( \sqrt{2\pi}\sigma_\xi\right)^{-k}\, \mathbb E[\det \mathcal D_{t_1d,...,t_kd}^{-\frac 12}]
 dt_1...dt_k
\\
&=  o(n^{\frac k 4})+n^{\frac k4}
\left(\sum_{a\in \mathbb Z} f(a)\sigma_\xi^{-1}\right)^k\mathbb E[(\mathcal L_1(0))^k] \, ,
\end{align*}
due to \cite[Theorem 3]{BFFN2}.
\end{proof}

\end{appendix}

\end{document}